\documentclass[12pt, a4paper]{article}

\usepackage{amsmath, amsthm, amssymb, bm, graphicx, mathrsfs,array,IEEEtrantools}

\usepackage[T1]{fontenc}
\usepackage[utf8]{inputenc}
\usepackage[margin=20mm,headheight=14pt,footskip=12mm]{geometry}

\usepackage{newtxtext,newtxmath}
\newtheorem{theorem}{Theorem}[section]
\newtheorem{definition}[theorem]{Definition}
\newtheorem{lemma}[theorem]{Lemma}
\newtheorem{corollary}[theorem]{Corollary}
\newtheorem{prop}[theorem]{Proposition}
\newtheorem{ass}[theorem]{Assumption}

\newtheorem{remark}[theorem]{Remark}

\newcommand{\mf}{\mathfrak}
\newcommand{\pa}{\partial}
\newcommand{\wh}{\widehat}
\newcommand{\ol}{\overline}
\newcommand{\vph}{\varphi}
\newcommand{\vep}{\varepsilon}
\newcommand{\om}{\omega}
\newcommand{\lam}{\lambda}
\newcommand{\md}{\mathrm{d}}

\newcommand{\R}{{\mathbb R}}
\newcommand{\N}{{\mathbb N}}

\newcommand{\Z}{{\mathbb Z}}
\newcommand{\ml}{\mathcal}

\newcommand{\x}{\overline{x}}
\newcommand{\np}{\overline{\partial}}

\newcommand{\lvp}{\overline{\varepsilon}}

\usepackage{hyperref}
\makeatletter
\def\@maketitle{%
  \newpage
  \null
  \vskip 2em%
  \begin{center}%
  \let \footnote \thanks
    {\Large \bfseries \@title \par}%
    \vskip 1.5em%
    {\large
      \lineskip .5em%
      \begin{tabular}[t]{c}%
        \@author
      \end{tabular}\par}%
    \vskip 1em%
    {\large \@date}%
  \end{center}%
  \par
  \vskip 1.5em}
\def\ps@myheadings{%
  \let\@oddfoot\@empty\let\@evenfoot\@empty
  \def\@oddhead{%
    \ifodd\value{page}%
      \hfil{\small\MakeUppercase{\rightmark}}\hfil\thepage
    \else
      \thepage\hfil{\small\MakeUppercase{\leftmark}}\hfil
    \fi
  }%
  \let\@mkboth\@gobbletwo
  \let\sectionmark\@gobble
  \let\subsectionmark\@gobble
}
\makeatother

\begin{document}
	\title{The Inviscid Limit of Viscous Burgers through Shock Formation and Development}
	\author{Chao Wang, Yueyi Yang and Zhifei Zhang}

	\author{ 
		 Chao Wang \thanks{wangchao@math.pku.edu.cn} 
		\quad Yueyi Yang
 \thanks{yueyiyang25@stu.pku.edu.cn} 
		\quad Zhifei Zhang  \thanks{zfzhang@pku.edu.cn}\\[10pt]
		\footnotesize{School of Mathematical Sciences, Peking University, Beijing 100871, P. R. China}\\
		[5pt]}

	\maketitle
	\pagestyle{myheadings}
	\markboth{Chao Wang, Yueyi Yang and Zhifei Zhang}{Shock Development of Burgers}

	\begin{abstract}
 
We study the inviscid limit of the one‑dimensional viscous Burgers equation from smooth initial data through nondegenerate shock formation and subsequent shock development. Chaturvedi and Graham \cite{CG} obtained sharp estimates of order $\nu^{1/4}$ up to the shock‑formation time. We extend their analysis beyond this time, comparing the viscous solution with the inviscid entropy solution evolving from the same initial data. With the shock‑formation time normalized to $t=0$, we prove a pointwise error bound on a fixed interval $0\le t\le\bar{\varepsilon}$, where $\bar{\varepsilon}>0$ is independent of the viscosity $\nu$. The bound consists of a term of order $\nu^{1/4}$ and an exponentially localized contribution originating from the viscous shock layer. The proof is based on matching an outer expansion, an inner expansion near the formation point, and a viscous shock‑profile expansion. Combined with the results of Chaturvedi and Graham, this yields a quantitative description of the inviscid limit from the initial time through the development of the young shock.

 	\end{abstract}

	\tableofcontents

	\numberwithin{equation}{section}

		\section{Introduction}
In this paper we study the vanishing viscosity limit of the
one-dimensional Burgers equation from smooth initial data through the
formation and development of the first shock. We consider
	\begin{equation}\label{1.1}
		\left\{\begin{aligned}
			&\pa_tu^{(\nu)}+u^{(\nu)}\pa_xu^{(\nu)}=\nu u^{(\nu)}_{xx}\\
			&u^{(\nu)}(t_{\mathrm{in}},x)=u_0(x),
		\end{aligned}
		\right.
	\end{equation}
where $\nu>0$, $t_{in}<0$ and $u_0$ is smooth. Let $u^{(0)}$ denote the entropy solution of the corresponding inviscid equation with the same initial data. We assume that its first singularity is a nondegenerate shock formation, normalized to occur at $(t,x)=(0,0)$. At this point the inviscid solution is continuous but has an unbounded spatial derivative; for small positive times, its entropy continuation contains a shock whose strength grows from zero. Our aim is to quantify how the smooth viscous solution follows this transition and approximates the developing shock on a time interval independent of $\nu$.

	\subsection{Background}

	The vanishing viscosity method provides a fundamental connection between parabolic regularization and entropy solutions of conservation laws. For Burgers equations, the Cole--Hopf transformation reduces the viscous equation to the heat equation and gives an explicit route to its inviscid limit  \cite{Cole1951,Hopf1950}. The works of Oleinik and Kru\v{z}kov \cite{Kruzhkov1970,Oleinik1957} established the broader scalar theory. Quantitative results include the $\mathcal{O}(\nu\log(1/\nu))$ $L^1$ error estimate of Tang and Teng \cite{TangTeng} for piecewise smooth solutions of strictly convex scalar conservation laws, allowing shock formation and wave interactions within their hypotheses. Wang \cite{Wang1998} showed that this order is optimal for viscous Burgers and identified the occurrence of the logarithmic error at the time of spontaneous shock formation. These results already capture the influence of shock formation on integral error estimates. A pointwise analysis must additionally resolve where the error is concentrated and how its spatial scale changes as the shock develops.

For established shocks, matched asymptotic expansions provide a detailed description of the viscous approximation. Goodman and Xin \cite{Goodman} combined regular outer expansions with viscous profiles for piecewise smooth solutions containing isolated weak Lax shocks under suitable stability assumptions. Yu \cite{Yu1999} treated discontinuous initial data and the associated initial layers and nonlinear wave phenomena. Rousset \cite{Rousset2003} extended the approximation theory to strong shocks possessing linearly stable viscous profiles. In several space dimensions, Gu\`es, M\'etivier, Williams, and Zumbrun \cite{GMWZ2006} justified the Navier--Stokes regularization of Euler shock fronts under appropriate profile and spectral stability hypotheses. These constructions describe the regularization of an existing shock. Near its birth, a further transition must be resolved: the inviscid gradient becomes unbounded while the jump tends to zero, so the formation region must be connected to the emerging viscous layer.

Relative entropy methods with a time-dependent shift provide stability and contraction estimates for inviscid and viscous shocks \cite{Leger2011,LegerVasseur2011,ChoiVasseur2015,KangVasseur2017}. For systems, Bianchini and Bressan  \cite{BianchiniBressan2005} established the small-total-variation limit for fully parabolic artificial viscosity, while weighted relative entropy yields stability of small barotropic Navier--Stokes shocks and their Euler limits \cite{KangVasseur2021Contraction,KangVasseur2021Uniqueness}. Chen, Kang, and Vasseur \cite{ChenKangVasseur2024} obtained small-$BV$ barotropic Euler solutions from Navier--Stokes under suitable physical viscosity and initial-data assumptions. This theory includes shock formation and interactions. The present problem requires further information: a pointwise estimate through nondegenerate shock birth from common smooth initial data, resolving the time-dependent width of the emerging viscous layer.

The geometry of inviscid shock formation was developed by Alinhac \cite{Alinhac1999Acta,Alinhac1999Annals} for two-dimensional quasilinear wave equations, by Christodoulou \cite{Christodoulou2007} for three-dimensional fluids, and by Speck \cite{Speck2016} for classes of three-dimensional quasilinear wave equations. For Euler flows near simple plane waves, Luk and Speck \cite{LukSpeck2018,LukSpeck2024} established stability of shock formation, allowing vorticity in two dimensions and both vorticity and variable entropy in three dimensions. Buckmaster, Shkoller, and Vicol \cite{BSV2022} constructed shock formation with a $C^{1/3}$ cusp for two-dimensional isentropic Euler in an azimuthal setting.

Beyond the first singular time, Abbrescia and Speck \cite{AbbresciaSpeck2022,AbbresciaSpeck2026Horizon} described local portions of the boundary of maximal globally hyperbolic developments (MGHDs) for three-dimensional Euler flows near plane symmetry, with vorticity and entropy. Abbrescia, Speck, and Yu  \cite{AbbresciaSpeckYu2026} established existence, uniqueness, and stability of a full MGHD for small isentropic Euler data under spherical symmetry and specific asymptotic conditions. In a $1+1$-dimensional model, Abbrescia, Blue, Sbierski, and Speck \cite{AbbresciaBlueSbierskiSpeck2025} constructed both an MGHD and an entropy solution with a developing shock, clarifying the distinction between classical development and weak continuation.

For the three-dimensional compressible Euler equations with small spherically symmetric initial perturbations with constant initial entropy and a nondegeneracy condition, Yin \cite{Yin2004} constructed a weak entropy solution emerging from the first gradient blowup, with a shock whose strength starts from zero and subsequently increases. Christodoulou \cite{Christodoulou2019}  solved a restricted shock development problem formulated as a free-boundary problem with singular initial data. Buckmaster, Drivas, Shkoller, and Vicol  \cite{BDSSV2022} continued suitably prepared smooth flows beyond the first singularity for the full Euler equations with azimuthal symmetry, constructing an entropy-producing shock together with weaker contact and rarefaction singularities. For one-dimensional scalar conservation laws, Yin and Zhu  \cite{YinZhu2022} studied shock formation and the optimal regularity of the resulting shock curves for finitely degenerate and certain infinitely degenerate initial profiles. These inviscid constructions describe how smooth flows become singular and how admissible discontinuities subsequently emerge. Understanding their approximation by viscous flows is a further problem, for which Burgers provides a setting where the formation and development scales can be analyzed together.

Chaturvedi and Graham \cite{CG} resolved the approximation problem up to the first singular time. They constructed an outer expansion and a time-dependent inner expansion adapted to nondegenerate shock formation, obtaining approximations to arbitrary order and the sharp $L^\infty$ convergence rate $\nu^{1/4}$. They also proposed a three-expansion description of shock development, combining the outer and formation-region expansions with a viscous shock-profile expansion. Their large positive-time asymptotics for the leading inner profile identify its connection with the developing shock profile. Extending this description to the original Cauchy problem requires higher-order matching and error estimates that remain valid as the shock emerges.

Related progress has been made for systems and for local rescaled limits.  Bressan, Caravenna, and Shen \cite{BressanCaravennaShen2026} studied local asymptotic patterns near shocks, shock interactions, and shock formation for scalar conservation laws with uniformly convex flux. At shock formation, they proved convergence, after suitable rescaling and recentering, to an eternal viscous Burgers profile on bounded sets of the rescaled variables. Recently,  Anderson, Chaturvedi, and Graham \cite{Anderson,ACG2} analyzed nondegenerate shock formation in one-dimensional systems and obtained sharp vanishing viscosity estimates up to the first singular time, including for compressible Navier--Stokes equations with degenerate physical viscosity.
Extending such pointwise estimates through shock development is explicitly posed as an open problem in \cite[Appendix~B]{AbbresciaBlueSbierskiSpeck2025}: 
{\it  "This leaves the following interesting open problems: can you extend \cite{Anderson,ACG2}
to the shock development problem?"} Our paper addresses the scalar Burgers case and gives a positive answer to the open problem.

The distinction between bounded rescaled times and a fixed physical time is essential here. The formation scale is $t\sim\nu^{1/2}$, so reaching a fixed time $t=\bar\varepsilon>0$ requires control up to rescaled times $T=\bar\varepsilon\nu^{-1/2}\to\infty$. In this paper we carry out the three-expansion construction over this growing time range and prove a pointwise error bound for the viscous and inviscid solutions evolving from the same prescribed smooth initial data. The estimate connects the formation scale with the developing shock layer and remains valid throughout $0\leq t\leq\bar\varepsilon$.

	\subsection{Main results}
Before  stating  the main results, we first recall the structure of the inviscid
	solution $u^{(0)}$; full details are given in Section~3. So long as
	$u^{(0)}$ remains smooth, it is described by the method of
	characteristics, along which the slope $\pa_xu^{(0)}$ solves a Riccati
	equation; consequently, if $\pa_xu_0$ attains a negative minimum, then
	$u^{(0)}$ forms a shock in finite time. Throughout the paper we impose
	the following standing assumption on the initial data, as in \cite{CG}.
	\begin{ass}\label{ass:init}
		The initial data $u_0(x)$ satisfies the following:
		\begin{enumerate}
			\item[(H1)] $u_0\in C^\infty(\R)$;
			\item[(H2)] There exist $L>0$ and $c\in\R$ such that $u_0(x)\equiv c$ for $|x|\geq L$;
			\item[(H3)] $\min_{x\in\R}\pa_xu_0(x)<0$, and the minimum is attained at a unique point $x_0$;
			\item[(H4)] $\pa_x^3u_0(x_0)>0$ (nondegeneracy condition);
			\item[(H5)] Set 
			\[
			t_{\mathrm{in}}=\frac{1}{\min \pa_xu_0}<0.
			\]
  By a spatial translation and a Galilean transformation, we  assume that $x_0=0$ and $u_0(0)=0$, so that a shock forms at $(t,x)=(0,0)$.
		\end{enumerate}
	\end{ass}
	Under Assumption~\ref{ass:init}, $u^{(0)}$ is smooth on $[t_{in}, 0)$ and, at the
	formation time $t=0$, develops a $C_x^{1/3}$ cusp at $x=0$. For $t>0$, a
	single shock curve $x=\mf{s}(t)$ with $\mf{s}(0)=0$ separates the two
	smooth branches of $u^{(0)}$, and the pair $(u^{(0)},\mf{s})$ satisfies
	the Rankine--Hugoniot condition and the Lax entropy condition:
	\[
  \dot{\mf{s}}(t)
   =\frac{u^{(0)}(t,\mf{s}(t)-)+u^{(0)}(t,\mf{s}(t)+)}{2},
  \qquad
u^{(0)}(t,\mf{s}(t)-)>\dot{\mf{s}}(t)>
u^{(0)}(t,\mf{s}(t)+).
\]
	The vanishing viscosity limit thus splits into
	two stages: the pre-shock stage $t\in[t_{\mathrm{in}},0]$, on which
	$u^{(\nu)}$ converges to the continuous profile $u^{(0)}$
	(Theorem~\ref{thm} (i) below), and the shock-development stage
	$t\in[0,\ol\vep]$, on which $u^{(\nu)}$ must be compared with the
	discontinuous $u^{(0)}$ across a viscous shock layer of width
	$\sim\nu/\sqrt{t}$ about $x=\mf{s}(t)$ (Theorem~\ref{thm} (ii) below); the
	exponential term in Theorem~\ref{thm} (ii) reflects precisely the width of
	this layer.

\medskip

	Now, we are in a position to state the main results of this paper.
	\begin{theorem}\label{thm}
Suppose Assumption~\ref{ass:init} holds. There exist
$\ol\vep >0$, $\nu_0>0$, $C>0$, and $\delta>0$, independent of
$\nu$, such that the solutions with the common initial data $u_0$
satisfy the following estimates for $0<\nu\leq\nu_0$:
\begin{enumerate} 
\item[(i)] For $(t,x)\in[t_{\mathrm{in}},0]\times\R$,
			\begin{equation}\label{eq:intro-pre}
				|u^{(\nu)}(t,x)-u^{(0)}(t,x)|\lesssim\nu^{\frac14}.
			\end{equation}		
\item[(ii)] For $t\in[0,\ol{\vep}]$ and $x\in\R$,
			\begin{equation}\label{eq:intro-post} 
			\begin{aligned}
				|u^{(\nu)}(t,x)-u^{(0)}(t,x)|\lesssim \nu^{\frac 14} +\sqrt{t}e^{-\delta\frac{|x-\mf{s}(t)|\sqrt{t}}{\nu}}.
			\end{aligned}
			\end{equation}
\end{enumerate}			
			
	\end{theorem}
	
	\medskip

	\begin{remark}
		Part (i) is due to \cite{CG}: on the pre-shock interval
		$[t_{\mathrm{in}},0]$, Chaturvedi and Graham established the sharp two-sided
		estimate
		\begin{equation*}
			C^{-1}\nu^{\frac14}\le\left\|u^{(\nu)}-u^{(0)}\right\|_{L^\infty([t_{\mathrm{in}},0]\times\R)}\le C\nu^{\frac14},
		\end{equation*}
which is proved in \cite{CG}. The contribution of the paper is the second assertion (ii).
	\end{remark}
	\begin{remark}
		Part (ii) uses the smallness of $\lvp$ only to ensure that the inviscid
		solution $u^{(0)}$ has the single-shock structure $x=\mf{s}(t)$ on
		$[0,\lvp]$. For $t>\lvp$, $u^{(0)}$ may generate additional shocks or
		rarefaction waves, and the shock $x=\mf{s}(t)$ may interact with these
		waves, producing new shocks or leading to the annihilation of the
		original shock. The smallness of $\lvp$ is used solely to fix this
		structural property of the inviscid solution.
	\end{remark}

	\begin{remark}
Estimate~\eqref{eq:intro-post} separates the error at the formation scale from the contribution of the developing shock layer. The first term retains the scale $\nu^{1/4}$, while the second has the amplitude $\sqrt{t}$ of the inviscid jump and decays exponentially on the spatial scale $\nu/\sqrt{t}$. For $t\lesssim\sqrt{\nu}$, the entire right-hand side is $\mathcal{O}(\nu^{1/4})$. When $t\sim\sqrt{\nu}$, the layer scale agrees with the formation scale $\nu^{3/4}$; for $t\gg\sqrt{\nu}$, it describes a progressively thinner viscous transition. At fixed positive times, the localized term allows an error comparable to the jump near the shock, as is necessary when comparing a smooth function with a discontinuous entropy solution. The estimate therefore gives a uniform description through formation without asserting uniform convergence across a shock of positive strength.
	\end{remark}
	
	\begin{remark}
		Although we mainly focus on the Burgers equation, we think our methods can be extended to a general scalar viscous conservation law $u_t+f(u)_x=\nu u_{xx}$ in a similar fashion, with the requirement that $0<f''(u)\leq M$ for some constant $M$. The extension to systems [esp. compressible Euler equations] remains a topic for future work.
	\end{remark}

\section{Proof Overview}
In this section we outline our main strategy.
\subsection{Inviscid structure}
We first recall the inviscid structure near shock formation. The inviscid solution $u=u^{(0)}$ satisfies
	\begin{equation*}
	\left\{\begin{aligned}
		&\pa_tu+u\pa_xu=0\\
		&u(t_{\mathrm{in}},x)=u_0(x).
	\end{aligned}
	\right
	.
\end{equation*}
Under Assumption~\ref{ass:init} (as in \cite{CG}), $u$ forms a single
nondegenerate shock at $(t,x)=(0,0)$. To capture the singularity, we
introduce the characteristic flow $\pa_t\eta(t,y)=u\circ\eta$ with
$\eta(t_{\mathrm{in}},y)=y$, which yields
	\begin{equation}\label{0.2}
	\begin{aligned}
		u(t,\eta(t,y))&=u_0(y),\\
		\eta(t,y)&=y+(t-t_{\mathrm{in}})u_0(y).\\
	\end{aligned}
\end{equation}
The Eulerian solution $u$ loses regularity at $(t,x)=(0,0)$, where $\pa_y\eta=0$. Moreover, Taylor expansion near $(0,0)$ gives
	\begin{equation}\label{0.3}
	\begin{aligned}
	x=\eta(t,y)=-aty+cy^3+\ml{O}(t^2+y^4),
	\end{aligned}
\end{equation}
with constants $a,c>0$. For later use, we introduce the gauge $\mf{y}(t,x)$ by solving
	\begin{equation}\label{0.4}
	\begin{aligned}
	x=-at\mf{y}+c\mf{y}^3.
	\end{aligned}
\end{equation}
For $t\leq 0$, \eqref{0.4} uniquely determines $\mf{y}$, and we have $\mf{y}\sim y$. Moreover, the scaling identity $\mf{y}(\lam^2 t,\lam^3 x)=\lam\mf{y}(t,x)$ ($\lam>0$) captures the $C^{\frac13}_x$ cusp at $t=0$.
\par
We are interested in the shock development stage, i.e.,\ $t\in[0,\ol{\vep}]$ for some small $\ol{\vep}>0$. Let $\mf{s}:[0,\ol{\vep}]\to\R$, with $\mf{s}(0)=0$, denote the shock curve. Then $u(t,x)$ is smooth on $[0,\ol{\vep}]\times\R\setminus\{(t,x):x=\mf{s}(t)\}$. At $t=0$, it is continuous with a $C_x^{1/3}$ cusp at $x=0$; for $0<t\leq\ol{\vep}$, it has a jump discontinuity across the shock $\{x=\mf{s}(t)\}$. Moreover, $(u,\mf{s})$ satisfies the Rankine--Hugoniot condition:
	\begin{equation*}
	\begin{aligned}
	\dot{\mf{s}}(t)=\frac{1}{2}[u(t,\mf{s}(t)^-)+u(t,\mf{s}(t)^+)].
	\end{aligned}
\end{equation*}
Under the Lax entropy condition $u(t,\mf{s}(t)^-)>\dot{\mf{s}}(t)>u(t,\mf{s}(t)^+)$ for $0<t\leq\ol{\vep}$, we solve for $(u,\mf{s})$ by a standard fixed point argument and obtain $\mf{s}(t)=\ml{O}(t^2)$ (see Proposition~\ref{prop sdb}).\par
With $(u,\mf{s})$ at hand, we describe the inviscid structure in shock development. Introduce the shifted variable $\x=x-\mf{s}(t)$, so that the shock lies at $\x=0$; we henceforth work in the $(t,\x)$ coordinates. Rewriting \eqref{0.3} yields 
	\begin{equation}\label{0.6}
	\begin{aligned}
		\x+\mf{s}(t)=-aty+cy^3+\ml{O}(t^2+y^4).
	\end{aligned}
\end{equation}
	Since $\mf{s}(t)=\ml{O}(t^2)$, \eqref{0.6} admits a unique most negative root $y_-(t)$ and a unique most positive root $y_+(t)$ at $\x=0$. Define the set of labels
		\begin{equation*}
		\begin{aligned}
			\Upsilon=(-\infty,y_-(t))\cup(y_+(t),+\infty).
		\end{aligned}
	\end{equation*}
Then $y(t,\x)$ is uniquely determined by \eqref{0.6} for $\x\not=0$ and $y\in\Upsilon$. Moreover, the identity $u(t,\eta(t,y))=u_0(y)$ in \eqref{0.2} persists on each branch for $t>0$, while the jump discontinuity across the shock originates from the discontinuity of $y(t,\x)$. Following \cite{CG}, we extend $\mf{y}$ to $t\geq 0$ by solving 
		\begin{equation}\label{0.8}
		\begin{aligned}
			\x=-at\mf{y}+c\mf{y}^3.
		\end{aligned}
	\end{equation}
Then \eqref{0.8} uniquely determines $\mf{y}(t,\x)$ for $\x\not=0$ and $\mf{y}\in(-\infty,-\sqrt{\frac{at}c})\cup(\sqrt{\frac{at}c},+\infty)$, with $\mf{y}\sim y$ near $(t,x)=(0,0)$. Moreover, the scaling identity $\mf{y}(\lam^2 t,\lam^3\x)=\lam\mf{y}(t,\x)$ ($\lam>0$) agrees with the one for $t\leq 0$.
\subsection{Approximate solution at shock formation}
	For $t\leq 0$, i.e.,\ before shock formation, the authors in \cite{CG} obtained an $L^{\infty}$ error estimate between the inviscid and viscous solutions.
For later use, we briefly recall their proof. The main effort is devoted to constructing an approximation $u_{K,L}^{\mathrm{app}}$ to $u^{(\nu)}$, with truncation orders $K,L\in\N_0$, having an explicit profile. To this end, they impose the outer expansion
	\begin{equation}\label{0.10}
	\begin{aligned}
	u^{(\nu)}\sim\sum\limits_{k\geq 0} \nu^k u_k^{\mathrm{out}}.
	\end{aligned}
\end{equation}
Substituting \eqref{0.10} into \eqref{1.1} and collecting terms of order $\nu^k$ yields
		\begin{equation}\label{0.11}
		\begin{aligned}
			\pa_tu_k^{\mathrm{out}}+\frac12\pa_x \sum\limits_{k_1+k_2=k}u_{k_1}^{\mathrm{out}}u_{k_2}^{\mathrm{out}}=\pa_x^2u_{k-1}^{\mathrm{out}}.
		\end{aligned}
	\end{equation}
Note that $u_0^{\mathrm{out}}$ agrees with the inviscid solution $u^{(0)}$. The forcing term in \eqref{0.11} is small away from the pre-shock $(0,0)$, but becomes
nonperturbative as $(t,x)\to (0,0)$. The breakdown region is $\{\mf{d}\lesssim \nu^{\frac14}\}$, where 
\begin{equation*}
	\begin{aligned}
		\mf{d}=(a\pa_x\mf{y})^{-\frac12}\sim (|t|+\mf{y}^2)^{\frac12}.
	\end{aligned}
\end{equation*}
In view of the anisotropic scaling of $\mf{y}$, we introduce the inner coordinates
\begin{equation*}
	\begin{aligned}
	T=\nu^{-\frac12}t,\quad X=\nu^{-\frac34}x,\quad U=\nu^{-\frac14}u,
	\end{aligned}
\end{equation*}
in which \eqref{1.1} becomes 
\begin{equation}\label{0.14}
	\begin{aligned}
	\pa_TU^{(\nu)}+U^{(\nu)}\pa_XU^{(\nu)}=\pa_X^2U^{(\nu)}.
	\end{aligned}
\end{equation}
The authors in \cite{CG} impose an inner expansion:
\begin{equation}\label{0.15}
	\begin{aligned}
	U^{(\nu)}\sim \sum\limits_{\ell\geq 0}\nu^{\frac \ell4}U_\ell^{\mathrm{in}}.
	\end{aligned}
\end{equation}
Substituting \eqref{0.15} into \eqref{0.14} and collecting terms of order $\nu^{\frac \ell4}$ yields
\begin{equation}\label{0.16}
	\begin{aligned}
	\pa_TU_\ell^{\mathrm{in}}+\frac12\pa_X\sum\limits_{\ell_1+\ell_2=\ell}U_{\ell_1}^{\mathrm{in}}U_{\ell_2}^{\mathrm{in}}-\pa_X^2U_\ell^{\mathrm{in}}=0.
	\end{aligned}
\end{equation}
To construct $u_{K,L}^{\mathrm{app}}$, we need the outer expansion \eqref{0.10} and the inner expansion \eqref{0.15} to agree in the matching zone $\{\mf{d}\sim\nu^{\frac14-\epsilon}\}$ for small $\epsilon>0$. To bridge the gap, they introduce the grid terms $\{u_{k,\ell}\}_{k,\ell\in\N_0}$. More precisely, let $u_\ell^{\mathrm{in}}$ denote the rescaling of $U_\ell^{\mathrm{in}}$ back to the $(t,x)$ coordinates. Then the inner expansion becomes
\begin{equation}\label{0.17}
	\begin{aligned}
		u^{(\nu)}\sim \sum\limits_{\ell\geq 0}u_\ell^{\mathrm{in}}.
	\end{aligned}
\end{equation}
In the matching zone, we expand each $u_\ell^{\mathrm{in}}$ in powers of $\nu$:
\begin{equation}\label{0.18}
	\begin{aligned}
		u_\ell^{\mathrm{in}}\sim \sum\limits_{k\geq 0}\nu^k u_{k,\ell}.
	\end{aligned}
\end{equation}
Combining \eqref{0.10}, \eqref{0.17} and \eqref{0.18}, we expect
\begin{equation}\label{0.19}
	\begin{aligned}
		u_k^{\mathrm{out}}\sim \sum\limits_{\ell\geq 0} u_{k,\ell}.
	\end{aligned}
\end{equation}
In \cite{CG}, they show that both \eqref{0.18} and \eqref{0.19} hold in the matching zone, so that the outer and inner expansions agree in the matching zone $\{\mf{d}\sim\nu^{\frac14-\epsilon}\}$. Introduce the truncated sums
\begin{equation*}
	\begin{aligned}
		u_{[K]}^{\mathrm{out}}=\sum\limits_{k= 0}^K \nu^ku_k^{\mathrm{out}},\quad u_{[L]}^{\mathrm{in}}=\sum\limits_{\ell= 0}^L u_\ell^{\mathrm{in}}.
	\end{aligned}
\end{equation*}
The approximate solution is then given by 
\begin{equation}\label{0.20}
	\begin{aligned}
		u_{K,L}^{\mathrm{app}}=(1-\xi) u_{[K]}^{\mathrm{out}}+\xi u_{[L]}^{\mathrm{in}},
	\end{aligned}
\end{equation}
where $\xi$ is a cutoff function whose gradient $D\xi$ is supported in the matching zone. The following estimate uses the inner--outer construction and energy argument of \cite[Section~5]{CG}, with the exponent written as in \cite[Theorem~8.7]{ACG2}; the estimates for our cutoff are detailed in Appendix~\ref{app:prop61}:
\begin{prop}\label{prop0.1}
	For $\vep<\frac16$ and $\beta\in (\frac1{2+\vep},\frac12)$, there exists a constant $C>0$ such that for all $K,L\in \N_0$ and $\nu\in (0,1]$:
	\begin{equation*}
		\begin{aligned}
			\left\|u^{(\nu)}-u_{K,L}^{\mathrm{app}}\right\|_{L^{\infty}_tH^3_x([t_{\mathrm{in}},0]\times \R)}\lesssim \nu^{\min\{(1-2\beta)K,L\beta/3\}-C}.
		\end{aligned}
	\end{equation*}
\end{prop}
For fixed truncation orders $K,L$ satisfying $\min\{(1-2\beta)K,L\beta/3\}>C$, Proposition~\ref{prop0.1} and the Sobolev embedding theorem imply $\|u^{(\nu)}-u_{K,L}^{\mathrm{app}}\|_{L^{\infty}([t_{\mathrm{in}},0]\times\R)}\to0$ as $\nu\to0$.
\subsection{Three-zone structure in shock development}
\begin{figure}[htbp]
	\centering
	\includegraphics[width=0.72\textwidth]{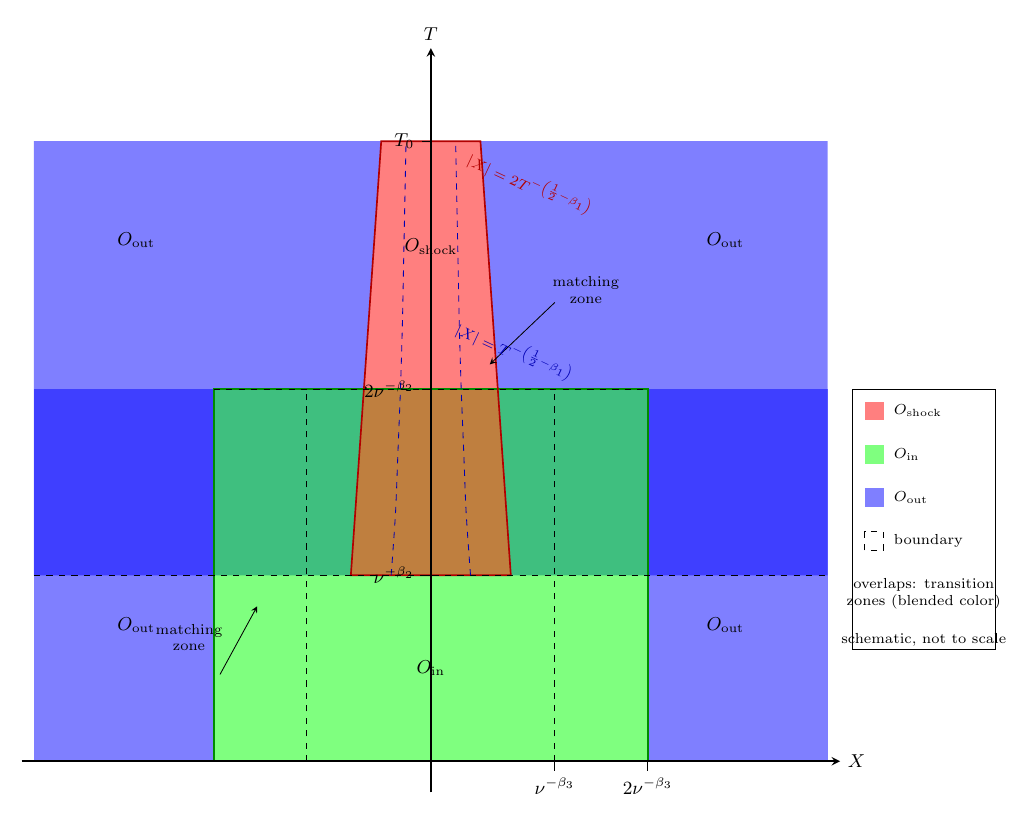}
	\caption{A schematic illustration of the regions $O_{\mathrm{shock}}$, $O_{\mathrm{in}}$ and $O_{\mathrm{out}}$ in the $(T,X)$ plane, with the boundaries given in \eqref{0.25}.}
	\label{fig1}
\end{figure}
	We now turn to the inviscid limit for shock development, i.e.,\ for $t\geq 0$. For convenience, we work in the inner variables given by
	\begin{equation}\label{0.23}
		\begin{aligned}
			T=\nu^{-\frac12}t,\quad X=\nu^{-\frac34}\x,\quad U=\nu^{-\frac14}u,\quad \mf{S}=\nu^{-\frac34}\mf{s}.
		\end{aligned}
	\end{equation}
	In these variables, \eqref{1.1} becomes
	\begin{equation}\label{0.24}
		\begin{aligned}
			\pa_TU-\dot{\mf{S}}\pa_XU+U\pa_XU=\pa_X^2U.
		\end{aligned}
	\end{equation}
We study \eqref{0.24} for $T\in[0,T_0]$, where $T_0=\ol{\vep}\nu^{-\frac12}$, and impose three distinct asymptotic expansions in the following regions:
	\begin{equation}\label{0.25}
		\begin{aligned}
			&O_{\mathrm{shock}}=\{(T,X): \nu^{-\beta_2}< T\leq T_0,\ |X|T^{\frac12-\beta_1}< 2\}\\
			&O_{\mathrm{in}}=\{(T,X): 0\leq T< 2\nu^{-\beta_2},\ |X|< 2\nu^{-\beta_3}\}\\
			&O_{\mathrm{out}}=\{(T,X): 0\leq T< 2\nu^{-\beta_2},\ |X|> \nu^{-\beta_3}\}
			\cup\{(T,X): \nu^{-\beta_2}< T\leq T_0,\ |X|T^{\frac12-\beta_1}> 1\},
		\end{aligned}
	\end{equation}
where $\beta_1,\beta_2,\beta_3\in(0,\frac1{100})$ are positive constants. Then $\{O_{\mathrm{shock}},O_{\mathrm{in}},O_{\mathrm{out}}\}$ forms an open covering of $[0,T_0]\times\R$ (see Figure~\ref{fig1}), and the matching zone consists of the intersections of these regions. The corresponding asymptotic expansions are described in the following subsections.\par

\subsection{Outer and shock-profile expansions}
	Returning to the $(t,\x)$ coordinates, we impose the asymptotic expansion:
	\begin{equation}\label{0.26}
		\begin{aligned}
			u^{(\nu)}\sim\sum\limits_{k\geq 0} \nu^k u_k^{\mathrm{out}}.
		\end{aligned}
	\end{equation}
	This is the outer expansion \eqref{0.10}, now continued in the shifted variables and solved with a jump across $\ol{x}=0$.
Rewriting \eqref{0.24} in the $(t,\x)$ coordinates and substituting \eqref{0.26}, we obtain
	\begin{equation}\label{0.27}
		\begin{aligned}
			\np_tu_k^{\mathrm{out}}-\dot{\mf{s}}\np_xu_k^{\mathrm{out}}+\frac12\np_x \sum\limits_{k_1+k_2=k}u_{k_1}^{\mathrm{out}}u_{k_2}^{\mathrm{out}}=\np_x^2u_{k-1}^{\mathrm{out}}.
		\end{aligned}
	\end{equation}
Under the Lax entropy condition $u(t,\mf{s}(t)^-)>\dot{\mf{s}}(t)>u(t,\mf{s}(t)^+)$, \eqref{0.27} can be solved as in \cite{CG}, with a jump discontinuity at $\x=0$ (see Proposition~\ref{prop3.2}). We expect \eqref{0.26} to hold in $O_{\mathrm{out}}$.\par
	Next, we seek an appropriate shock-profile expansion in $O_{\mathrm{shock}}$. It follows from \eqref{0.2}, \eqref{0.8} and \eqref{0.23} that the shock strength $U^{(0)}(T,0^-)-U^{(0)}(T,0^+)$ is of order $\ml{O}(T^{\frac12})$; moreover, a $\tanh$ profile analysis shows that the shock layer has width $|X|\sqrt{T}=\ml{O}(1)$. Hence we introduce the new variables
	\begin{equation}\label{0.28}
		\begin{aligned}
			V=\frac1{\sqrt{T}}U,\qquad Y=\sqrt{T}X+\Theta(T),
		\end{aligned}
	\end{equation}
	where $\Theta$ is a function to be chosen later. A direct computation then gives
	\begin{equation}\label{0.29}
		\begin{aligned}
			\frac1T\pa_TV+(\frac{Y-\Theta}{2T^2}+\frac{{\dot{\Theta}}}{T}-\frac{\dot{\mf{S}}}{\sqrt{T}})\pa_YV+V\pa_YV+\frac{1}{2T^2}V={\pa_Y^2V}.
		\end{aligned}
	\end{equation}
	We make the following shock-profile expansion for $V$:
	\begin{equation}\label{0.30}
		\begin{aligned}
			V(T,Y)\sim \sum\limits_{j\geq 0} T^{-2j}V_j^{\mathrm{shock}}(T,Y),
		\end{aligned}
	\end{equation}
	and, accordingly, for $\Theta$:
	\begin{equation}\label{0.31}
		\begin{aligned}
			\Theta(T)\sim \sum\limits_{j\geq 0}T^{-2j}\Theta_j^{\mathrm{S}}(T),
		\end{aligned}
	\end{equation}
	where the phase coefficients are chosen to ensure solvability, and we expect $\Theta_j^{\mathrm{S}}=\ml{O}(1)$. Collecting terms of order $T^{-2j}$ yields the equation for $V_j^{\mathrm{shock}}$. Moreover, $V_0^{\mathrm{shock}}$ satisfies
	\begin{equation}\label{0.32}
		\begin{aligned}
			\pa_Y^2V_0^{\mathrm{shock}}=V_0^{\mathrm{shock}}\pa_YV_0^{\mathrm{shock}}-\frac{\dot{\mf{S}}}{\sqrt{T}}\pa_YV_0^{\mathrm{shock}}.
		\end{aligned}
	\end{equation}
This is the viscous shock profile. At fixed physical time $t>0$, we have $T\sim\nu^{-\frac12}$; after the corresponding rescalings, these expansions have the outer and shock-layer structure of \cite{Goodman}. Their construction uses carefully chosen viscous initial data. Here both solutions evolve from the same prescribed smooth initial data, so we must also control the regime $T\ll\nu^{-\frac12}$, i.e.,\ $t\ll1$, as the shock strength tends to zero.
	\subsection{Inner expansion and grids}
At fixed $(T,X)$, we seek the expansions
\begin{equation}\label{0.33}
\begin{aligned}
U^{(\nu)}\sim\sum_{\ell\ge0}\nu^{\ell/4}U_\ell^{\mathrm{in}},
\qquad
\frac{\dot{\mf S}}{\sqrt T}\sim\sum_{\ell\ge0}\nu^{(\ell+1)/4}\mf S_\ell.
\end{aligned}
\end{equation}
Substitution into \eqref{0.24} gives
\begin{equation}\label{0.34}
\begin{aligned}
\pa_TU_\ell^{\mathrm{in}}
&-\sqrt T\sum_{\ell_1+\ell_2=\ell-1}\mf S_{\ell_1}\pa_XU_{\ell_2}^{\mathrm{in}}
+\frac12\pa_X\sum_{\ell_1+\ell_2=\ell}U_{\ell_1}^{\mathrm{in}}U_{\ell_2}^{\mathrm{in}}
=\pa_X^2U_\ell^{\mathrm{in}}.
\end{aligned}
\end{equation}
We extend the inner profiles of \cite{CG} through $0\le T\le1$ and use
their traces at $T=1$ for the subsequent construction.\footnote{Standard parabolic local well-posedness and continuation in spaces allowing polynomial growth give this extension; cf.\ \cite[proof of Proposition~4.2]{CG}. Proposition~\ref{prop6.1} gives a weighted refinement of the energy estimates in \cite[Section~5]{CG}; see Appendix~\ref{app:prop61}. Together with the initial corrections in Section~6, it yields \eqref{6.10} at $T=1$ and supplies the initial error bounds on $[1,T_0]$. The short-time bound and the subsequent estimates together cover $[0,T_0]$.}

Section~5 introduces two grids connecting the inner expansion with the
outer and shock-profile expansions. The outer--inner grid has partial sums
\begin{equation}\label{0.36}
U_{[K],\ell}^{\mathrm{OI}}=\sum_{k=0}^KU_{k,\ell}^{\mathrm{OI}},
\qquad
U_{k,[L]}^{\mathrm{OI}}=\sum_{\ell=0}^L\nu^{\ell/4}U_{k,\ell}^{\mathrm{OI}}.
\end{equation}
The latter approximates the rescaled outer profile $U_k^{\mathrm{out}}$.
For the shock--inner grid, we expand $V_j^{\mathrm{shock}}$ and $\Theta$
in powers of $\nu^{1/4}$, with coefficients $V_{j,\ell}^{\mathrm{SI}}$
and $\Theta_\ell^{\mathrm I}$. Taylor expansion of the composite shock
profile at fixed $(T,X)$ about $Y=\sqrt TX+\Theta_0^{\mathrm I}$ defines
$U_{j,\ell}^{\mathrm{SI}}$; see \eqref{4.36}. In particular,
\begin{equation}\label{0.38}
\begin{aligned}
U_{j,0}^{\mathrm{SI}}
&=T^{\frac12-2j}V_{j,0}^{\mathrm{SI}}(T,\sqrt TX+\Theta_0^{\mathrm I}),\\
U_{j,1}^{\mathrm{SI}}
&=T^{\frac12-2j}\left[V_{j,1}^{\mathrm{SI}}
+\Theta_1^{\mathrm I}\pa_YV_{j,0}^{\mathrm{SI}}\right]
_{Y=\sqrt TX+\Theta_0^{\mathrm I}}.
\end{aligned}
\end{equation}
Their partial sums approximate $U_j^{\mathrm{shock}}$
(Proposition~\ref{prop:SI-reconstruction}). The two grids match by
Proposition~\ref{prop4.9}, and we combine them as
\begin{equation}\label{0.39}
U_{[J],[K],\ell}^{\mathrm{SOI}}
=\zeta U_{[J],\ell}^{\mathrm{SI}}+(1-\zeta)U_{[K],\ell}^{\mathrm{OI}},
\qquad U_{[J],\ell}^{\mathrm{SI}}=\sum_{j=0}^JU_{j,\ell}^{\mathrm{SI}}.
\end{equation}
Here $\zeta$ selects the shock region. The phase terms in \eqref{4.43}
combine into truncation tails controlled by Lemma~\ref{lem5.8}; together
with grid matching, this gives residual decay of increasing order as
$J,K$ grow.

The remaining difficulty is the spatial mass. The normalized linear model
\begin{equation}\label{0.44}
\pa_TP+\pa_X(FP)-\pa_X^2P=G,
\qquad F=-\sqrt T\tanh(\sqrt TX/2)
\end{equation}
illustrates why localized initial data alone do not ensure time decay:
nonzero mass can produce a component of size $\sqrt T$ in a layer of
width $T^{-1/2}$. With zero-mass initial data and forcing, decay is established by weighted estimates for the spatial primitive below.

On $1\le T\le2\nu^{-\beta_2}$, Section~6 solves the hierarchy
\eqref{0.34} exactly for $U_{\ell,0}^{\mathrm{in}}$, with
\begin{equation}\label{0.40}
\begin{aligned}
U_{\ell,0}^{\mathrm{in}}
&=U_{[J],[K],\ell}^{\mathrm{SOI}}+P_{J,K,\ell,0},\\
U_{\ell,0}^{\mathrm{in}}(1,X)
&=U_\ell^{\mathrm{in}}(1,X)-\mu_\ell\rho(X),
\qquad |\mu_\ell|\lesssim\nu^{-\beta_2\ell/2},
\end{aligned}
\end{equation}
where $\rho\in C_c^\infty(\R)$ is fixed and $\int_\R\rho\,\md X=1$.
The correction equation is nonlinear for $\ell=0$ and linear at each
subsequent induction step. Subtracting a suitable time-dependent multiple
of $\rho$ gives an auxiliary equation with zero-mean initial data and
forcing. The Cole--Hopf transform at $\ell=0$ and weighted estimates for
spatial primitives yield, for fixed $\ell,m,n$ and sufficiently large $J,K$,
\begin{equation}\label{0.43}
|\pa_T^m\pa_X^nP_{J,K,\ell,0}|
\lesssim\nu^{-\beta_2\ell/2}
(T+X^2)^{-\psi_\ell(J,K)+a_{m,n}^{(\ell)}}.
\end{equation}
The index function $\psi_\ell$ tends to infinity with $J,K$, while
$a_{m,n}^{(\ell)}$ is independent of $\nu,J,K$.

Finally, Section~7 uses a partition of unity subordinate to \eqref{0.25}
to assemble
\[
U^{\mathrm{app}}_{J,K,L}
=\varsigma_oU_{[K]}^{\mathrm{out}}
+\varsigma_sU_{[J]}^{\mathrm{shock}}
+\varsigma_i\sum_{\ell=0}^L\nu^{\ell/4}U_{\ell,0}^{\mathrm{in}}.
\]
Matching makes its residual arbitrarily small in $\nu$ by choosing the
truncation orders large. For $W=U^{(\nu)}-U^{\mathrm{app}}_{J,K,L}$,
we write $W=W_1+W_2+W_3$: $W_1$ accounts for the mass change induced by the residual forcing,
and $W_2$ carries the remaining initial mass. Although $W_2$ can grow in amplitude,
it remains localized near the shock and decays exponentially away from it.
The zero-mass remainder $W_3$ satisfies a correction equation similar to those in
Section~6 and is controlled by analogous weighted estimates. Together, these bounds give
\[
|W(T,X)|\lesssim1+\sqrt T\,e^{-\delta\sqrt T|X|},
\qquad 1\le T\le T_0.
\]
Comparing $U^{\mathrm{app}}_{J,K,L}$ with $U^{(0)}$ in the three regions
and returning to $(t,x)$ yields Theorem~\ref{thm} (ii).

%

	\section{Inviscid Structure}
	In this section we describe the structure of the inviscid solution
	$u=u^{(0)}$ under Assumption~\ref{ass:init}, using the characteristic flow
	\eqref{2.2} below; the homogeneous expansions obtained here underlie the
	matched asymptotics of Sections~4--5. Section~3.1 constructs the shock
	development for $t>0$ by a fixed-point argument, with the Rankine--Hugoniot
	and Lax conditions and estimates on the jump and mean of $u$. Section~3.2
	introduces the homogeneous gauge $\mf{y}$, the distance $\mf{d}$, and the
	homogeneous expansions describing the $C_x^{1/3}$ cusp.
		We consider the inviscid equation
	\begin{equation}\label{2.1}
		\left\{\begin{aligned}
			&\pa_tu+u\pa_xu=0\\
			&u(t_{\mathrm{in}},x)=u_0(x).
		\end{aligned}
		\right
		.
	\end{equation}
	Under Assumption~\ref{ass:init}, $u$ develops a single nondegenerate
	shock at $(t,x)=(0,0)$; as in Section~2, the characteristic flow
	$\pa_t\eta(t,y)=u\circ\eta$, $\eta(t_{\mathrm{in}},y)=y$, yields
		\begin{equation}\label{2.2}
		\begin{aligned}
			u(t,\eta(t,y))&=u_0(y),\\
			\eta(t,y)&=y+(t-t_{\mathrm{in}})u_0(y).\\
		\end{aligned}
	\end{equation}
For later use, set $u^*(x)=u(0,x)$.

	\subsection{Shock development of Burgers equation}
	In this subsection, we consider the inviscid solution for $t>0$, which corresponds to the shock development stage. 
	The shock curve $\Gamma_s=\{(t,\mf{s}(t)):t\geq 0\}$ satisfies the Rankine--Hugoniot condition, which states that
	\begin{equation}\label{2.10}
		\begin{aligned}
			\dot{\mf{s}}(t)\,[[u]]=\frac{1}{2}[[u^2]].
		\end{aligned}
	\end{equation}
	We construct $\mf{s}(t)$ by a fixed-point argument. To this end, we first introduce some notation.\par

	\begin{definition}[Jump, mean, left/right limits, domain]
		Given a smooth shock curve $\mf{s}:[0,T]\to\R$, we denote by
		\[
			D_T=\bigl([0,T]\times\R\bigr)\setminus\{(t,\mf{s}(t)):t\in[0,T]\}
		\]
		the space-time domain excluding the shock curve. For any function $f:D_T\to\R$, the left and right limits at the shock are
		\[
			f^-(t)=\lim_{x\to\mf{s}(t)^-}f(t,x),\qquad 
			f^+(t)=\lim_{x\to\mf{s}(t)^+}f(t,x).
		\]
		The jump and the mean of $f$ across the shock are respectively
		\[
			[[f]]=[[f(t)]]=f^-(t)-f^+(t),\qquad
			\langle\!\langle f\rangle\!\rangle=\langle\!\langle f(t)\rangle\!\rangle
			=\frac{f^-(t)+f^+(t)}{2},
		\]
		for all $t\in[0,T]$. The dependence of $f^\pm$, $[[f]]$, and $\langle\!\langle f\rangle\!\rangle$ on $\mf{s}$ is suppressed.
	\end{definition}
	
	Next, we introduce the notion of a regular shock curve, which will be used throughout this subsection.
	\begin{definition}[Regular shock curve]\label{def3.2}
		Define the constant
			\begin{equation}\label{b}
			\begin{aligned}
			b=\sqrt{\frac ac}.
			\end{aligned}
		\end{equation}
		
		Throughout, $T_s>0$ is an arbitrary given number, not a special value.
		For every $t\in [0,T_s]$, we define
		\begin{equation*}
			\begin{aligned}
				\Sigma(t)&=[-\frac 12C_1t^2,\frac 12C_1t^2]\\
				\Omega(t)&=[-2bt^{\frac12},-\frac b2t^{\frac 12}]\cup[\frac b2t^{\frac12},2bt^{\frac 12}]. \\
			\end{aligned}
		\end{equation*}
For a given $T\in(0,T_s]$, we say that $\mf{s}:[0,T]\to\R$ is a regular shock curve if $\mf{s}$ satisfies
\begin{equation}\label{2.12}
	\begin{aligned}
	\mf{s}\in \Sigma(t),\qquad |\dot{\mf{s}}(t)|\leq C_1t,\qquad  |\ddot{\mf{s}}(t)|\leq C_1
	\end{aligned}
\end{equation}
for all $t\in (0,T]$.
	\end{definition}
We note that $C_1>0$ is a sufficiently large constant, to be chosen later. Our main result of this subsection is given as follows:
\begin{prop}[Shock development for Burgers equation]\label{prop sdb}
Suppose that the initial data $u_0$ of \eqref{2.1} satisfies Assumption \ref{ass:init}. Then there exists a sufficiently small constant $\ol{\vep}>0$, a regular shock curve $\mf{s}:[0,\ol{\vep}]\to \R$ satisfying \eqref{2.12}, and a function $u:D_{\ol{\vep}}\to \R$ such that 
	\[
\pa_tu+u\pa_xu=0\quad\text{in }D_{\ol{\vep}},\qquad u|_{t=0}=u^*.
\]
Moreover, $(u,\mf{s})$ satisfies the Rankine--Hugoniot condition \eqref{2.10} and the Lax entropy condition, i.e.,\ $u(t,\mf{s}(t)^-)>\dot{\mf{s}}(t)>u(t,\mf{s}(t)^+)$ for $0<t\leq\ol{\vep}$.
\end{prop}
\begin{remark}
Although the construction requires small $\lvp$, \eqref{2.10} allows continuation beyond $\lvp$ while the two branches remain smooth up to the shock and the strict Lax condition persists. The conclusions of Proposition~\ref{prop sdb} then remain valid, with $C_1$ in \eqref{2.12} enlarged if necessary. The argument for Theorem~\ref{thm}(ii) applies on any fixed finite interval with this single-shock structure, with constants depending on the interval but independent of $\nu$.
\end{remark}

As a first step toward the proof of Proposition \ref{prop sdb}, we construct the solution for a prescribed shock curve.

\begin{prop}[Burgers solution with a prescribed shock location]\label{prop2.7}
	Under Assumption~\ref{ass:init}, let $\mf{s}:[0,T_s]\to\R$ be a regular shock curve. There exist $\ol{\vep}\in(0,T_s]$ and a function $u:D_{\ol{\vep}}\to\R$ solving
	\[
		\pa_tu+u\pa_xu=0\quad\text{in }D_{\ol{\vep}},\qquad u|_{t=0}=u^*,
	\]
	which is $C^2$ in $D_{\ol{\vep}}$ and has a jump discontinuity across $\Gamma_s$ for $t>0$. Moreover, the jump $[[u(t)]]$ and the mean $\langle\!\langle u(t)\rangle\!\rangle$ satisfy the following estimates for all $t\in(0,\ol{\vep}]$:
	\[
		\begin{aligned}
		\bigl|[[u(t)]]-2abt^{1/2}\bigr| &\leq CC_1 t,\quad
		\bigl|\langle\!\langle u(t)\rangle\!\rangle-Bt+\frac{\mf{s}(t)}{2t}\bigr|  \leq CC_1 t^{\frac 32},\\
		\bigl|\tfrac{d}{dt}[[u(t)]]-abt^{-1/2}\bigr| &\leq CC_1,\quad
		\bigl|\tfrac{d}{dt}(\langle\!\langle u(t)\rangle\!\rangle+\frac{\mf{s}(t)}{2t})-B\bigr| \leq CC_1t^{\frac 12},\\
		\bigl|\tfrac{d^2}{dt^2}[[u(t)]]+\tfrac12abt^{-3/2}\bigr| &\leq CC_1 t^{-1},\quad
		\bigl|\tfrac{d^2}{dt^2}(\langle\!\langle u(t)\rangle\!\rangle+\frac{\mf{s}(t)}{2t})\bigr|  \leq CC_1 t^{-\frac 12},
		\end{aligned}
	\]
	where the constants $B$ and $C>0$ depend only on $u_0$.
\end{prop}
We will also need the following lemma on the invertibility of the flow map.

\begin{lemma}[Local inversion of the Burgers flow map]\label{lem2.8}
	Assume that $\mf{s}$ is a regular shock curve on $[0,T_s]$. Then there exists a sufficiently small constant $\ol{\vep}\in (0,T_s]$, depending only on $u_0$ and $C_1$, such that for all $t\in (0,\ol{\vep}]$ the following holds. There exists a unique largest $y_{+}=y_+(t)>0$ and a unique smallest $y_-=y_-(t)$ such that
	\begin{equation}\label{2.13}
		\begin{aligned}
			\mf{s}(t)=\eta(t,y_{\pm}(t)),
		\end{aligned}
	\end{equation}
	and moreover we have
		\begin{equation}\label{2.15}
		\begin{aligned}
			|y_{\pm}(t)\mp bt^{\frac 12}|\leq \mathsf{m}C_1 t\qquad \Rightarrow   \qquad \frac 23 bt^{\frac12}\leq |y_{\pm}(t)|\leq \frac 32 bt^{\frac12},
		\end{aligned}
	\end{equation}
	where $\mathsf{m}$ depends only on $u_0$. Define $y_{\pm}(0)=0$. Note that $y_{\pm}(t)\in \Omega(t)$ for all $t\in [0,\ol{\vep}]$. Moreover, define the set of labels
	\begin{equation*}
		\begin{aligned}
		\Upsilon(t)=\R\setminus[y_-(t),y_+(t)];
		\end{aligned}
	\end{equation*}
	then the map $\eta(t,\cdot):\Upsilon(t)\to\R\setminus\{\mf{s}(t)\}$ is a bijection. Setting $J(t,y)=\pa_y\eta(t,y)$, for $t\in [0,\ol{\vep}]$ and $y\in \Upsilon(t)\cap [-1,1]$ we have
	\begin{equation}\label{jty}
		\begin{aligned}
			J(t,y)\gtrsim t+y^2,
		\end{aligned}
	\end{equation}
and, for all $s\in [0,t)$ and $y\in \Upsilon(t)$, we have
\begin{equation}\label{tra}
	\begin{aligned}
		|\eta(s,y)-\mf{s}(s)|\gtrsim t^{1/2}|t-s|.
	\end{aligned}
\end{equation}
The above estimate implies that, for every $y\in \Upsilon(t)$, the trajectory $\{\eta(s,y)\}_{s\in [0,t]}$ cannot intersect the shock curve before time $s=t$.
\end{lemma}

\begin{proof}
	We first note that from Assumption~\ref{ass:init}, $\pa_y^k\eta(0,0)=0$ for $k\leq 2$ and $\pa_y^3\eta(0,0)>0$. Taking a Taylor expansion of $J=\pa_y\eta$ near $(t,y)=(0,0)$ gives, for some positive constants $a,c>0$,
		\begin{equation}\label{J}
		\begin{aligned}
			J(t,y)=-at+3cy^2+\ml{O}(|y|^3+|ty|).
		\end{aligned}
	\end{equation}
	Integrating \eqref{J} once and taking $x=\eta(t,y)$, we obtain, for some smooth functions $g_1(t,y)$ and $g_2(t,y)$,
		\begin{equation}\label{xy}
		\begin{aligned}
		x=-aty+cy^3+ty^3g_1(t,y)+y^4g_2(t,y).
		\end{aligned}
	\end{equation}
	Hence \eqref{2.13} is equivalent to 
	\begin{equation}\label{2.14}
		\begin{aligned}
			\mf{s}(t)=-aty+cy^3+ty^3g_1(t,y)+y^4g_2(t,y).
		\end{aligned}
	\end{equation}
Introduce
	\begin{equation*}
		\begin{aligned}
			\tau=t^{\frac 12}, \quad z=y t^{-\frac 12},\quad \zeta=\mf{s}(t)t^{-\frac32}.
		\end{aligned}
	\end{equation*}
Then \eqref{2.14} reduces to
	\begin{equation}\label{2.16}
		\begin{aligned}
			\zeta=-az+cz^3+G(\tau,z),
		\end{aligned}
	\end{equation}
where $G(\tau,z)=\tau^2z^3g_1(\tau^2,z\tau)+\tau z^4g_2(\tau^2,z\tau)$. If $\tau=0$, then for $\zeta\in(-\zeta_{\rm crit},\zeta_{\rm crit})$ with $\zeta_{\rm crit}=\frac{2a^{3/2}}{3\sqrt{3c}}$, the cubic equation $cz^3-az-\zeta=0$ has three distinct real roots, with a unique largest root $Z_+=Z_+(0,\zeta)$ and a unique smallest root $Z_-(0,\zeta)$, given by
	\begin{equation*}
		\begin{aligned}
			Z_+(0,\zeta)&=2\sqrt{\frac{a}{3c}}\,
			\cos\!\Bigl(\frac13\arccos\!\Bigl(\frac{3\zeta}{2}\sqrt{\frac{3c}{a^3}}\Bigr)\Bigr),\\
			Z_-(0,\zeta)&=2\sqrt{\frac{a}{3c}}\,
			\cos\!\Bigl(\frac13\arccos\!\Bigl(\frac{3\zeta}{2}\sqrt{\frac{3c}{a^3}}\Bigr)+\frac{2\pi}{3}\Bigr).
		\end{aligned}
	\end{equation*}
The proof is then completed by an application of the implicit function theorem, provided we show that
	\begin{equation*}
		\begin{aligned}
		\pa_z[-az+cz^3+G(\tau,z)]|_{(\tau,z)=(0,Z_{\pm})}=3cZ^2_{\pm}-a>0.
		\end{aligned}
	\end{equation*}
Using the trigonometric form,
	\[
		\begin{aligned}
		3cZ_+^2-a &= a\bigl(1+2\cos\!\bigl(\tfrac23\arccos\!\bigl(\tfrac{3\zeta}{2}\sqrt{\tfrac{3c}{a^3}}\bigr)\bigr)\bigr),\\
		3cZ_-^2-a &= a\bigl(1+2\cos\!\bigl(\tfrac23\arccos\!\bigl(\tfrac{3\zeta}{2}\sqrt{\tfrac{3c}{a^3}}\bigr)+\tfrac{4\pi}{3}\bigr)\bigr),
		\end{aligned}
	\]
	which are both strictly positive, we obtain positive and negative roots $Z_{\pm}(\tau,\zeta)$. Now we take
	\begin{equation}\label{yeq}
		\begin{aligned}
		y_{\pm}(t)=t^{\frac{1}{2}}Z_{\pm}(t^{\frac12},\mf{s}(t)t^{-\frac32}).
		\end{aligned}
	\end{equation}
Using \eqref{2.12}, we see that $y_{\pm}$ solves \eqref{2.13}.\par
Now we investigate the location of $y_{\pm}(t)$. Note that $Z_{+}(0,0)=-Z_{-}(0,0)$, $Z_{\pm}(\tau,\zeta)$ is smooth near $(\tau,\zeta)=(0,0)$, and $\mf{s}$ is a regular shock curve. Since $Z_+(0,0)=\sqrt{\frac ac}=b$, from \eqref{yeq} we obtain
	\begin{equation}\label{2.24}
		\begin{aligned}
			|y_{\pm}(t)\mp bt^{\frac 12}|\leq CC_1 t.\qquad
		\end{aligned}
	\end{equation}
For later use, we compute $y_{\pm}(t)$ more explicitly. By the implicit function theorem, we obtain
	\begin{equation*}
		\begin{aligned}
		&\pa_{\zeta}{Z}_{\pm}(0,0)=\frac{1}{2a}\\
		&\pa_{\tau}{Z}_{\pm}(0,0)=-\frac{b^4}{2a}g_2(0,0).\\
		\end{aligned}
	\end{equation*}
Hence we obtain
		\begin{equation*}
		\begin{aligned}
		|y_{\pm}\mp bt^{\frac12}+\frac{b^4}{2a}g_2(0,0)t-\frac{\mf{s}(t)}{2at}|\leq CC_1t^{\frac 32}.
		\end{aligned}
	\end{equation*}
We note that $\eta(0,\cdot)$ is a bijection on $\R$. To prove that $\eta(t,\cdot)$ is a bijection for $t>0$, it suffices to show that $J(t,y)>0$ on $\Upsilon(t)$. 
From \eqref{yeq} and \eqref{J}, we have
	\begin{equation*}
		\begin{aligned}
		\lim\limits_{y\to y_{\pm}}J(t,y)\gtrsim t\gtrsim t+y_{\pm}^2.\\
		\end{aligned}
	\end{equation*}
Taking $\ol{\vep}\ll 1$, we deduce from \eqref{J} that there exists a constant $\delta>0$, independent of $t$, such that $J(t,y)\gtrsim t+y^2$ for $y\in [y_--\delta,y_-]\cup[y_+,y_++\delta]$. Moreover, $J(t,y)$ has a positive lower bound for $|y|\geq \delta$ provided $\ol{\vep}$ is sufficiently small. Hence \eqref{jty} holds and $\eta(t,\cdot)$ is a bijection.\par
Now we prove \eqref{tra}. Using \eqref{2.12}, \eqref{2.24} and Assumption \ref{ass:init}, we obtain, for $s\in [0,t]$,
\begin{equation*}
	\begin{aligned}
		\mf{s}(s)-\eta(s,y_-(t))&=-\int_s^t\bigl(\dot{\mf{s}}(\tau)-(\pa_t\eta)(\tau,y_-(t))\bigr)\md \tau\\
		&=\int_s^t\bigl(u_0(y_-(t))-\dot{\mf{s}}(\tau)\bigr)\md \tau\\
		&\gtrsim \int_s^t\bigl(|y_-(t)|-C\tau\bigr)\md \tau\\
		&\gtrsim t^{1/2}|t-s|.
	\end{aligned}
\end{equation*}
Similarly, we can prove that $\eta(s,y_+(t))-\mf{s}(s)\gtrsim t^{1/2}|t-s|$. Combining these estimates with the monotonicity of $\eta$, we obtain \eqref{tra}.

\end{proof}
\begin{proof}[Proof of Proposition \ref{prop2.7}]
The smooth roots in \eqref{yeq} satisfy $Z_\pm(-\tau,-\zeta)=-Z_\mp(\tau,\zeta)$.
Thus, with $\mf{s}(t)/t^2$ held fixed, the jump and mean are respectively odd and even in $\tau=\sqrt t$.
Taylor expansion gives
\begin{equation}\label{2.29}
\begin{aligned}
[[u(t)]]&=2abt^{\frac12}+t^{\frac32}E_1\Bigl(t^{\frac12},\frac{\mf{s}(t)}{t^2}\Bigr),\\
\langle\!\langle u(t)\rangle\!\rangle
&=Bt-\frac{\mf{s}(t)}{2t}+t^2E_2\Bigl(t^{\frac12},\frac{\mf{s}(t)}{t^2}\Bigr),
\end{aligned}
\end{equation}
where $B=\frac12b^4g_2(0,0)$ and $E_1,E_2$ are smooth in both arguments.
By \eqref{2.12}, $|(t\pa_t)^k(\mf{s}(t)/t^2)|\lesssim C_1$ for $k=0,1,2$.
The chain rule therefore yields
\[
\left|\pa_t^k E_i\Bigl(t^{\frac12},\frac{\mf{s}(t)}{t^2}\Bigr)\right|
\leq C(C_1)t^{-k},\qquad i=1,2,\quad k=0,1,2,
\]
where $C(C_1)$ also depends on $u_0$.
Differentiating \eqref{2.29} and taking $\ol{\vep}$ sufficiently small gives all the stated estimates.
\end{proof}
It remains to prove Proposition \ref{prop sdb}.

\begin{proof}[Proof of Proposition~\ref{prop sdb}]
	Let $\ml{R}_{\ol{\vep}}=\{\mf{s}\in W^{2,\infty}[0,\ol{\vep}]:\mf{s}\text{ is a regular shock curve}\}$. The reader can check $\ml{R}_{\lvp}$ is a convex set (cf. Definition~\ref{def3.2}).
	 To distinguish the means corresponding to different regular shock curves, we write $\langle\!\langle u(t)\rangle\!\rangle_{\mf{s}}$ for the mean defined by $\mf{s}$.
For any $\mf{s}\in \ml{R}_{\ol{\vep}}$, we define the nonlinear mapping $\ml{A}$:
	\begin{equation*}
		\begin{aligned}
		\ml{A}[\mf{s}](t)=\int_0^t \langle\!\langle u(\tau)\rangle\!\rangle_{\mf{s}}\md \tau;
		\end{aligned}
	\end{equation*}
 then $\ml{A}$ is continuous on $\ml{R}_{\lvp}$. We wish to show that $\ml{A}$ has a fixed point in $\ml{R}_{\ol{\vep}}$ provided that $C_1$ is chosen sufficiently large and $\ol{\vep}$ sufficiently small.\par
	First, we show that $\ml{A}$ maps $\ml{R}_{\ol{\vep}}$ into itself. Given $\mf{s}\in \ml{R}_{\ol{\vep}}$, we clearly have $\ml{A}[\mf{s}](0)=0$. Also, from \eqref{2.12} and Proposition \ref{prop2.7} we obtain
	\begin{equation}\label{2.33}
		\begin{aligned}
			|\ml{A}[\mf{s}](t)|&=|\int_0^t \langle\!\langle u(\tau)\rangle\!\rangle_{\mf{s}}\md \tau|\leq \int_0^t\bigl(|B|\tau+|\frac{\mf{s}(\tau)}{2\tau}|+CC_1\tau^{\frac 32}\bigr)\md \tau\\
			&\leq  \frac{|B|}{2}t^2+\frac{C_1}{8}t^2+CC_1t^{\frac 52}\\
			|\frac{\md}{\md  t}\ml{A}[\mf{s}](t)|&=|\langle\!\langle u(t)\rangle\!\rangle_{\mf{s}}|
			\leq |B|t+ |\frac {{\mf{s}}(t)}{2t}|+CC_1t^{\frac 32}\\
			&\leq |B|t+\frac{1}{4}C_1t+CC_1t^{\frac 32}\\
				|\frac{\md^2}{\md  t^2}\ml{A}[\mf{s}](t)|&=|\frac{\md}{\md t}\langle\!\langle u(t)\rangle\!\rangle_{\mf{s}}|
				\leq |B|+ |\frac {\dot{\mf{s}}(t)}{2t}|+|\frac {{\mf{s}}(t)}{2t^2}|+CC_1t^{\frac 12}\\
				&\leq |B|+\frac{3}{4}C_1+CC_1t^{\frac 12}.\\
		\end{aligned}
	\end{equation}
Taking $C_1>20|B|$ and $\ol{\vep}$ sufficiently small, we obtain $\ml{A}[\mf{s}]\in \ml{R}_{\ol{\vep}}$. Moreover, if we denote by $\ol{\ml{R}}_{\ol{\vep}}$ the closure of $\ml{R}_{\ol{\vep}}$ in the $W^{1,\infty}[0,\ol{\vep}]$ topology, then $\ml{A}(\ol{\ml{R}}_{\ol{\vep}})$ is uniformly bounded in $W^{2,\infty}[0,\ol{\vep}]$ and hence precompact in the $W^{1,\infty}$ topology. By Schauder's fixed point theorem, there exists $\mf{s}\in \ol{\ml{R}}_{\ol{\vep}}$ such that $$\ml{A}[\mf{s}]=\mf{s}.$$
Following the proof of \eqref{2.33}, we obtain $\mf{s}\in \ml{R}_{\ol{\vep}}$, which completes the proof.
	
\end{proof}

\subsection{Homogeneous gauge}
In this subsection we introduce an appropriate gauge to capture the singularity of $u$. Define $\mf{y}(t,\x)$ by
\begin{equation}\label{gt}
	\begin{aligned}
		\x=-at\mf{y}+c\mf{y}^3,
	\end{aligned}
\end{equation}
where
 \begin{equation}\label{2.38}
	\begin{aligned}
		\x=x \text{ for }t\leq0,\quad \x=x-\mf{s}(t) \text{ for }t>0.
	\end{aligned}
\end{equation}
If $t\leq 0$, then $\mf{y}$ is uniquely determined by \eqref{gt}. For $t>0$, we need to impose an additional requirement to determine it.
Introduce $\mf{y}_{\pm}(t)$ as the positive [resp. negative] roots of the equation:
\begin{equation*}
	\begin{aligned}
		-at\mf{y}+c\mf{y}^3=0.
	\end{aligned}
\end{equation*}
A direct computation shows that $\mf{y}_{\pm}(t)=\pm\sqrt{\frac {at}c}=\pm b t^{\frac 12}$  for $t\geq 0$ (cf. \eqref{b}). Define the homogeneous map
\begin{equation*}
	\begin{aligned}
	\tilde{\eta}(t,\mf{y})=-at\mf{y}+c\mf{y}^3+\mf{s}(t)
	\end{aligned}
\end{equation*}
and the set of labels
\begin{equation*}
	\begin{aligned}
		\tilde{\Upsilon}(t)=\R\setminus[\mf{y}_-(t),\mf{y}_+(t)].
	\end{aligned}
\end{equation*}
For $t>0$, $\tilde{\eta}(t,\cdot):\tilde{\Upsilon}(t)\to\R\setminus\{\mf{s}(t)\}$ is a bijection. In the straightened coordinate $\x$ of \eqref{2.38}, we define $\mf{y}(t,\x)=\tilde{\eta}^{-1}(t,\x+\mf{s}(t))$, so that $\mf{y}(t,\x)$ takes values in $\tilde{\Upsilon}(t)$.

We introduce the following definition to capture the anisotropic scaling property of $\mf{y}$:
	\begin{definition}
	Given $r\in \R$, a function $f:\R\times \R\to \R$ is $r$-homogeneous if
	\begin{equation*}
		\begin{aligned}
			f(\lam^2t,\lam^3 \x)=\lam^rf(t,\x)
		\end{aligned}
	\end{equation*}
	for all $t\in \R, \x\in \R\setminus\{0\}$, and $\lam \in \R_+$.
\end{definition}
From \eqref{gt}, $\mf{y}(t,\x)$ is a 1-homogeneous function. Moreover, for $t>0$, $\mf{y}$ has a jump discontinuity at $\x=0$:
\begin{equation*}
	\begin{aligned}
	\lim\limits_{\x\to 0^-}\mf{y}(t,\x)=\mf{y}_-(t),\qquad \lim\limits_{\x\to 0^+}\mf{y}(t,\x)=\mf{y}_+(t).
	\end{aligned}
\end{equation*}
Hence it is convenient to work in the straightened coordinates $(t,\x)$. Let $(\np_t,\np_{x})$ denote the partial derivative operators in the $(t,\x)$ coordinates; time derivatives at $t=0$ are understood as one-sided. For $t>0$, a direct computation shows 
\begin{equation*}
	\begin{aligned}
	\np_t=\pa_t+\dot{\mf{s}}\pa_x,\qquad \np_x=\pa_x.
	\end{aligned}
\end{equation*}
Define
\begin{equation}\label{2.43}
	\begin{aligned}
		\mf{m}=a\np_x\mf{y}=\frac{a}{-at+3c\mf{y}^2}, \qquad \mf{d}=\mf{m}^{-\frac12}.
	\end{aligned}
\end{equation}
Then $\mf{m}$ is $-2$-homogeneous and $\mf{d}$ is 1-homogeneous. The following lemma shows that $\mf{d}$ is a ``distance'' function.
\begin{lemma}\label{lem2.11}
	For all $(t,\x)\in \R\times \R\setminus\{0\}$, we have
	\begin{equation}\label{2.44}
		\begin{aligned}
			\mf{d}\sim (|t|+\mf{y}^2)^{\frac12}.
		\end{aligned}
	\end{equation}
\end{lemma}
\begin{proof} For $t\le0$ the claim is immediate. For $t>0$,
	since $\mf{y}(t,\x)$ takes values in $\tilde{\Upsilon}(t)$, we have $|\mf{y}(t,\x)|\geq b t^{\frac12}$ and hence $-at+3c\mf{y}^2>2c\mf{y}^2>c(b^2t+\mf{y}^2)$. Then \eqref{2.44} follows from \eqref{2.43}.
	
\end{proof}
In analogy with Lemma \ref{lem2.8}, we define
\begin{equation*}
	\begin{aligned}
	y(t,\x)=\eta^{-1}(t,\x+\mf{s}(t)).
	\end{aligned}
\end{equation*}
This implies that $y(t,\x)$ takes values in $\Upsilon(t)$. Moreover, $y(t,\cdot)$ has a jump discontinuity at $\x=0$ for $t>0$. Introduce 
\begin{equation}\label{2.47}
	\begin{aligned}
		\ml{D}=\{(t,\x):0\leq t\leq \kappa, 0<|\x|\leq \kappa\},
	\end{aligned}
\end{equation}
where $\kappa$ is a sufficiently small constant. 
 We expect that $y(t,\x)=\mf{y}(t,\x)+\ml{O}(\mf{d}^2)$.
\begin{prop}\label{prop2.12}
	Suppose $t\in [t_{\mathrm{in}},\lvp]$. For every $N\in \N$, there exists a polynomial $q_N$ in $(t,\mf{y},\mf{m})$ whose terms have homogeneity at most $N$ such that for all $m,n\in \N_0$ and $(t,\x)\in \ml{D}$,
	\begin{equation*}
		\begin{aligned}
			\np_t^m\np_x^ny=\np_t^m\np_x^nq_N+\ml{O}(\mf{d}^{N+1-2m-3n}).
		\end{aligned}
	\end{equation*}
	Moreover, $y=\mf{y}+\ml{O}(\mf{d}^2)$.
\end{prop}
The proof of Proposition~\ref{prop2.12} for $t\le0$ is due to \cite{Anderson}; hence we focus on $t\in[0,\ol{\vep}]$.
To prove this proposition, we need some lemmas.
\begin{lemma}\label{lem2.13}
	In $\ml{D}$, we have $|y|,|\mf{y}|\lesssim \kappa^{\frac13}$ provided $\kappa$ is sufficiently small.
\end{lemma}
This follows directly from \eqref{xy} and \eqref{gt}, so we omit the proof. In what follows, we always assume that $\kappa$ is sufficiently small.

\begin{lemma}\label{lem2.14}
	In $\ml{D}$, we have $|y|\lesssim \mf{d}$.

\end{lemma}	
\begin{proof}
From \eqref{jty}, the map $\x\mapsto y(t,\x)$ is increasing. We treat the case $\x>0$; the case $\x<0$ follows in a similar fashion. Choose $M$ to be a large parameter. For $\x>0$, we have $y(t,\x)>y_{+}(t)=bt^{\frac 12}+\ml{O}(t)$. Note that
$y(t,\x)$ satisfies
\begin{equation}\label{2.48}
	\begin{aligned}
		\x+\mf{s}(t)=-aty+cy^3+ty^3g_1(t,y)+y^4g_2(t,y),
	\end{aligned}
\end{equation}
where $g_1,g_2$ are smooth. If $y(t,\x)\leq Mt^{\frac 12}$, then the bound follows directly from Lemma \ref{lem2.11}; hence we suppose that $y(t,\x)>Mt^{\frac 12}$. Choosing $M$ sufficiently large and $\kappa$ sufficiently small, we have
\begin{equation}\label{2.49}
	\begin{aligned}
		\x+\mf{s}(t)&=y(cy^2-at)+ty^3g_1(t,y)+y^4g_2(t,y)\\
		&\geq \frac c2 y^3.
	\end{aligned}
\end{equation}
It then follows from Proposition \ref{prop2.7} that
\begin{equation}\label{2.50}
	\begin{aligned}
	|y|&\lesssim \x^{\frac 13} +t^{\frac 23}\lesssim \x^{\frac 13}+\mf{d}.
	\end{aligned}
\end{equation}
To bound $\x$, we notice that from \eqref{2.49} we have $$\x\geq \frac{cM^3}2t^{\frac32}-\mf{s}(t)\geq \frac{cM^3}{4}t^{\frac32},$$
where $M$ is chosen sufficiently large. Moreover, \eqref{gt} implies that, for some constant $\mathrm{r}>0$,
$$\mf{y}(t,\x)\geq \mathrm{r}Mt^{\frac 12},$$
and hence, in view of Lemma \ref{lem2.11},
\begin{equation}\label{2.51}
	\begin{aligned}
	 |\x|^{\frac 13}&\leq |-at\mf{y}+c\mf{y}^3|^{\frac 13}
	\lesssim \mf{d}.
		\end{aligned}
	\end{equation}
Combining \eqref{2.50} and \eqref{2.51}, we complete the proof.

\end{proof}

\begin{proof}[Proof of Proposition \ref{prop2.12}]
	From \eqref{2.15} we have $y_{\pm}(t)=\mf{y}_{\pm}(t)+\ml{O}(\mf{d}^2)$. We treat the case $\x>0$; the case $\x<0$ is analogous. Following the proof of Lemma \ref{lem2.8}, we obtain
\begin{equation}\label{2.52}
	\begin{aligned}
		y_+(t,\x)=t^{\frac12}Z_+(t^{\frac12},\frac{\x+\mf{s}(t)}{t^{\frac32}}).
	\end{aligned}
\end{equation}
We also choose $M_1$ to be a large parameter. If $\x<M_1t^2$, then from \eqref{2.52} and \eqref{2.43} we obtain
\begin{equation}\label{2.53}
	\begin{aligned}
		y(t,\x)=y_+(t)+\ml{O}(t)\\
		\mf{y}(t,\x)=\mf{y}_+(t)+\ml{O}(t).\\
	\end{aligned}
\end{equation}
Hence $y(t,\x)=\mf{y}(t,\x)+\ml{O}(\mf{d}^2)$ for $0<\x<M_1t^2$. We claim that, by choosing $M_1$ large and $\kappa$ small, we have, for $\x\in [M_1t^2,\kappa]$,
\begin{equation}\label{2.54}
	\begin{aligned}
	\x-|\mf{s}(t)|-|ty^3g_1+y^4g_2|>\frac12 \x.
	\end{aligned}
\end{equation}
Let $M$ be the large parameter chosen in the proof of Lemma \ref{lem2.14}. If $y<Mt^{\frac 12}$, then for $\x>M_1t^2$ we have
\begin{equation*}
	\begin{aligned}
		|\mf{s}(t)|+|ty^3g_1+y^4g_2|\lesssim (1+M^4)t^2\lesssim \frac{1+M^4}{M_1}\x.
	\end{aligned}
\end{equation*}
If $y>Mt^{\frac 12}$, from \eqref{2.50} we get
\begin{equation*}
	\begin{aligned}
		|\mf{s}(t)|+|ty^3g_1+y^4g_2|\lesssim t^2+\x^{\frac{4}3}\lesssim (\frac{1}{M_1}+\kappa^{\frac13})\x.
	\end{aligned}
\end{equation*}
Choosing $M_1\gg 1$ and $\kappa\ll 1$, we get \eqref{2.54}.\par
Now suppose that $\x>M_1t^2$ and write $y(t,\x)=\mf{y}(t,\x+q(t,\x))$. In view of \eqref{gt}, \eqref{2.48}, \eqref{2.54} and Lemma \ref{lem2.14}, $q(t,\x)$ is well-defined and $\x+q(t,\x)>0$. Moreover,
\begin{equation*}
	\begin{aligned}
	|q(t,\x)|&\leq |\mf{s}(t)|+|ty^3g_1+y^4g_2|\lesssim \mf{d}^4\\
	|\x+q(t,\x)|&\geq \frac 12 \x.
	\end{aligned}
\end{equation*}
Hence, applying the mean value theorem, we obtain
\begin{equation}
	\begin{aligned}
		y(t,\x)&=\mf{y}(t,\x+q(t,\x))\\
		&=\mf{y}(t,\x)+\np_x\mf{y}(t,\x+\theta q)q(t,\x)\\
		&=\mf{y}(t,\x)+\ml{O}(\mf{d}^2).
	\end{aligned}
\end{equation}
This completes the proof of Proposition \ref{prop2.12} for $N=1$. The case $N\geq 2$ follows in a similar way.

\end{proof}	
From \eqref{2.2} and Proposition \ref{prop2.12}, we get the desired expansion, whose proof is similar to that in \cite{Anderson}. 
	\begin{prop}\label{prop2.15}
	Suppose $t\in[t_{\mathrm{in}},\ol{\vep}]$ and $\x\ne0$. There exists a sequence $\{\ol{u}_{\ell}\}_{{\ell}\in \N_0}$ of $(\ell+1)$-homogeneous functions such that for all $L,m,n\in \N_0$
	\begin{equation*}
		\begin{aligned}
			\np_t^m\np_x^nu(t,\x)=\np_t^m\np_x^n\sum\limits_{\ell=0}^L\ol{u}_\ell+\ml{O}(\mf{d}^{L+2-2m-3n}).
		\end{aligned}
	\end{equation*}
	Moreover, $\ol{u}_0=-a \mf{y}$.
\end{prop}
To end this section, we give the homogeneous expansion of the shock $\mf{s}(t)$, which follows directly from Propositions \ref{prop sdb} and \ref{prop2.15}.
	\begin{prop}\label{prop2.16}
	Suppose $t\in (0,\ol{\vep} ]$. There exists a sequence $\{\mf{s}_{\ell}\}_{{\ell}\in \N_0}$ of $(\ell+4)$-homogeneous functions such that for all $L,m\in \N_0$
	\begin{equation}\label{2.60}
		\begin{aligned}
			\pa_t^m\mf{s}=\pa_t^m\sum\limits_{\ell=0}^L\mf{s}_{\ell}+\ml{O}(t^{\frac{L+5}{2}-m}).
		\end{aligned}
	\end{equation}
\end{prop}

\section{Outer and Shock-Profile Expansions and Matching}
In this section we construct the outer expansion and the shock-profile
expansion of the viscous solution, and we prove that the two expansions
match on their overlapping region. We work in the shifted coordinates
$\x=x-\mf{s}(t)$ introduced in \eqref{2.38}; away from the shock the
estimates are expressed in terms of the distance $\mf{d}$, while near the
shock we use the inner variables \eqref{inner} and the layer variables
\eqref{3.8} on the interval $T\in[1,T_0]$, with
$T_0=\ol\vep\nu^{-1/2}$. Section~4.1 derives the outer expansion
$u\sim\sum_k\nu^k u_k^{\mathrm{out}}$ and records the estimates for the
outer profiles. Section~4.2 introduces the shock-profile expansion for $V$
and the phase coefficients $\Theta_j^{\mathrm{S}}$; Sections~4.3 and 4.4
then construct the principal profile $V_0^{\mathrm{shock}}$ and,
inductively, the higher-order profiles $V_j^{\mathrm{shock}}$ and phases
$\Theta_{j-1}^{\mathrm{S}}$. Section~4.5 proves the matching of the two
expansions in the regime $T\gg1$, $|Y|\gg1$; together with the inner-grid
expansion of Section~5, these expansions form the three-layer ansatz that
is used in Section~7 to approximate $U^{(\nu)}$.
\subsection{Derivation of outer expansion}
We consider the viscous Burgers equation
\begin{equation*}
	\begin{aligned}
		\pa_tu^{(\nu)}+u^{(\nu)}\pa_xu^{(\nu)}=\nu\pa_x^2u^{(\nu)}
	\end{aligned}
\end{equation*}
and rewrite it in the $(t,\x)$ coordinates introduced in \eqref{2.38}:
\begin{equation}\label{3.2}
	\begin{aligned}
		\np_tu^{(\nu)}-\dot{\mf{s}}\np_xu^{(\nu)}+u^{(\nu)}\np_xu^{(\nu)}=\nu \np_x^2u^{(\nu)}.
	\end{aligned}
\end{equation}
For convenience, we set $\mf{s}(t)\equiv 0$ for $t\in [t_{\mathrm{in}},0]$.

We seek an outer expansion of the form
$$u^{(\nu)}\sim \sum\limits_{k\geq 0}\nu^k u_k^{\mathrm{out}}.$$
Substituting this ansatz into \eqref{3.2} and collecting terms of like order in $\nu$, we obtain the following hierarchy of equations for the coefficients $u_k^{\mathrm{out}}$.
\begin{prop}\label{prop3.1}
	For every $k\in \N_0$, the function $u_k^{\mathrm{out}}$ satisfies
	\begin{equation}\label{3.3}
		\begin{aligned}
		\np_tu_k^{\mathrm{out}}-\dot{\mf{s}}\np_xu_k^{\mathrm{out}}+\frac12\np_x \sum\limits_{k_1+k_2=k}u_{k_1}^{\mathrm{out}}u_{k_2}^{\mathrm{out}}=\np_x^2u_{k-1}^{\mathrm{out}},
		\end{aligned}
	\end{equation}
	with initial data 
	\begin{equation*}
		\begin{aligned}
		u_0^{\mathrm{out}}(t_{\mathrm{in}},\x)=u_0(\x), \qquad 	u_k^{\mathrm{out}}(t_{\mathrm{in}},\x)=0 \text{ for }k\in \N,
		\end{aligned}
	\end{equation*}
	where we define $u_{-1}^{\mathrm{out}}\equiv0$.
\end{prop}
In \cite{CG}, estimates for all $u_k^{\mathrm{out}}$ are proved for $t\leq 0$; the extension to $t\in [0,\ol{\vep}]$ is straightforward and we do not repeat the details.
\begin{prop}\label{prop3.2}
	For $t\in [0,\lvp]$ and $\x\ne0$, the following estimate holds for all $k,m,n\in \N_0$:
	$$|\np_t^m\np_x^nu_k^{\mathrm{out}}(t,\x)|\lesssim \mf{d}^{1-4k-2m-3n}(t,\x).$$
\end{prop}
For later use, we calculate $u_1^{\mathrm{out}}$ explicitly. We solve for it in the original $(t,x)$ coordinates. Hence $u_1^{\mathrm{out}}(t,x)$ satisfies
	\begin{equation}\label{u1}
	\begin{aligned}
	\pa_tu_1^{\mathrm{out}}+u_0^{\mathrm{out}}\pa_xu_1^{\mathrm{out}}+\pa_xu_0^{\mathrm{out}}u_1^{\mathrm{out}}=\pa_x^2u_0^{\mathrm{out}}.
	\end{aligned}
\end{equation}

Recall we have solved for $u_0^{\mathrm{out}}$ in \eqref{2.2}. Set $J(t,y)=\pa_y\eta(t,y)=1+(t-t_{\mathrm{in}})u'_0(y)$. Then we can use $\eta$ to rewrite \eqref{u1} in a more convenient form:
\begin{equation}\label{u10}
	\begin{aligned}
		\pa_t(Ju_1^{\mathrm{out}}\circ\eta)&=J(\pa_x^2u_0^{\mathrm{out}})\circ\eta
		=\pa_y(\frac{u_0'}{J})=\frac{u_0''}{J^2}\\
		&=\pa_t(\frac{(t-t_{\mathrm{in}})u_0''}{J}).
	\end{aligned}
\end{equation}
Hence 
\begin{equation}\label{u1s}
	\begin{aligned}
		u_1^{\mathrm{out}}(t,\eta(t,y))&=\frac{(t-t_{\mathrm{in}})u_0''(y)}{(J(t,y))^2}.
	\end{aligned}
\end{equation}
From Lemma~\ref{lem2.8}, we obtain for $t\in [0,\lvp]$,
 \begin{equation}\label{48}
 	\begin{aligned}
 	\langle \langle u_1^{\mathrm{out}}\rangle \rangle=\ml{O}(t^{-1}).
 	\end{aligned}
 \end{equation}
\subsection{Derivation of the shock-profile expansion}
The outer expansion fails in a small neighborhood of the shock, where the profiles $u_k^{\mathrm{out}}$ have jump discontinuities across the shock. To describe the solution in this region, we introduce the inner variables
\begin{equation}\label{inner}
	\begin{aligned}
		t=\nu^{1/2}T,\quad \x=\nu^{3/4}X,\quad u=\nu^{1/4}U,\quad \mf{s}=\nu^{3/4}\mf{S}
	\end{aligned}
\end{equation}
together with the associated homogeneous rescalings of the gauges $\mf{y}$ and $\mf{d}$:
\begin{equation}\label{D}
	\begin{aligned}
			\mf{y}(\nu^{1/2}T,\nu^{3/4}X)=\nu^{1/4}\mf{Y}(T,X)\\
		\mf{d}(\nu^{1/2}T,\nu^{3/4}X)=\nu^{1/4}\mf{D}(T,X).
	\end{aligned}
\end{equation}
In these variables, equation \eqref{3.2} becomes
\begin{equation}\label{3.7}
	\begin{aligned}
		\pa_TU-\dot{\mf{S}}\pa_XU+U\pa_XU=\pa_X^2U.
	\end{aligned}
\end{equation}
Setting $T_0=\lvp \nu^{-1/2}$, we study \eqref{3.7} on the interval $[0,T_0]$. Note that $\dot{\mf{S}}=\ml{O}(\nu^{\frac14}T)$.
Note that the shock profiles constructed below are defined on
$T\in[1,T_0]$, although they are used in the approximation only for
$T\ge\nu^{-\beta_2}$; the regime $T\le2\nu^{-\beta_2}$ near the formation
point is covered by the inner-grid and outer expansions, as described in
Sections~5--7.
To capture the emergence of the shock profile, we introduce the change of variables
\begin{equation}\label{3.8}
	\begin{aligned}
		V=\frac1{\sqrt{T}}U,\qquad Y=\sqrt{T}X+\Theta(T).
	\end{aligned}
\end{equation}
A direct computation then gives
\begin{equation}\label{3.9}
	\begin{aligned}
		\frac1T\pa_TV+(\frac{Y-\Theta}{2T^2}+\frac{{\dot{\Theta}}}{T}-\frac{\dot{\mf{S}}}{\sqrt{T}})\pa_YV+V\pa_YV+\frac{1}{2T^2}V={\pa_Y^2V}.
	\end{aligned}
\end{equation}
We make the following shock-profile expansion for $V$:
\begin{equation}\label{3.10}
	\begin{aligned}
		V(T,Y)\sim \sum\limits_{j\geq 0} T^{-2j}V_j^{\mathrm{shock}}(T,Y),
	\end{aligned}
\end{equation}
and, accordingly, for $\Theta$:
\begin{equation}\label{3.11}
	\begin{aligned}
		\Theta(T)\sim \sum\limits_{j\geq 0}T^{-2j}\Theta_j^{\mathrm{S}}(T),
	\end{aligned}
\end{equation}
where we expect $\Theta_j^{\mathrm{S}}=\ml{O}(1)$.
Substituting \eqref{3.10} and \eqref{3.11} into \eqref{3.9} and collecting terms of like order in $T^{-2j}$, we obtain the following hierarchy of equations for the profiles $V_j^{\mathrm{shock}}$.

\begin{prop}\label{prop3.3}
	For each $j\in \N$, the profile $V_j^{\mathrm{shock}}$ satisfies
	\begin{equation}\label{3.12}
		\begin{aligned}
			\pa_Y^2&V_j^{\mathrm{shock}}-\pa_Y(V_0^{\mathrm{shock}}V_j^{\mathrm{shock}})+\frac{\dot{\mf{S}}}{\sqrt{T}}\pa_YV_j^{\mathrm{shock}}=\frac12\pa_Y\sum\limits_{\substack{j_1+j_2=j\\j_1,j_2>0}}V_{j_1}^{\mathrm{shock}}V_{j_2}^{\mathrm{shock}}+\frac Y2\pa_YV_{j-1}^{\mathrm{shock}}\\
			&+T\pa_TV_{j-1}^{\mathrm{shock}}
			-(2j-\frac52)V_{j-1}^{\mathrm{shock}}+\sum\limits_{j_1+j_2=j-1}[T\pa_T\Theta_{j_1}^{\mathrm{S}}-(2j_1+\frac 12)\Theta_{j_1}^{\mathrm{S}}]\pa_YV_{j_2}^{\mathrm{shock}}.
		\end{aligned}
	\end{equation}
	Moreover, the leading profile $V_0^{\mathrm{shock}}$ satisfies
	\begin{equation}\label{3.13}
		\begin{aligned}
			\pa_Y^2V_0^{\mathrm{shock}}=V_0^{\mathrm{shock}}\pa_YV_0^{\mathrm{shock}}-\frac{\dot{\mf{S}}}{\sqrt{T}}\pa_YV_0^{\mathrm{shock}}.
		\end{aligned}
	\end{equation}
\end{prop}

It remains to prescribe the boundary conditions for the profiles $V_j^{\mathrm{shock}}$, which are obtained by matching with the outer expansion. For the outer expansion, a Taylor expansion near $\x=0$ gives
\begin{equation}\label{3.14}
	\begin{aligned}
		u(t,\x)\sim \sum\limits_{k}\nu^ku_k^{\mathrm{out}}(t,\x)\sim\sum\limits_{i,k} \nu^k \x^i\np_x^iu_{k,\pm}^{\mathrm{out}}(t,0)(i!)^{-1},
	\end{aligned}
\end{equation}
where $\np_x^iu_{k,\pm}^{\mathrm{out}}(t,0)$ denotes the one-sided derivative of $u_k^{\mathrm{out}}$ at $\x=0$, with the sign $\pm$ chosen according to $\pm\x>0$. Define
\begin{equation}\label{vik}
	\begin{aligned}
	{V}_{i,k}^{\pm}(T)=\frac1{i!}\nu^{-\frac{1-4k-3i}{4}}\np_x^iu_{k,\pm}^{\mathrm{out}}(\nu^{1/2}T,0)|T|^{-\frac{1-4k-3i}{2}}.
	\end{aligned}
\end{equation}
 In view of Proposition \ref{prop3.1} and \eqref{D}, we have
\begin{equation*}
	\begin{aligned}
		|{V}^{\pm}_{i,k}|&\lesssim [\mf{D}(T,0)]^{{1-4k-3i}}|T|^{-\frac{1-4k-3i}{2}}\lesssim 1.
	\end{aligned}
\end{equation*}
Combining \eqref{inner}, \eqref{3.8}, \eqref{3.11} and \eqref{3.14} with the above, we obtain the following expansion for $V$:
\begin{equation}\label{3.16}
	\begin{aligned}
		V&\sim \sum\limits_{i,k\geq 0} T^{-2i-2k}[Y-\Theta(T)]^i {V}_{i,k}^{\pm}(T)\\
		&= \sum\limits_{i,k\geq 0} T^{-2i-2k} {V}_{i,k}^{\pm}(T)\sum\limits_{p+q=i}(-1)^{p}C_{i}^{p}\Theta^{p}Y^{q}\\
		&= \sum\limits_{i,k\geq 0} T^{-2i-2k} {V}_{i,k}^{\pm}(T)\sum\limits_{p+q=i}(-1)^{p}C_{i}^{p}Y^{q}\sum\limits_{j_1,...,j_p\geq 0}T^{-2(j_1+...+j_p)}\prod_{r=1}^p\Theta_{j_r}^{\mathrm{S}}\\
			&= \sum\limits_{j\geq 0}T^{-2j}\sum\limits_{\substack{i,k\geq 0\\0\leq p\leq i\\ j_1,...,j_p\geq 0\\i+k+j_1+...+j_p=j}}(-1)^{p}C_{i}^{p}Y^{i-p}{V}_{i,k}^{\pm}(T)\prod_{r=1}^p\Theta_{j_r}^{\mathrm{S}},
	\end{aligned}
\end{equation}
where we use the convention that the empty product $\prod_{r=1}^p\Theta_{j_r}^{\mathrm{S}}$ equals $1$ for $p=0$. Matching the orders in $T$ of \eqref{3.10} and \eqref{3.16}, we obtain
\begin{equation}\label{3.17}
	\begin{aligned}
	V_j^{\mathrm{shock}}\sim \sum\limits_{\substack{i,k\geq 0\\0\leq p\leq i\\ j_1,...,j_p\geq 0\\i+k+j_1+...+j_p=j}}(-1)^{p}C_{i}^{p}Y^{i-p}{V}_{i,k}^{\pm}(T)\prod_{r=1}^p\Theta_{j_r}^{\mathrm{S}} \qquad \text {as }Y\to \pm \infty.
	\end{aligned}
\end{equation}
Hence, we will construct solutions of \eqref{3.12} and \eqref{3.13} that satisfy the matching conditions \eqref{3.17}.\par
\begin{remark}\label{rem3.4}
	Here $\Theta_j^{\mathrm{S}}$ is used to ensure that \eqref{3.12} together with \eqref{3.17} can be solved. Although we have no explicit equation for $\Theta$, we note that, since only finitely many terms of the expansion are needed, it suffices to determine finitely many $\Theta_j^{\mathrm{S}}$. Choosing a sufficiently large $J$ and fixing it, we define $\Theta=\sum\limits_{j=0}^JT^{-2j}\Theta_j^{\mathrm{S}}$. Then we can solve \eqref{3.12} and \eqref{3.17} for $V_{\leq J+1}^{\mathrm{shock}}$, and no convergence issue arises in \eqref{3.11}. We use this observation throughout the paper.
\end{remark}
\subsection{Construction of $V_0^{\mathrm{shock}}$}
In this subsection, we construct the principal shock profile $V_0^{\mathrm{shock}}$ for $T\in[1,T_0]$. Notice that, in view of the Rankine--Hugoniot condition \eqref{2.10}, $V_{0,0}^+$ is determined by $\dot{\mf{S}}$ and $V_{0,0}^-$; hence we set $V_0^{\mathrm{shock}}(T,Y)=\Phi(Y,\frac{\dot{\mf{S}}}{\sqrt{T}},V_{0,0}^-)$.
From \eqref{3.13} and the matching condition \eqref{3.17}, the profile $\Phi$ satisfies
\begin{equation}\label{3.18}
	\begin{aligned}
	\pa_Y\Phi=\frac{1}{2}|\Phi|^2-\frac{1}{2}|V_{0,0}^-|^2-\frac{\dot{\mf{S}}}{\sqrt{T}}(\Phi-V_{0,0}^-).
	\end{aligned}
\end{equation}
The Rankine--Hugoniot condition \eqref{2.10}, together with \eqref{inner} and \eqref{3.8}, implies
\begin{equation}\label{3.19}
	\begin{aligned}
	\frac{\dot{\mf{S}}}{\sqrt{T}}=\frac{V_{0,0}^-+V_{0,0}^+}{2}.
	\end{aligned}
\end{equation}
Substituting it into \eqref{3.18}, we obtain
\begin{equation*}
	\begin{aligned}
		\pa_Y\Phi&=\frac{1}{2}(\Phi-V_{0,0}^-)(\Phi+V_{0,0}^--2\frac{\dot{\mf{S}}}{\sqrt{T}})\\
		&=\frac{1}{2}(\Phi-V_{0,0}^-)(\Phi-V_{0,0}^+).
	\end{aligned}
\end{equation*}
Solving this equation, we obtain the explicit profile
\begin{equation}\label{3.21}
	\begin{aligned}
		\Phi(Y,\frac{\dot{\mf{S}}}{\sqrt{T}},V_{0,0}^-)=\frac{V_{0,0}^-+V_{0,0}^+}{2}-\frac{V_{0,0}^--V_{0,0}^+}{2}\tanh\frac{(V_{0,0}^--V_{0,0}^+)Y}{4}.
	\end{aligned}
\end{equation}

Introduce
\begin{equation}\label{om}
	\begin{aligned}
	\om(T)=\frac{V_{0,0}^-(T)-V_{0,0}^+(T)}{2}.
	\end{aligned}
\end{equation} 
 From Proposition~\ref{prop2.7}, \eqref{gt}, Proposition~\ref{prop2.15}, \eqref{inner} and \eqref{3.8}, we have
\begin{equation}\label{4.22}
	\begin{aligned}
\om(T)&\gtrsim ab>0.
	\end{aligned}
\end{equation}
Using \eqref{3.19} and \eqref{3.21}, we derive
\begin{equation}\label{pro}
	\begin{aligned}
		V_0^{\mathrm{shock}}(T,Y)=\frac{\dot{\mf{S}}}{\sqrt{T}}-\om(T)\tanh\frac{\om(T)Y}{2}.
	\end{aligned}
\end{equation}
We introduce the cutoff profile
\begin{equation}\label{3.24}
	V_0^{\mathrm{cut}}(T,Y)=\chi(Y)\left\{\begin{aligned}
		&V_{0,0}^-(T),\quad Y<0\\
		&V_{0,0}^+(T),\quad Y\geq0.\\
	\end{aligned}
	\right
	.
\end{equation}
Here $\chi$ is a smooth even cutoff function such that $\chi=0$ for $|Y|\leq \frac 12$ and $\chi=1$ for $|Y|\geq 1$. Then we have the following estimate:
\begin{prop}\label{prop3.4}
	For all $m,n\in \N_0$ and $T\in [1,T_0]$, we have
	\begin{equation*}
		\begin{aligned}
			|\pa_T^m\pa_Y^n(V_0^{\mathrm{shock}}-V_{0}^{\mathrm{cut}})|\lesssim T^{-m}e^{-\frac34\om(T)|Y|}.
		\end{aligned}
	\end{equation*}
\end{prop}

\subsection{Construction of $V_j^{\mathrm{shock}}$ and $\Theta_{j-1}^{\mathrm{S}}$ for $j\in\N$}
In this subsection we construct, by induction, the shock profile $V_j^{\mathrm{shock}}$ and the phase coefficient $\Theta_{j-1}^{\mathrm{S}}$ for $j\in\N$ and $T\in[1,T_0]$. Note that, by \eqref{3.12}, $V_j^{\mathrm{shock}}$ solves a linear equation for $j\in\N$. In order to satisfy the matching condition \eqref{3.17}, we need to choose $\Theta_{j-1}^{\mathrm{S}}$ appropriately. We shall give estimates for $V_j^{\mathrm{shock}}$ and $\Theta_j^{\mathrm{S}}$. For later use, we define $V_j^{\mathrm{cut}}$ as the smooth truncation of the right-hand side of \eqref{3.17}:
\begin{equation}\label{3.27}
	V_j^{\mathrm{cut}}(T,Y)=\chi(Y)\sum\limits_{\substack{i,k\geq 0\\0\leq p\leq i\\ j_1,...,j_p\geq 0\\i+k+j_1+...+j_p=j}}(-1)^{p}C_{i}^{p}Y^{i-p}{V}_{i,k}^{\pm}(T)\prod_{r=1}^p\Theta_{j_r}^{\mathrm{S}},
\end{equation}
where the sign $\pm$ is chosen according to $\pm Y>0$, and $\chi$ is the smooth even cutoff function given in \eqref{3.24}.
\begin{prop}\label{prop3.6}
	There exist functions $V_j^{\mathrm{shock}}$ and $\Theta_j^{\mathrm{S}}$ such that for all $j,m,n\in\N_0$ and $T\in[1,T_0]$ we have
	\begin{equation*}
		\begin{aligned}
			|\pa_T^m\pa_Y^n(V_j^{\mathrm{shock}}-V_{j}^{\mathrm{cut}})|\lesssim T^{-m}e^{-\frac{1}{2}\om(T)|Y|}\\
			|\pa_T^m\Theta_{j}^{\mathrm{S}}|\lesssim T^{-m}.
		\end{aligned}
	\end{equation*}
	Moreover, $V_j^{\mathrm{shock}}$ solves \eqref{3.12} and satisfies \eqref{3.17}.
\end{prop}
The proof of Proposition~\ref{prop3.6} for general $j\in\N$ is given in
Appendix~\ref{app:prop46}. In the main text we prove only the case $j=1$, as
an example illustrating the induction; the profiles $V_1^{\mathrm{shock}}$ and
$\Theta_0^{\mathrm{S}}$ are constructed below. To this end, we need the
following lemma.

\begin{lemma}\label{lem3.5}
	For $k\in \N_0$ and $i\in \N_0$, the scaled Taylor coefficients $V_{i,k}^{\pm}(T)$ satisfy
	\begin{equation}\label{429}
		\begin{aligned}
			\Big(\frac{1-4k-3i}{2}V_{i,k}^{\pm}+T\pa_TV_{i,k}^{\pm}\Big)+\frac{i+1}{2}\sum\limits_{\substack{k_1+k_2=k\\i_1+i_2=i+1}}V_{i_1,k_1}^{\pm}V_{i_2,k_2}^{\pm}&=(i+2)(i+1)V_{i+2,k-1}^{\pm}\\
			&+\dot{\mf{S}}(i+1)T^{-\frac12}V_{i+1,k}^{\pm},
		\end{aligned}
	\end{equation}
	where we use the convention that any term with a negative index vanishes. 
	Moreover, we have 
		\begin{equation}\label{430}
		\begin{aligned}
			V_{1,0}^{+}+V_{1,0}^-=-1-2T\frac{\om'(T)}{\om(T)}.
		\end{aligned}
	\end{equation}
\end{lemma}
\begin{proof}
	We prove \eqref{429} for $Y>0$; the case $Y<0$ follows in a similar fashion. Starting from \eqref{3.2} and writing the Taylor coefficients
	$\widetilde{u}_{k,i}(t):=\frac1{i!}\np_x^i u_{k,+}^{\mathrm{out}}(t,0)$,
	we apply $\pa_x^i$ at $\x=0$ and divide by $i!$.
	Collecting terms of like order in $x$, we derive
	\[
	\np_t\widetilde{u}_{k,i}-\dot{\mf{s}}(i+1)\widetilde{u}_{k,i+1}
	+\frac{i+1}{2}\sum_{\substack{k_1+k_2=k\\i_1+i_2=i+1}}\widetilde{u}_{k_1,i_1}\widetilde{u}_{k_2,i_2}
	=(i+2)(i+1)\widetilde{u}_{k-1,i+2}.
	\]
	By the definition of the scaled coefficients $V_{i,k}$, we have $$\widetilde{u}_{k,i}(t)=\nu^{\frac{1-4k-3i}{4}}T^{\frac{1-4k-3i}{2}}V_{i,k}(T).$$
	 By \eqref{inner}, $\np_t=\nu^{-1/2}\pa_T$.
	After these relations are substituted into the equation above, all four terms carry the common factor $\nu^{\frac{-1-4k-3i}{4}}$; factoring out the lowest power of $T$, namely $T^{\frac{-1-4k-3i}{2}}$, we obtain \eqref{429}. The identity \eqref{430} comes from \eqref{429} by choosing $i=k=0$.
\end{proof}

\begin{proof}[Proof of Proposition~\ref{prop3.6}]

We now carry out the construction of $V_1^{\mathrm{shock}}$ and
$\Theta_0^{\mathrm{S}}$ in detail. When $j=1$ in the profile
equation \eqref{3.12}, the sum over $j_1,j_2>0$ is empty and the term
$-(2j-\frac52)V_{j-1}^{\mathrm{shock}}$ becomes
$\frac12V_0^{\mathrm{shock}}$; hence, moving the terms containing
$V_1^{\mathrm{shock}}$ to the left, we obtain
\[
\begin{aligned}
\pa_Y^2V_1^{\mathrm{shock}}-\pa_Y[(V_0^{\mathrm{shock}}-\tfrac{\dot{\mf{S}}}{\sqrt{T}})V_1^{\mathrm{shock}}]
&=\tfrac12V_0^{\mathrm{shock}}+\tfrac Y2\pa_YV_0^{\mathrm{shock}}+T\pa_TV_0^{\mathrm{shock}}\\
&\quad+[T\pa_T\Theta_0^{\mathrm{S}}-\tfrac12\Theta_0^{\mathrm{S}}]\pa_YV_0^{\mathrm{shock}}.
\end{aligned}
\]
By \eqref{3.27} with $j=1$, the cutoff profile takes the explicit form
\[
V_1^{\mathrm{cut}}(T,Y)=\chi(Y)\bigl(V_{0,1}^{\pm}(T)+YV_{1,0}^{\pm}(T)-V_{1,0}^{\pm}(T)\Theta_0^{\mathrm{S}}(T)\bigr),
\]
where the sign $\pm$ is chosen according to $\pm Y>0$ and $\chi$ is the
cutoff function of \eqref{3.24}. Set
$\ol V_1:=V_1^{\mathrm{shock}}-V_1^{\mathrm{cut}}$; by the matching
condition \eqref{3.17} we require $\ol V_1(T,Y)\to0$ as $Y\to\pm\infty$.
Writing $V_0^{\mathrm{shock}}=\ol V_0+V_0^{\mathrm{cut}}$ with
$\ol V_0:=V_0^{\mathrm{shock}}-V_0^{\mathrm{cut}}$ (so that $\ol V_0$ and its
derivatives are exponentially decaying by Proposition~\ref{prop3.4}) and
substituting $V_1^{\mathrm{shock}}=\ol V_1+V_1^{\mathrm{cut}}$ into the
equation above, we obtain
\[
\pa_Y^2\ol V_1-\pa_Y[(V_0^{\mathrm{shock}}-\tfrac{\dot{\mf{S}}}{\sqrt{T}})\ol V_1]
=g_{\ge1}+g_0+[T\pa_T\Theta_0^{\mathrm{S}}-\tfrac12\Theta_0^{\mathrm{S}}]\pa_YV_0^{\mathrm{shock}},
\]
where
\begin{align*}
g_{\ge1}&=\tfrac12\ol V_0+\tfrac Y2\pa_Y\ol V_0+T\pa_T\ol V_0+\pa_Y(\ol V_0V_1^{\mathrm{cut}}),\\
g_0&=\tfrac12V_0^{\mathrm{cut}}+\tfrac Y2\pa_YV_0^{\mathrm{cut}}+T\pa_TV_0^{\mathrm{cut}}
+\pa_Y[(V_0^{\mathrm{cut}}-\tfrac{\dot{\mf{S}}}{\sqrt{T}})V_1^{\mathrm{cut}}]-\pa_Y^2V_1^{\mathrm{cut}}.
\end{align*}
By Proposition~\ref{prop3.4} the term $g_{\ge1}$ is exponentially decaying.
To see that $g_0$ vanishes on $|Y|\ge1$, note that there $\chi\equiv1$, so
that $V_0^{\mathrm{cut}}=V_{0,0}^{\pm}$ with
$\pa_YV_0^{\mathrm{cut}}=\pa_Y^2V_0^{\mathrm{cut}}=0$, while
$V_1^{\mathrm{cut}}=V_{0,1}^{\pm}+YV_{1,0}^{\pm}-V_{1,0}^{\pm}\Theta_0^{\mathrm{S}}$
with $\pa_YV_1^{\mathrm{cut}}=V_{1,0}^{\pm}$ and
$\pa_Y^2V_1^{\mathrm{cut}}=0$; consequently for $|Y|\geq 1$,
\[
g_0=\tfrac12V_{0,0}^{\pm}+T\pa_TV_{0,0}^{\pm}+(V_{0,0}^{\pm}-\tfrac{\dot{\mf{S}}}{\sqrt{T}})V_{1,0}^{\pm}=0,
\]
where the last identity is the case $(k,i)=(0,0)$ of Lemma~\ref{lem3.5}
(the term involving $V_{2,-1}^{\pm}$ is absent). Hence $g_0$ is supported in
$\{\tfrac12\le|Y|\le1\}$, and $F_1:=g_{\ge1}+g_0$ decays exponentially as
$|Y|\to\infty$.

Integrating the equation for $\ol V_1$ once in $Y$ over
$(-\infty,Y]$ and using $\ol V_1\to0$, $\pa_Y\ol V_1\to0$ as $Y\to-\infty$,
we obtain
\begin{equation}\label{431}
	\begin{aligned}
		\pa_Y\ol V_1-(V_0^{\mathrm{shock}}-\tfrac{\dot{\mf{S}}}{\sqrt{T}})\ol V_1
		=\int_{-\infty}^{Y}F_1(T,Y')\,\md Y'
		+[T\pa_T\Theta_0^{\mathrm{S}}-\tfrac12\Theta_0^{\mathrm{S}}](V_0^{\mathrm{shock}}-V_{0,0}^{-}).
	\end{aligned}
\end{equation}

We need to extract the contribution of $\Theta_0^{\mathrm{S}}$ in the $\int F_1$ term, which comes from the presence of $\pa_Y[\chi (V_0^{\mathrm{shock}}-\frac{\dot{\mf{S}}}{\sqrt{T}})V_{1,0}^{\pm}\Theta_0^{\mathrm{S}}]$.  To cancel the boundary contribution at $Y=+\infty$ (recall that
$V_0^{\mathrm{shock}}\to V_{0,0}^{+}$ as $Y\to+\infty$ and
$V_{0,0}^{-}-V_{0,0}^{+}=2\om(T)$), we choose $\Theta_0^{\mathrm{S}}$ and use \eqref{430} to ensure
\begin{equation}\label{432}
	\begin{aligned}
		0&=\int_{\R}F_1(T,Y')\,\md Y'-2\om(T)[T\pa_T\Theta_0^{\mathrm{S}}-\tfrac12\Theta_0^{\mathrm{S}}]\\
		&=\int_{\R}F_1|_{\Theta_0^{\mathrm{S}}=0}(T,Y')\,\md Y'-2\om(T)[T\pa_T\Theta_0^{\mathrm{S}}-\tfrac12\Theta_0^{\mathrm{S}}-\frac{V_{1,0}^++V_{1,0}^-}{2}\Theta_0^{\mathrm{S}}]\\
		&=\int_{\R}F_1|_{\Theta_0^{\mathrm{S}}=0}(T,Y')\,\md Y'-2T\pa_T(\om\Theta_0^{\mathrm{S}}).
	\end{aligned}
\end{equation}
Setting $\wh{\ml F}_1(T)=\int_{\R}F_1|_{\Theta_0^{\mathrm{S}}=0}(T,Y')\,\md Y'$, we solve for $\Theta_0^{\mathrm{S}}$ from \eqref{432} to obtain
 \begin{equation}\label{t0}
	\begin{aligned}
	\Theta_0^{\mathrm{S}}(T)=-\frac1{2\om(T)}\int_{T}^{T_0}\frac{\wh{\ml F}_1(\tau)}{\tau}\,\md\tau.
	\end{aligned}
\end{equation}

 To show $|\Theta_0^{\mathrm{S}}|\lesssim 1$, we need to bound $\wh{\ml F}_1(T)$ carefully. Indeed, introduce $W=V_1^{\mathrm{cut}}|_{\Theta_0^{\mathrm{S}}=0}$. Then we rewrite $F_1|_{\Theta_0^{\mathrm{S}}=0}$ as  
 \begin{equation*}
 	\begin{aligned}
 	F_1|_{\Theta_0^{\mathrm{S}}=0}=\frac12 V_0^{\mathrm{shock}}+\frac Y2\pa_YV_0^{\mathrm{shock}}+T\pa_TV_0^{\mathrm{shock}}+\pa_Y[(V_0^{\mathrm{shock}}-\frac{\dot{\mf{S}}}{\sqrt{T}})W]-\pa_Y^2W.
 	\end{aligned}
 \end{equation*}
For $R>1$, introduce
 \begin{equation*}
 	\begin{aligned}
 		J(T,R)=\int_{-R}^R F_1|_{\Theta_0^{\mathrm{S}}=0}(T,Y')\md Y'.
 	\end{aligned}
 \end{equation*}
 After tedious calculation, we obtain
  \begin{equation*}
 	\begin{aligned}
 		J(T,R)&=R[\frac{\dot{\mf{S}}}{\sqrt{T}}+2T\pa_T(\frac{\dot{\mf{S}}}{\sqrt{T}})]-\om R\tanh\frac{\om R}{2}(V_{1,0}^+-V_{1,0}^-)\\
 		&-\om \tanh\frac{\om R}{2}(V_{0,1}^++V_{0,1}^-)-(V_{1,0}^+-V_{1,0}^-).
 	\end{aligned}
 \end{equation*}
Using \eqref{om}, \eqref{pro}, \eqref{429} and letting $R\to +\infty$, we obtain
 \begin{equation*}
	\begin{aligned}
		\wh{\ml F}_1=-\om (V_{0,1}^++V_{0,1}^-)+V_{1,0}^--V_{1,0}^+.
	\end{aligned}
\end{equation*}
We claim that 
 \begin{equation}\label{442}
	\begin{aligned}
		|\wh{\ml F}_1(T)|\lesssim \nu^{\frac14}\sqrt{T}.
	\end{aligned}
\end{equation}
Indeed, from Proposition~\ref{prop2.15} and \eqref{vik} we have
\begin{equation*}
	\begin{aligned}
	V_{1,0}^-(T)-V_{1,0}^+(T)&=\nu^{\frac12}T[[\pa_xu_0^{\mathrm{out}}(\nu^{\frac12}T,0)]]\\
	&=\nu^{\frac12}T [[\pa_x\ol{u}_0(\nu^{\frac12}T,0)]]+\ml{O}(\nu^{\frac14}\sqrt{T})\\
	&=\ml{O}(\nu^{\frac14}\sqrt{T}).
	\end{aligned}
\end{equation*}
Here we use the fact that $\ol{u}_0(t,x)$ is odd in $\x$. From \eqref{48}, we can also obtain $V_{0,1}^++V_{0,1}^-=\ml{O}(\nu^{\frac14}\sqrt{T})$. Hence \eqref{442} is true. We use it to estimate \eqref{t0} to get
\begin{equation*}
	\begin{aligned}
	|\Theta_0^{\mathrm{S}}(T)|\lesssim \nu^{\frac14}(\sqrt{T_0}-\sqrt{T})\lesssim 1.
	\end{aligned}
\end{equation*}

Similarly $|\pa_T^m\Theta_0^{\mathrm{S}}|\lesssim T^{-m}$ for all
$m\in\N_0$. With this choice the right-hand side of the integrated equation
decays exponentially as $Y\to+\infty$; setting $\ol V_1(T,0)=0$ and choosing
$\eta\in(0,\frac12)$, we obtain
\[
|\ol V_1(T,Y)|\lesssim e^{-(\frac12+\eta)\om(T)|Y|},
\]
and the estimates for $\pa_T^m\pa_Y^n\ol V_1$, $m,n\in\N_0$, follow in the
same way. Together with the explicit form of $V_1^{\mathrm{cut}}$, this
proves Proposition~\ref{prop3.6} for $j=1$.

\end{proof}

Combining Remark~\ref{rem3.4} with Proposition~\ref{prop3.6}, we obtain the following estimate:
\begin{corollary}\label{cor3.8}
	For all $m,J\in \N_0$, we have
	\begin{equation*}
		\begin{aligned}
		|\pa_T^m(\Theta-\Theta_{[J]}^{\mathrm{S}})|\lesssim T^{-2J-2-m}.
		\end{aligned}
	\end{equation*}
\end{corollary}

\subsection{Outer-Shock Matching}
For convenience, we switch to the inner variables and set
\begin{equation*}
	\begin{aligned}
		U_j^{\mathrm{shock}}(T,X)&=T^{\frac12-2j}V_j^{\mathrm{shock}}(T,\sqrt{T}X+\Theta(T))
		\\
		U_j^{\mathrm{cut}}(T,X)&=T^{\frac12-2j}V_j^{\mathrm{cut}}(T,\sqrt{T}X+\Theta(T)).
	\end{aligned}
\end{equation*}
We first rewrite the equation for $V_j^{\mathrm{shock}}$ given in Proposition~\ref{prop3.3} in these variables.
\begin{prop}\label{shockeq}
	For $j\in\N_0$ we have
	\begin{equation}\label{348}
		\begin{aligned}
		\pa_X^2U_{j}^{\mathrm{shock}}&+\dot{\mf{S}}\pa_XU_j^{\mathrm{shock}}=\frac{1}{2}\pa_X\sum\limits_{j_1+j_2=j}U_{j_1}^{\mathrm{shock}}U_{j_2}^{\mathrm{shock}}+\pa_TU_{j-1}^{\mathrm{shock}}\\
		&-\pa_T(\frac{\Theta}{T^{\frac12}})\pa_XU_{j-1}^{\mathrm{shock}}+\sum\limits_{j_1+j_2=j-1}\pa_T(\frac{\Theta_{j_1}^{\mathrm{S}}}{T^{2j_1+\frac12}})\pa_XU_{j_2}^{\mathrm{shock}}.
		\end{aligned}
	\end{equation}
\end{prop}

We next restate Proposition~\ref{prop3.6} in the variables $(T,X)$.
\begin{prop}\label{prop3.8}
	There exists a constant $\mu>0$ such that for all $j,m,n\in\N_0$ and $T\in[1,T_0]$ we have
	\begin{equation*}
		\begin{aligned}
			|\pa_T^m\pa_X^n(U_j^{\mathrm{shock}}-U_{j}^{\mathrm{cut}})|\lesssim T^{\frac12-2j-m+\frac n2}e^{-\mu\sqrt{T}|X|}.\\
		\end{aligned}
	\end{equation*}
\end{prop}
We introduce the rescaled outer profiles $U_k^{\mathrm{out}}$ by
\begin{equation}\label{3.60}
	\begin{aligned}
	u_k^{\mathrm{out}}(t,\x)=\nu^{\frac14-k}U_k^{\mathrm{out}}(\nu^{-1/2}t,\nu^{-3/4}\x).
	\end{aligned}
\end{equation}

Define the truncated expansions
\begin{equation*}
	\begin{aligned}
		U_{[J]}^{\mathrm{shock}}=\sum\limits_{j=0}^JU_j^{\mathrm{shock}},\quad U_{[K]}^{\mathrm{out}}=\sum\limits_{k=0}^KU_k^{\mathrm{out}},\quad 	U_{[J]}^{\mathrm{cut}}=\sum\limits_{j=0}^JU_j^{\mathrm{cut}}.
	\end{aligned}
\end{equation*}
We will show that $U_{[J]}^{\mathrm{shock}}$ and $U_{[K]}^{\mathrm{out}}$ match in $O_{\mathrm{out}}\cap O_{\mathrm{shock}}$ as $J,K\to\infty$.
\begin{prop}\label{prop3.9}
Suppose $T\in[1,T_0]$, $|X|\geq 2T^{-1/2}$,
$|X+T^{-\frac12}\Theta|\geq T^{-\frac12}$ and
$X(X+T^{-\frac12}\Theta)>0$. Then for all $m,n,J,K\in\N_0$, we have
	\begin{equation*}
		\begin{aligned}
			|\pa_T^m\pa_X^n(U_{[J]}^{\mathrm{shock}}-U_{[K]}^{\mathrm{out}})|\lesssim |T|^{\frac12-m+\frac n2}(e^{-\mu\sqrt{T}|X|}+(\frac{|X|}{T^{\frac32}})^{K+1}+(\frac{|X|}{T^{\frac32}})^{J+1}).
		\end{aligned}
	\end{equation*}
\end{prop}
To prove Proposition~\ref{prop3.9}, we need the following lemma. For convenience we set
\begin{equation}\label{th}
	\begin{aligned}
		\wh{\Theta}_j^{\mathrm{S}}=T^{-2j}\Theta_j^{\mathrm{S}},\quad \wh{\Theta}_{[J]}^{\mathrm{S}}=\sum\limits_{j=0}^JT^{-2j}\Theta_j^{\mathrm{S}}.
	\end{aligned}
\end{equation}
\begin{lemma}\label{lem3.10}
Suppose $T\in[1,T_0]$, $|X|\geq 2T^{-1/2}$,
$|X+T^{-\frac12}\Theta|\geq T^{-\frac12}$ and
$X(X+T^{-\frac12}\Theta)>0$. Then for all $J,m,n\in\N_0$ we have
		\begin{equation*}
		\begin{aligned}
			|\pa_T^m\pa_X^n(U_{[J]}^{\mathrm{cut}}-&\sum_{\substack{i+k\leq J}}	\frac{1}{i!}\pa_X^iU_k^{\mathrm{out}}(T,0)[X+T^{-\frac12}(\Theta-\wh{\Theta}_{[J-i-k]}^{\mathrm{S}})]^i)|\\
			&\lesssim T^{-\frac 32-m+\frac n2}(T^{-2J}+(\frac{|X|}{T^{\frac32}})^J).
		\end{aligned}
	\end{equation*}
\end{lemma}
\begin{proof}
The profiles $U_j^{\mathrm{cut}}$ serve as intermediate objects connecting the outer expansion with the shock-profile expansion, so we compute them explicitly. 
From \eqref{3.27} a direct computation gives
	\begin{equation*}
	\begin{aligned}
	U_j^{\mathrm{cut}}&=\chi(\sqrt{T}X+\Theta)\sum\limits_{p=0}^j\sum_{\substack{i\geq p,k\geq 0\\ i+k+h_1+...+h_p=j}}\frac{(-1)^p}{i!}C_i^p(\sqrt{T}X+\Theta)^{i-p}\pa_X^iU_k^{\mathrm{out}}(T,0)\\
	&\qquad \times T^{2k+\frac 32i-2j}\prod_{r=1}^p\Theta_{h_r}^{\mathrm{S}}.
	\end{aligned}
\end{equation*}
To reorganize the expression for $U_j^{\mathrm{cut}}$, we expand $(\sqrt{T}X+\Theta)^{i-p}$ by means of the binomial theorem:
\[
(\sqrt{T}X+\Theta)^{i-p}=\sum_{q=0}^{i-p}C_{i-p}^{q}\,T^{q/2}X^{q}\,\Theta^{\,i-p-q}.
\]

Substituting \eqref{th} into the formula for $U_j^{\mathrm{cut}}$, we obtain
\begin{align*}
U_j^{\mathrm{cut}}&=\chi(\sqrt{T}X+\Theta)\sum_{p=0}^{j}\sum_{\substack{i\geq p,k\geq0\\\sum h_r=j-i-k}}
\frac{(-1)^p}{i!}C_i^{p}
\Bigl(\sum_{q=0}^{i-p}C_{i-p}^{q}T^{q/2}X^{q}\Theta^{\,i-p-q}\Bigr)\notag\displaybreak[1]\\
&\qquad\times\pa_X^iU_k^{\mathrm{out}}(T,0)\,T^{2k+\frac32i-2j}\prod_{r=1}^{p}\Theta_{h_r}^{\mathrm{S}}\notag\displaybreak[1]\\
&=\chi(\sqrt{T}X+\Theta)\sum_{p=0}^{j}\sum_{\substack{k\geq0\\0\leq q\leq i-p\\\sum h_r=j-i-k}}
\frac{(-1)^p}{p!q!(i-p-q)!}T^{\frac q2-\frac i2}X^{q}\pa_X^iU_k^{\mathrm{out}}(T,0) \times\Theta^{\,i-p-q}\prod_{r=1}^{p}\wh{\Theta}_{h_r}^{\mathrm{S}},
\end{align*}
where we used the identity $C_i^{p}C_{i-p}^{q}=i!/(p!\,q!\,(i-p-q)!)$. We also define the truncated cutoff profiles $U_{[J]}^{\mathrm{cut}}$ and $\wh{\Theta}_{[J]}^{\mathrm{S}}$ by
\begin{equation}\label{3.61}
	\begin{aligned}
		\wh{\Theta}_{[J]}^{\mathrm{S}}=\sum\limits_{j=0}^J\wh{\Theta}_j^{\mathrm{S}}.
	\end{aligned}
\end{equation}
Notice that the sums can be reorganized as follows:
\begin{equation*}
	\begin{aligned}
	&\sum\limits_{j=0}^J\sum_{p=0}^{j}\sum_{\substack{k\geq0\\0\leq q\leq i-p\\i+k+\sum h_r=j}}=\sum_{p=0}^{J}\sum\limits_{j=p}^J\sum_{\substack{k\geq0\\0\leq q\leq i-p\\i+k+\sum h_r=j}}
	=\sum_{q=0}^{J}\sum_{\substack{k\geq0\\0\leq p\leq i-q\\i+k+\sum h_r\leq J}}.\\
	\end{aligned}
\end{equation*}
We next record the combinatorial identity (valid for $0\leq q\leq i$, $0\leq p\leq i-q$):
\[
\frac{1}{p!\,q!\,(i-p-q)!}=\frac{C_i^{q}}{i!}\,C_{i-q}^{p},
\]
since $C_i^{q}=i!/(q!(i-q)!)$ and $C_{i-q}^{p}=(i-q)!/(p!(i-p-q)!)$. Summing over $p$ and using the above identity, we obtain
\begin{align*}
U_{[J]}^{\mathrm{cut}}&=\chi(\sqrt{T}X+\Theta)\sum_{q=0}^{J}\sum_{\substack{k\geq0\\0\leq p\leq i-q\\i+k+\sum h_r\leq J}}
\frac{(-1)^p}{p!q!(i-p-q)!}T^{\frac q2-\frac i2}X^{q}
\pa_X^iU_k^{\mathrm{out}}(T,0)\,\notag\displaybreak[1]\\
&\qquad \times\Theta^{\,i-p-q}\prod_{r=1}^{p}\wh{\Theta}_{h_r}^{\mathrm{S}}\notag\displaybreak[1]\\
&=\chi(\sqrt{T}X+\Theta)\sum_{q=0}^{J}\sum_{\substack{k\geq0\\0\leq p\leq i-q\\i+k+\sum h_r\leq J}}
\frac{1}{i!}C_i^qT^{\frac q2-\frac i2}X^{q}
\pa_X^iU_k^{\mathrm{out}}(T,0)C_{i-q}^p(-1)^p  \times\Theta^{\,i-p-q}\prod_{r=1}^{p}\wh{\Theta}_{h_r}^{\mathrm{S}}.
\label{U_partial}
\end{align*}
We note the following splitting of the binomial expansion:
\begin{equation}
	\begin{aligned}
	&\sum_{\substack{0\leq p\leq i-q\\\sum h_r\leq J-i-k}}C_{i-q}^p(-1)^p
	\Theta^{\,i-p-q}\prod_{r=1}^{p}\wh{\Theta}_{h_r}^{\mathrm{S}}\\
	&=(\Theta-\wh{\Theta}_{[J-i-k]}^{\mathrm{S}})^{i-q}-\sum_{\substack{0\leq p\leq i-q\\\sum h_r>J-i-k\\h_r\leq J-i-k}}C_{i-q}^p(-1)^p
	\Theta^{\,i-p-q}\prod_{r=1}^{p}\wh{\Theta}_{h_r}^{\mathrm{S}}.\\
	\end{aligned}
\end{equation}
Substituting this identity into \eqref{U_partial}, we obtain
\begin{align*}
U_{[J]}^{\mathrm{cut}}
		&=\chi(\sqrt{T}X+\Theta)\sum_{q=0}^{J}[\sum_{\substack{k\geq0\\0\leq i-q\\i+k\leq J}}
		\frac{1}{i!}C_i^qT^{\frac q2-\frac i2}X^{q}
		\pa_X^iU_k^{\mathrm{out}}(T,0)
		(\Theta-\wh{\Theta}_{[J-i-k]}^{\mathrm{S}})^{i-q}\notag\displaybreak[1]\\
		&\quad-\sum_{\substack{k\geq0\\0\leq p\leq i-q\\\sum h_r> J-i-k\\h_r\leq J-i-k}}
		\frac{1}{i!}C_i^qT^{\frac q2-\frac i2}X^{q}
		\pa_X^iU_k^{\mathrm{out}}(T,0)C_{i-q}^p(-1)^p
		\Theta^{\,i-p-q}\prod_{r=1}^{p}\wh{\Theta}_{h_r}^{\mathrm{S}}]\notag\\*
		&:=I_1+I_2.
\end{align*}
We now simplify $I_1$ and estimate $I_2$:
\begin{align*}
I_1&=\sum_{\substack{i+k\leq J}}	\frac{1}{i!}T^{-\frac i2}\pa_X^iU_k^{\mathrm{out}}(T,0)\sum_{\substack{0\leq q\leq i}}C_i^qT^{\frac q2}X^q(\Theta-\wh{\Theta}_{[J-i-k]}^{\mathrm{S}})^{i-q}\notag\displaybreak[1]\\
	&=\sum_{\substack{i+k\leq J}}	\frac{1}{i!}\pa_X^iU_k^{\mathrm{out}}(T,0)[X+T^{-\frac12}(\Theta-\wh{\Theta}_{[J-i-k]}^{\mathrm{S}})]^i,\notag\displaybreak[1]\\
	|I_2|&\lesssim \sum\limits_{0\leq q\leq J} T^{-\frac32 -2J}(\sqrt{T}X)^q\lesssim T^{-\frac32 -2J}+ T^{-\frac32}(\frac{X}{T^{\frac32}})^J.
\end{align*}
The identity for $I_1$ produces exactly the main term in the statement of Lemma~\ref{lem3.10}, while the estimate for $I_2$ bounds the remainder. Combining the two, we obtain the desired bound for $m=n=0$; the estimates for general $m,n$ follow by differentiating the identities above with respect to $T$ and $X$. This completes the proof of the lemma.
\end{proof}

\begin{proof}[Proof of Proposition~\ref{prop3.9}]
	Fix $J,K\in\N_0$. We split the difference into three parts:
	\begin{equation*}
		\begin{aligned}
			|U_{[J]}^{\mathrm{shock}}-U_{[K]}^{\mathrm{out}}|&\leq|U_{[J]}^{\mathrm{shock}}-U_{[J]}^{\mathrm{cut}}|+|U_{[K]}^{\mathrm{out}}-U_{[K]}^{\mathrm{cut}}|+|U_{[J]}^{\mathrm{cut}}-U_{[K]}^{\mathrm{cut}}|\\
			&:=I_1+I_2+I_3.
		\end{aligned}
	\end{equation*}
	For $I_1$, Proposition~\ref{prop3.8} gives
	\begin{equation*}
		\begin{aligned}
			|I_1|\lesssim |T|^{\frac12}e^{-\mu\sqrt{T}|X|}.
		\end{aligned}
	\end{equation*}
For $I_2$, combining the Taylor expansion of $U_k^{\mathrm{out}}$ at $X=0$ with Proposition~\ref{prop3.2}, \eqref{inner}, \eqref{3.60}, and Lemma~\ref{lem3.10}, we obtain
	\begin{align*}
|I_2|\leq& \sum\limits_{i+k\leq K}|\frac{1}{i!}\pa_X^iU_k^{\mathrm{out}}(T,0)[X^i-[X+T^{-\frac12}(\Theta-\wh{\Theta}_{[K-i-k]}^{\mathrm{S}})]^i]|\notag\displaybreak[1]\\
		&+\sum\limits_{i+k= K+1}\frac{1}{i!}|\pa_X^iU_k^{\mathrm{out}}(T,\xi_k)||X|^i\notag\displaybreak[1]\\
		&+|U_{[K]}^{\mathrm{cut}}-\sum_{\substack{i+k\leq K}}	\frac{1}{i!}\pa_X^iU_k^{\mathrm{out}}(T,0)[X+T^{-\frac12}(\Theta-\wh{\Theta}_{[K-i-k]}^{\mathrm{S}})]^i|\notag\\*
		&\lesssim T^{\frac12}(\frac{|X|}{T^{\frac32}})^{K+1},
\end{align*}
	where $\xi_k$ lies between $0$ and $X$. For $I_3$, assuming $J<K$ without loss of generality, we have
	\begin{equation*}
		\begin{aligned}
			|I_3|&\lesssim \sum\limits_{J+1\leq i+k\leq K} |X^{i}\pa_X^{i}U_k^{\mathrm{out}}(T,0)|
			\lesssim T^{\frac12}((\frac{|X|}{T^{\frac32}})^{J+1}+(\frac{|X|}{T^{\frac32}})^{K+1}).
		\end{aligned}
	\end{equation*}

	Combining the estimates for $I_1$, $I_2$, and $I_3$, we obtain the desired bound for $m=n=0$; the general case follows by differentiating with respect to $T$ and $X$. This completes the proof of Proposition~\ref{prop3.9}.
	
\end{proof}

\section{Inner Expansion and Matching of Two Grids}
In this section we introduce the inner expansion and the two intermediate
grids that connect it with the outer and shock-profile expansions of
Section~4, and we establish the corresponding matching estimates. We work
in the inner variables \eqref{inner}; the grid functions are constructed
and matched on the interval $[1,2\nu^{-\beta_2}]$ for some small
$\beta_2>0$. Section~5.1 introduces the inner expansion
$V\sim\sum_{\ell}\nu^{\ell/4}V_\ell^{\mathrm{in}}$ and the rescaled
profiles $U_\ell^{\mathrm{in}}$, and derives formally the equations they
satisfy; the modified inner equations, on which the actual construction of
the profiles is based, are stated in Section~6.
Sections~5.2 and 5.3 introduce the outer--inner and shock--inner grid
functions, which match the inner expansion to the outer expansion as
$X\to\infty$ and to the shock-profile expansion as $T\to\infty$,
respectively. Section~5.4 constructs the shock--inner grid functions with
their prescribed asymptotic behavior and proves the shock-grid and
outer--shock matching estimates; the finite sums of these grids yield the
approximate inner profile $U_{[J],[K],\ell}^{\mathrm{SOI}}$ used in the
error analysis of Section~6.
\subsection{The inner expansion}
We make the following inner expansion:

\begin{equation}\label{4.1}
	\begin{aligned}
	V\sim \sum\limits_{\ell\geq 0}\nu^{\frac \ell4}V_\ell^{\mathrm{in}},\qquad \Theta\sim \sum\limits_{\ell\geq 0}\nu^{\frac \ell4}\Theta_{\ell}^{\mathrm{I}},\qquad \frac{\dot{\mf{S}}}{\sqrt{T}}\sim \sum\limits_{\ell\geq 0}\nu^{\frac {\ell+1}4} \mf{S}_\ell.
	\end{aligned}
\end{equation}
Substituting this expansion into \eqref{3.9} and collecting terms of like order in $\nu^{\frac \ell4}$, we obtain
\begin{equation*}
	\begin{aligned}
		\frac1T\pa_TV_\ell^{\mathrm{in}}+\frac{Y}{2T^2}\pa_YV_\ell^{\mathrm{in}}-\frac 1{2T^2}\sum\limits_{\ell_1+\ell_2=\ell}\Theta_{\ell_1}^{\mathrm{I}}\pa_YV_{\ell_2}^{\mathrm{in}}+\frac 1{T}\sum\limits_{\ell_1+\ell_2=\ell}\dot{\Theta}_{\ell_1}^{\mathrm{I}}\pa_YV_{\ell_2}^{\mathrm{in}}\\
		-\sum\limits_{\ell_1+\ell_2=\ell-1}{\mf{S}}_{\ell_1}\pa_YV_{\ell_2}^{\mathrm{in}}+\sum\limits_{\ell_1+\ell_2=\ell}V_{\ell_1}^{\mathrm{in}}\pa_YV_{\ell_2}^{\mathrm{in}}+\frac{1}{2T^2}V_\ell^{\mathrm{in}}=\pa_Y^2V_\ell^{\mathrm{in}}.
	\end{aligned}
\end{equation*}
Returning to the $(T,X)$ variables, we seek the expansion
\begin{equation*}
	\begin{aligned}
	U^{(\nu)}\sim \sum\limits_{\ell\geq 0}\nu^{\frac{\ell}4}U_\ell^{\mathrm{in}}(T,X).
	\end{aligned}
\end{equation*}
Substituting this expansion, together with the expansion of
$\dot{\mf S}/\sqrt T$ in \eqref{4.1}, into \eqref{3.7} and collecting
terms of order $\nu^{\ell/4}$ yields, for $\ell\in\N_0$,
	\begin{equation}\label{4.4}
		\begin{aligned}
		\pa_TU_\ell^{\mathrm{in}}
		&-T^{\frac12}\sum_{\ell_1+\ell_2=\ell-1}
		\mf{S}_{\ell_1}\pa_XU_{\ell_2}^{\mathrm{in}}+\frac12\pa_X\sum_{\ell_1+\ell_2=\ell}
		U_{\ell_1}^{\mathrm{in}}U_{\ell_2}^{\mathrm{in}}
		=\pa_X^2U_\ell^{\mathrm{in}}.
		\end{aligned}
	\end{equation}

Since the inner expansion must be matched with the outer expansion as $X\to\infty$ and with the shock-profile expansion as $T\to\infty$, we introduce two distinct grids below.
\subsection{The outer--inner grid}
We first introduce the standard outer--inner grid functions ${u}_{k,\ell}^{\mathrm{OI}}$.

For each $k\in\N_0$ we assume the expansion
\begin{equation*}
	\begin{aligned}
		u_k^{\mathrm{out}}\sim \sum\limits_{\ell\geq 0} {u}_{k,\ell}^{\mathrm{OI}}.
	\end{aligned}
\end{equation*}
Substituting this expansion together with \eqref{2.60} into \eqref{3.3}, we find that the grid functions satisfy
\begin{equation}\label{4.3}
	\begin{aligned}
		\np_t{u}_{k,\ell}^{\mathrm{OI}}=\sum\limits_{\ell_1+\ell_2=\ell-1}\dot{\mf{s}}_{\ell_1}\np_x{u}_{k,\ell_2}^{\mathrm{OI}}-\frac12\np_x\sum\limits_{\substack{k_1+k_2=k\\\ell_1+\ell_2=\ell}}{u}_{k_1,\ell_1}^{\mathrm{OI}}{u}_{k_2,\ell_2}^{\mathrm{OI}}+\np_x^2{u}_{k-1,\ell}^{\mathrm{OI}}.
	\end{aligned}
\end{equation}
We next introduce the rescaled grid functions
\begin{equation*}
	\begin{aligned}
		{U}_{k,\ell}^{\mathrm{OI}}(T,X)=\nu^{k-\frac{\ell+1}{4}}{u}_{k,\ell}^{\mathrm{OI}}(\nu^{1/2}T,\nu^{3/4}X).
	\end{aligned}
\end{equation*}
Accordingly, we expect the following expansion:
\begin{equation*}
	\begin{aligned}
	{U}_k^{\mathrm{out}}\sim \sum\limits_{\ell\geq 0}\nu^{\frac \ell4}{U}_{k,\ell}^{\mathrm{OI}}.
	\end{aligned}
\end{equation*}
From \eqref{4.3} we deduce
\begin{equation}\label{4.5}
	\begin{aligned}
		\pa_T{U}_{k,\ell}^{\mathrm{OI}}=T^{\frac12}\sum\limits_{\ell_1+\ell_2=\ell-1}\mf{S}_{\ell_1}\pa_X{U}_{k,\ell_2}^{\mathrm{OI}}-\frac12\pa_X\sum\limits_{\substack{k_1+k_2=k\\\ell_1+\ell_2=\ell}}{U}_{k_1,\ell_1}^{\mathrm{OI}}{U}_{k_2,\ell_2}^{\mathrm{OI}}+\pa_X^2{U}_{k-1,\ell}^{\mathrm{OI}}.
	\end{aligned}
\end{equation}
Define the truncated sums
\begin{equation*}
	\begin{aligned}
		{U}_{k,[L]}^{\mathrm{OI}}=\sum\limits_{\ell= 0}^L\nu^{\frac \ell4} U_{k,\ell}^{\mathrm{OI}}.
	\end{aligned}
\end{equation*}
We shall use the following estimates, which are due to \cite{CG}:
\begin{prop}\label{prop4.2}
	For all $k,\ell,m,n\in\N_0$ and $(T,X)\in [0,T_0]\times (\R\setminus\{0\})$, we have
	\begin{equation*}
		\begin{aligned}
			|\pa_T^m\pa_X^n{U}_{k,\ell}^{\mathrm{OI}}|\lesssim \mf{D}^{\ell+1-4k-2m-3n}.
		\end{aligned}
	\end{equation*}
\end{prop}
\begin{prop}\label{prop4.3}
	For all $k,L,m,n\in\N_0$ and $(T,X)\in [0,T_0]\times (\R\setminus\{0\})$, we have
	\begin{equation*}
		\begin{aligned}
			|\pa_T^m\pa_X^n(U_k^{\mathrm{out}}-{U}_{k,[L]}^{\mathrm{OI}})|\lesssim \nu^{\frac{L+1}{4}}\mf{D}^{L+2-4k-2m-3n}.
		\end{aligned}
	\end{equation*}
\end{prop}
Although these estimates are stated for $0\leq T\leq T_0$, the grid
functions are used in the construction only on $[1,2\nu^{-\beta_2}]$; the
short interval $T\in[0,1]$ is covered by the initial data and by the
estimate of Lemma~\ref{lem:52} in Section~6.

\subsection{The shock--inner grid}
For the shock-profile expansion, we assume that for each $j\in\N_0$,
\begin{equation}\label{4.7}
	\begin{aligned}
	V_j^{\mathrm{shock}}&\sim \sum\limits_{\ell\geq 0}\nu^{\frac \ell4}V_{j,\ell}^{\mathrm{SI}},\qquad	\Theta_j^{\mathrm{S}}&\sim \sum\limits_{\ell\geq 0}\nu^{\frac \ell4}\Theta_{j,\ell}.\qquad 
	\end{aligned}
\end{equation}
Substituting \eqref{4.7} into Proposition~\ref{prop3.3}, we obtain the following equations for $V_{j,\ell}^{\mathrm{SI}}$.
\begin{prop}\label{prop4.4}
	For all $j,\ell\in\N_0$ we have
	\begin{equation}\label{4.8}
		\begin{aligned}
			\pa_Y^2V_{j,\ell}^{\mathrm{SI}}&=\frac12\pa_Y\sum_{\substack{j_1+j_2=j\\\ell_1+\ell_2=\ell}}V_{j_1,\ell_1}^{\mathrm{SI}}V_{j_2,\ell_2}^{\mathrm{SI}}-\sum\limits_{\ell_1+\ell_2=\ell-1}\mf{S}_{\ell_1}\pa_YV_{j,\ell_2}^{\mathrm{SI}}+T\pa_TV_{j-1,\ell}^{\mathrm{SI}}+\frac Y2\pa_YV_{j-1,\ell}^{\mathrm{SI}}\\
		&-(2j-\frac52)V_{j-1,\ell}^{\mathrm{SI}}+\sum_{\substack{j_1+j_2=j-1\\\ell_1+\ell_2=\ell}}[T\pa_T\Theta_{j_1,\ell_1}-(2j_1+\frac 12)\Theta_{j_1,\ell_1}]\pa_YV_{j_2,\ell_2}^{\mathrm{SI}}.
		\end{aligned}
	\end{equation}
	Moreover, for $j=0,\ell\in\N_0$ we have
		\begin{equation}\label{0si}
		\begin{aligned}
			\pa_Y^2V_{0,\ell}^{\mathrm{SI}}&=\frac12\pa_Y\sum_{\substack{\ell_1+\ell_2=\ell}}V_{0,\ell_1}^{\mathrm{SI}}V_{0,\ell_2}^{\mathrm{SI}}-\sum\limits_{\ell_1+\ell_2=\ell-1}\mf{S}_{\ell_1}\pa_YV_{0,\ell_2}^{\mathrm{SI}}.\\
		\end{aligned}
	\end{equation}
\end{prop}
As in Section~4, we prescribe the boundary conditions for \eqref{4.8}. To this end, we introduce
\begin{equation}\label{4.15}
	\begin{aligned}
	V_{i,k,\ell}^{\pm}(T)=\frac{1}{i!} \pa_X^i{U}_{k,\ell}^{\mathrm{OI}}(T,0^{\pm})T^{-\frac{1-4k-3i}{2}}.
	\end{aligned}
\end{equation}
In analogy with \eqref{3.17}, we impose the matching conditions
\begin{equation}\label{4.11}
	\begin{aligned}
		V_{j,\ell}^{\mathrm{SI}}\sim \sum\limits_{\substack{0\leq p\leq i,k\geq 0\\ j_1,...,j_p\geq 0\\i+k+j_1+...+j_p=j\\\ell_0+\ell_1+...+\ell_p=\ell}}(-1)^{p}C_{i}^{p}Y^{i-p}{V}_{i,k,\ell_0}^{\pm}(T)\prod_{r=1}^p\Theta_{j_r,\ell_r} \qquad \text {as }Y\to \pm \infty.
	\end{aligned}
 \end{equation}

\subsection{Shock-grid estimate and matching}
In this subsection we construct the solution of \eqref{4.8} with the prescribed asymptotic behavior and prove the corresponding estimates. To this end, we introduce
\begin{equation*}
	V_{j,\ell}^{\mathrm{cut}}(T,Y)=\chi(Y)\sum\limits_{\substack{0\leq p\leq i,k\geq 0\\ j_1,...,j_p\geq 0\\i+k+j_1+...+j_p=j\\\ell_0+\ell_1+...+\ell_p=\ell}}\left\{\begin{aligned}
		&(-1)^{p}C_{i}^{p}Y^{i-p}{V}_{i,k,\ell_0}^{-}\prod_{r=1}^p\Theta_{j_r,\ell_r},\quad Y<0\\
		&(-1)^{p}C_{i}^{p}Y^{i-p}{V}_{i,k,\ell_0}^{+}\prod_{r=1}^p\Theta_{j_r,\ell_r},\quad Y\geq0.\\
	\end{aligned}
	\right.
\end{equation*}
We shall prove the following estimates.
\begin{prop}\label{prop4.5} Let $\beta_2\in (0,\frac1{100})$.
	There exists a constant $\mu>0$ such that for all $T\in[\frac12,2\nu^{-\beta_2}]$ and $j,\ell,m,n\in\N_0$ we have
	\begin{equation}\label{4.14}
		\begin{aligned}
			|\pa_T^m\pa_Y^n(V_{j,\ell}^{\mathrm{SI}}-V_{j,\ell}^{\mathrm{cut}})|\lesssim \nu^{-\beta_2\ell/2}T^{-m}e^{-\mu|Y|}\\
			|\pa_T^m\Theta_{j,\ell}|\lesssim \nu^{-\beta_2\ell/2}T^{-m}.
		\end{aligned}
	\end{equation}
\end{prop}
The proof of the estimates for $V_{j,\ell}^{\mathrm{SI}}$ and $\Theta_{j-1,\ell}$ for $j\in \N$, $\ell\in \N_0$ is analogous to that of Proposition~\ref{prop3.6}. The proof of these estimates is given in Appendix B. For $V_{0,\ell}^{\mathrm{SI}}$, the same argument may not apply due to the absence of $\Theta$ terms in \eqref{4.8}. Here we give a sketch of the proof for $V_{0,\ell}^{\mathrm{SI}}$ to overcome this difficulty. 
From Proposition~\ref{prop2.16}, \eqref{vik}, \eqref{4.1}, Proposition~\ref{prop4.3} and \eqref{4.15}, we obtain for any $m,L\in \N_0$,
\begin{equation}\label{4.19}
\begin{aligned}
	\pa_T^mV_{0,0}^{\pm}(T)&=\pa_T^m\sum\limits_{\ell=0}^{L}\nu^{\frac{\ell}{4}}V_{0,0,\ell}^{\pm}(T)+O(\nu^{\frac{L+1}{4}}T^{\frac{L+1}{2}-m})\\
	\pa_T^m(\frac{\dot{\mf{S}}}{\sqrt{T}})(T)&=\pa_T^m\sum\limits_{\ell=0}^{L}\nu^{\frac{\ell+1}{4}}\mf{S}_{\ell}(T)+O(\nu^{\frac{L+2}{4}}T^{\frac{L+2}{2}-m}).\\
	\end{aligned}
\end{equation}
Using \eqref{3.19},\eqref{3.21}, \eqref{4.7} and \eqref{4.19}, we can get the desired result for $V_{0,\ell}^{\mathrm{SI}}$.\par
Define the truncated expansions
\begin{equation}\label{4.32}
	\begin{aligned}
		V_{j,[L]}^{\mathrm{SI}}=\sum\limits_{\ell=0}^L \nu^{\frac \ell4}V_{j,\ell}^{\mathrm{SI}},\qquad \Theta_{j,[L]}=\sum\limits_{\ell=0}^L \nu^{\frac \ell4}\Theta_{j,\ell},\quad \mf{S}_{[L]}=\sum\limits_{\ell=0}^{L-1}\nu^{\frac{\ell+1}4}\mf{S}_\ell
	\end{aligned}
\end{equation}
and the corresponding remainders

\begin{equation}\label{4.33}
\begin{aligned}
	Q_{j,L}=V_j^{\mathrm{shock}}-V_{j,[L]}^{\mathrm{SI}},\quad 
	Q_{j,L}^{\mathrm{cut}}=V_j^{\mathrm{cut}}-V_{j,[L]}^{\mathrm{cut}}\\
	\ol{\mf{S}}_L=\frac{\dot{\mf{S}}}{\sqrt{T}}-\mf{S}_{[L]},\quad \ol{\Theta}_{j,L}={\Theta_{j}^{\mathrm{S}}-\Theta_{j,[L]}}.
\end{aligned}
\end{equation}
\begin{prop}\label{prop4.7}
	Fix $\beta_2\in(0,\frac1{100}) $. There exists $\mu>0$ such that, in the region $\{(T,X):T\in[1,2\nu^{-\beta_2}]\}$, the following estimates hold for all $j,L,m,n\in\N_0$:
	\begin{equation}\label{4.34}
		\begin{aligned}
			|\pa_T^m\pa_Y^n(Q_{j,L}-Q_{j,L}^{\mathrm{cut}})|&\lesssim \nu^{(L+1)(\frac14-\frac{\beta_2}{2})}T^{-m}e^{-\mu|Y|}\\
			|\pa_T^{m}({\Theta_{j}^{\mathrm{S}}-\Theta_{j,[L]}})|&\lesssim \nu^{(L+1)(\frac14-\frac{\beta_2}{2})}T^{-m}.
		\end{aligned}
	\end{equation}
\end{prop}

\begin{proof}
	We first treat $Q_{0,L}$. From Proposition~\ref{prop4.4}, we have
	\begin{equation*}
		\begin{aligned}
		&\pa_Y^2V_{0,[L]}^{\mathrm{SI}}-V_{0,[L]}^{\mathrm{SI}}\pa_YV_{0,[L]}^{\mathrm{SI}}+\frac{\dot{\mf{S}}}{\sqrt{T}}\pa_YV_{0,[L]}^{\mathrm{SI}}\\
		&=-\frac12\pa_Y\sum_{\substack{\ell_1,\ell_2\leq L\\\ell_1+\ell_2>L}}\nu^{\frac{\ell_1+\ell_2}4}V_{0,\ell_1}^{\mathrm{SI}}V_{0,\ell_2}^{\mathrm{SI}} +\sum\limits_{\ell=0}^{L}\nu^{\frac{\ell}4}\ol{\mf{S}}_{L-\ell}\pa_YV_{0,\ell}^{\mathrm{SI}}\\
		&:=H_{0,L}.
		\end{aligned}
	\end{equation*}
We have
	\begin{equation*}
		\begin{aligned}
		|H_{0,L}|\lesssim \nu^{(L+1)(\frac14-\frac{\beta_2}{2})}.
		\end{aligned}
	\end{equation*}
We split $V_{0,[L]}^{\mathrm{SI}}=V_{1}+V_2$, where
		\begin{equation*}
		\begin{aligned}
			\pa_Y^2V_1-V_1\pa_YV_1&+\mf{S}_{[L]}\pa_YV_1=0,
		\end{aligned}
	\end{equation*}
with $V_1(T,Y)\to T^{-\frac12}U_{0,[L]}^{\mathrm{OI}}(T,0^{\pm})$ as $Y\to \pm\infty$.
We take $V_1$ centered as in \eqref{3.21}. The Taylor construction gives
$V_{0,[L]}^{\mathrm{SI}}(T,0)=\mf S_{[L]}=V_1(T,0)$, hence $V_2(T,0)=0$.
By Proposition~\ref{prop4.3} and the proof of \eqref{3.21}, we obtain
	\begin{equation}\label{8.41}
		\begin{aligned}
		|V_1-V_0^{\mathrm{shock}}|\lesssim \nu^{(L+1)(\frac14-\frac{\beta_2}{2})}.
		\end{aligned}
	\end{equation}
Then $V_2$ satisfies
		\begin{equation*}
		\begin{aligned}
		\pa_Y^2V_2-\pa_Y(V_1V_2)+\frac{\dot{\mf{S}}}{\sqrt{T}}\pa_YV_2-V_2\pa_YV_2=H_{0,L}-\ol{\mf{S}}_L\pa_YV_1,
		\end{aligned}
	\end{equation*}
with $V_2(T,Y)\to 0$ as $Y\to\pm\infty$.
The parameter Taylor remainder in \eqref{3.21}, using \eqref{4.19}, gives
		\begin{equation}\label{8.43}
		\begin{aligned}
		|V_2|\lesssim \nu^{(L+1)(\frac14-\frac{\beta_2}{2})}.
		\end{aligned}
	\end{equation}
Combining \eqref{8.41} and \eqref{8.43}, we obtain
		\begin{equation*}
		\begin{aligned}
			|V_0^{\mathrm{shock}}-V_{0,[L]}^{\mathrm{SI}}|\lesssim \nu^{(L+1)(\frac14-\frac{\beta_2}{2})}.
		\end{aligned}
	\end{equation*}
Applying the differentiated parameter Taylor formula to
$V_0^{\mathrm{shock}}-V_0^{\mathrm{cut}}$ gives the first estimate in
\eqref{4.34} for $j=0$: by \eqref{4.22}, polynomial factors in $|Y|$
from parameter derivatives are absorbed by a smaller exponential rate,
and \eqref{4.19} supplies the factor $T^{-m}$ after $T$-differentiation.\par
Now suppose \eqref{4.34} holds for $Q_{<j,L}$ and $\Theta_{<j-1}^{\mathrm{S}}-\Theta_{<j-1,[L]}$. To estimate $Q_{j,L}$ and $\Theta_{j-1}^{\mathrm{S}}-\Theta_{j-1,[L]}$, we sum \eqref{4.8} over $\ell=0,\dots,L$; this yields the following equation for $V_{j,[L]}^{\mathrm{SI}}:$
	\begin{align*}
 \pa_Y^2V_{j,[L]}^{\mathrm{SI}}=&\frac12\pa_Y\sum_{j_1+j_2=j}V_{j_1,[L]}^{\mathrm{SI}}V_{j_2,[L]}^{\mathrm{SI}}-\frac{\dot{\mf{S}}}{\sqrt{T}}\pa_YV_{j,[L]}^{\mathrm{SI}}-(2j-\frac52)V_{j-1,[L]}^{\mathrm{SI}}\notag\\
		&\quad+\frac Y2\pa_YV_{j-1,[L]}^{\mathrm{SI}}+T\pa_TV_{j-1,[L]}^{\mathrm{SI}}+\sum_{\substack{j_1+j_2=j-1}}\bigl[T\pa_T\Theta_{j_1}^{\mathrm{S}}-(2j_1+\tfrac12)\Theta_{j_1}^{\mathrm{S}}\bigr]\pa_YV_{j_2,[L]}^{\mathrm{SI}}\notag\displaybreak[1]\\
		&\quad-\frac12\pa_Y\sum_{j_1+j_2=j}\ \sum_{\substack{\ell_1,\ell_2\leq L\\\ell_1+\ell_2>L}}\nu^{\frac{\ell_1+\ell_2}4}V_{j_1,\ell_1}^{\mathrm{SI}}V_{j_2,\ell_2}^{\mathrm{SI}}+\ol{\mf{S}}_{L}\pa_YV_{j,[L]}^{\mathrm{SI}}\notag\displaybreak[1]\\
		&\quad+\sum_{\substack{0\leq\ell_1\leq L-1\\0\leq\ell_2\leq L\\\ell_1+\ell_2+1>L}}\nu^{\frac{\ell_1+\ell_2+1}4}\mf{S}_{\ell_1}\pa_YV_{j,\ell_2}^{\mathrm{SI}} -\sum_{\substack{j_1+j_2=j-1}}\bigl[T\pa_T\ol{\Theta}_{j_1,L}-(2j_1+\tfrac12)\ol{\Theta}_{j_1,L}\bigr]\pa_YV_{j_2,[L]}^{\mathrm{SI}}\notag\\*
		&\quad-\sum_{\substack{j_1+j_2=j-1}}\ \sum_{\substack{\ell_1,\ell_2\leq L\\\ell_1+\ell_2>L}}\nu^{\frac{\ell_1+\ell_2}4}\bigl[T\pa_T\Theta_{j_1,\ell_1}-(2j_1+\tfrac12)\Theta_{j_1,\ell_1}\bigr]\pa_YV_{j_2,\ell_2}^{\mathrm{SI}}.
\end{align*}
Subtracting the preceding equation from Proposition~\ref{prop3.3} gives, for $j\geq1$,
\begin{align}
&\pa_Y^2Q_{j,L}-\pa_Y(V_0^{\mathrm{shock}}Q_{j,L})
+\frac{\dot{\mf{S}}}{\sqrt{T}}\pa_YQ_{j,L}\notag\\*
=&\pa_Y(Q_{0,L}V_{j,[L]}^{\mathrm{SI}}) +\pa_Y\sum_{\substack{j_1+j_2=j\\j_1,j_2>0}}V_{j_1,[L]}^{\mathrm{SI}}Q_{j_2,L}
+\frac12\pa_Y\sum_{\substack{j_1+j_2=j\\j_1,j_2>0}}Q_{j_1,L}Q_{j_2,L}\notag\displaybreak[1]\\
&-(2j-\tfrac52)Q_{j-1,L}+\frac Y2\pa_YQ_{j-1,L}+T\pa_TQ_{j-1,L} +\sum_{j_1+j_2=j-1}\bigl[T\pa_T\Theta_{j_1}^{\mathrm{S}}
-(2j_1+\tfrac12)\Theta_{j_1}^{\mathrm{S}}\bigr]\pa_YQ_{j_2,L}\notag\displaybreak[1]\\
&+\frac12\pa_Y\sum_{j_1+j_2=j}\ \sum_{\substack{\ell_1,\ell_2\leq L\\\ell_1+\ell_2>L}}
\nu^{\frac{\ell_1+\ell_2}4}V_{j_1,\ell_1}^{\mathrm{SI}}V_{j_2,\ell_2}^{\mathrm{SI}}
-\ol{\mf{S}}_L\pa_YV_{j,[L]}^{\mathrm{SI}} -\sum_{\substack{0\leq\ell_1\leq L-1\\0\leq\ell_2\leq L\\\ell_1+\ell_2+1>L}}
\nu^{\frac{\ell_1+\ell_2+1}4}\mf{S}_{\ell_1}\pa_YV_{j,\ell_2}^{\mathrm{SI}}\notag\displaybreak[1]\\
&+\sum_{j_1+j_2=j-1}\bigl[T\pa_T\ol{\Theta}_{j_1,L}
-(2j_1+\tfrac12)\ol{\Theta}_{j_1,L}\bigr]\pa_YV_{j_2,[L]}^{\mathrm{SI}}\notag\\*
&+\sum_{j_1+j_2=j-1}\ \sum_{\substack{\ell_1,\ell_2\leq L\\\ell_1+\ell_2>L}}
\nu^{\frac{\ell_1+\ell_2}4}\bigl[T\pa_T\Theta_{j_1,\ell_1}
-(2j_1+\tfrac12)\Theta_{j_1,\ell_1}\bigr]\pa_YV_{j_2,\ell_2}^{\mathrm{SI}}.
\label{eq:M16-difference}
\end{align}
Indeed, the nonlinear difference before applying $\pa_Y$ is exactly
\[
\begin{aligned}
\frac12\sum_{j_1+j_2=j}
\bigl(V_{j_1}^{\mathrm{shock}}V_{j_2}^{\mathrm{shock}}
-V_{j_1,[L]}^{\mathrm{SI}}V_{j_2,[L]}^{\mathrm{SI}}\bigr)
&=V_0^{\mathrm{shock}}Q_{j,L}+Q_{0,L}V_{j,[L]}^{\mathrm{SI}}+\sum_{\substack{j_1+j_2=j\\j_1,j_2>0}}
\left(V_{j_1,[L]}^{\mathrm{SI}}Q_{j_2,L}
+\frac12Q_{j_1,L}Q_{j_2,L}\right).
\end{aligned}
\]
In particular, when $j=1$ the two sums with $j_1,j_2>0$ vanish, but
$\pa_Y(Q_{0,L}V_{1,[L]}^{\mathrm{SI}})$ remains. All five truncation
groups have the opposite signs to those in the equation for
$V_{j,[L]}^{\mathrm{SI}}$.

Introducing $\ol Q_{j,L}=Q_{j,L}-Q_{j,L}^{\mathrm{cut}}$, we obtain
\begin{equation}\label{eq:M16-cut}
\pa_Y^2\ol Q_{j,L}-\pa_Y(V_0^{\mathrm{shock}}\ol Q_{j,L})
+\frac{\dot{\mf{S}}}{\sqrt T}\pa_Y\ol Q_{j,L}=H_{j,L}.
\end{equation}
Here $H_{j,L}$ is the entire right-hand side of
\eqref{eq:M16-difference}, including both nonlinear difference terms,
together with
\[
-\pa_Y^2Q_{j,L}^{\mathrm{cut}}
+\pa_Y(V_0^{\mathrm{shock}}Q_{j,L}^{\mathrm{cut}})
-\frac{\dot{\mf{S}}}{\sqrt T}\pa_YQ_{j,L}^{\mathrm{cut}}.
\]
The grid profiles and phases have already been constructed in
Proposition~\ref{prop4.5}. We hold them fixed when extracting the
dependence on $\ol\Theta_{j-1,L}$. The notation
$H_{j,L}|_{\ol\Theta_{j-1,L}=\pa_T\ol\Theta_{j-1,L}=0}$ below means
that $\Theta_{j-1}^{\mathrm{S}}$ and its derivative are replaced by
$\Theta_{j-1,[L]}$ and its derivative throughout this expression,
including in $Q_{j,L}^{\mathrm{cut}}$.

The only term of $Q_{j,L}^{\mathrm{cut}}$ containing
$\ol\Theta_{j-1,L}$ is
$-\chi(Y)V_{1,0}^{\pm}\ol\Theta_{j-1,L}$. Moreover, the two phase
terms with $(j_1,j_2)=(j-1,0)$ in \eqref{eq:M16-difference} satisfy
\[
\begin{aligned}
&\bigl[T\pa_T\Theta_{j-1}^{\mathrm{S}}-(2j-\tfrac32)\Theta_{j-1}^{\mathrm{S}}\bigr]\pa_YQ_{0,L}\\
&\quad+\bigl[T\pa_T\ol\Theta_{j-1,L}-(2j-\tfrac32)\ol\Theta_{j-1,L}\bigr]\pa_YV_{0,[L]}^{\mathrm{SI}}\\
&=\bigl[T\pa_T\Theta_{j-1,[L]}-(2j-\tfrac32)\Theta_{j-1,[L]}\bigr]\pa_YQ_{0,L}\\
&\quad+\bigl[T\pa_T\ol\Theta_{j-1,L}-(2j-\tfrac32)\ol\Theta_{j-1,L}\bigr]\pa_YV_0^{\mathrm{shock}}.
\end{aligned}
\]
Thus the small coefficient multiplying the current phase derivative is
included exactly in $\pa_YV_0^{\mathrm{shock}}$. Also,
$\pa_Y(Q_{0,L}V_{j,[L]}^{\mathrm{SI}})$ contains only the already
constructed grid phase; no estimate for the current phase difference
is used to bound that term. Consequently,
\[
\begin{aligned}
H_{j,L}={}&H_{j,L}|_{\ol\Theta_{j-1,L}=\pa_T\ol\Theta_{j-1,L}=0}
+\bigl[T\pa_T\ol\Theta_{j-1,L}-(2j-\tfrac32)\ol\Theta_{j-1,L}\bigr]\pa_YV_0^{\mathrm{shock}}\\
&+\ol\Theta_{j-1,L}\left\{\pa_Y^2(\chi V_{1,0}^{\pm})
-\pa_Y\left[\left(V_0^{\mathrm{shock}}-\frac{\dot{\mf S}}{\sqrt T}\right)
\chi V_{1,0}^{\pm}\right]\right\}.
\end{aligned}
\]
By the limits $V_0^{\mathrm{shock}}-\dot{\mf S}/\sqrt T\to\mp\om$
at $Y\to\pm\infty$, the condition $\int_{\R}H_{j,L}\,\md Y=0$
and \eqref{430} give
\[
\begin{aligned}
2T\pa_T(\om\ol\Theta_{j-1,L})-2(2j-2)\om\ol\Theta_{j-1,L}
&=\int_{\R}H_{j,L}|_{\ol\Theta_{j-1,L}=\pa_T\ol\Theta_{j-1,L}=0}\,\md Y.
\end{aligned}
\]
The terminal conditions in Appendix~B give
$\ol\Theta_{j-1,L}(2\nu^{-\beta_2})=0$. Hence the phase difference
satisfies the exact formula
\begin{equation}\label{eq:M16-phase}
\ol\Theta_{j-1,L}(T)
=-\frac{T^{2j-2}}{2\om(T)}\int_T^{2\nu^{-\beta_2}}
\frac{\displaystyle\int_{\R}
H_{j,L}|_{\ol\Theta_{j-1,L}=\pa_T\ol\Theta_{j-1,L}=0}(\tau,Y)\,\md Y}
{\tau^{2j-1}}\,\md\tau.
\end{equation}

We next estimate $H_{j,L}|_{\ol\Theta_{j-1,L}=\pa_T\ol\Theta_{j-1,L}=0}$.
By \eqref{4.14}, the finite sums $V_{j,[L]}^{\mathrm{SI}}$ and their
derivatives have uniformly bounded polynomial growth in $Y$.
The inductive hypothesis and Proposition~\ref{prop4.3} give the
corresponding polynomial bounds for $Q_{j',L}$, $j'<j$, with the
factor $\nu^{(L+1)(\frac14-\frac{\beta_2}{2})}$.
The quadratic products $Q_{j_1,L}Q_{j_2,L}$ have its square, and
$\pa_Y(Q_{0,L}V_{j,[L]}^{\mathrm{SI}})$ has one such factor.
The omitted grid products have total index at least $L+1$; \eqref{4.14}
and \eqref{4.19} give the same bound for all five truncation groups.
These polynomial terms must be combined with the cutoff contributions
before estimating their spatial integrals. The coefficient identities
of Lemmas~\ref{lem3.5} and \ref{lem4.6}, used in the proofs of
Propositions~\ref{prop3.6} and \ref{prop4.5}, cancel the pure
polynomial terms on $|Y|\geq1$. Every remaining term either contains
an exponentially decaying profile remainder or a derivative of $\chi$.
After decreasing the exponential rate within the fixed finite
hierarchy, this proves
\[
\left|\pa_T^m\pa_Y^n
\left(H_{j,L}|_{\ol\Theta_{j-1,L}=\pa_T\ol\Theta_{j-1,L}=0}\right)\right|
\lesssim\nu^{(L+1)(\frac14-\frac{\beta_2}{2})}T^{-m}e^{-\mu|Y|}.
\]
For $j\geq2$, \eqref{eq:M16-phase} and
\[
T^{2j-2}\int_T^{2\nu^{-\beta_2}}\tau^{-(2j-1)}\,\md\tau
\leq\frac1{2j-2}
\]
give the required phase bound with a constant independent of $\nu$.

For $j=1$, a uniform bound for the forcing alone would introduce a
logarithm, so we retain the jump--phase product before truncation.
The spatial integral in the construction of $V_1^{\mathrm{shock}}$
gives
\[
2T\pa_T(\om\Theta_0^{\mathrm{S}})
=-\om(V_{0,1}^++V_{0,1}^-)+V_{1,0}^--V_{1,0}^+.
\]
Applying the same calculation coefficient by coefficient to
\eqref{4.8} yields
\[
\begin{aligned}
T\pa_T\!\sum_{\ell_1+\ell_2\leq L}
&\nu^{\frac{\ell_1+\ell_2}4}
(V_{0,0,\ell_1}^--V_{0,0,\ell_1}^+)\Theta_{0,\ell_2}\\
={}&-\frac12\sum_{\ell_1+\ell_2\leq L}\nu^{\frac{\ell_1+\ell_2}4}
(V_{0,0,\ell_1}^--V_{0,0,\ell_1}^+)
(V_{0,1,\ell_2}^++V_{0,1,\ell_2}^-)\\
&+\sum_{\ell=0}^L\nu^{\frac\ell4}(V_{1,0,\ell}^--V_{1,0,\ell}^+).
\end{aligned}
\]
By Propositions~\ref{prop4.2}--\ref{prop4.3} and \eqref{4.15}, the
difference between the two right-hand sides is bounded by
$\nu^{(L+1)/4}T^{(L+1)/2}$, with the corresponding differentiated
bounds. Furthermore, \eqref{4.14} implies
\[
\begin{aligned}
\biggl|\pa_T^m\biggl(2\om\Theta_{0,[L]}
-\sum_{\ell_1+\ell_2\leq L}\nu^{\frac{\ell_1+\ell_2}4}
(V_{0,0,\ell_1}^--V_{0,0,\ell_1}^+)\Theta_{0,\ell_2}\biggr)\biggr|
\lesssim\nu^{(L+1)(\frac14-\frac{\beta_2}{2})}T^{-m}.
\end{aligned}
\]
Indeed, expand $2\om$ to order $L$ using \eqref{4.19}; every omitted
product has total index at least $L+1$, and
$|V_{0,0,\ell}^{\pm}|\lesssim T^{\ell/2}$ while
$|\Theta_{0,\ell}|\lesssim\nu^{-\beta_2\ell/2}$.
Subtracting the two product equations, integrating in $\md T/T$,
and using the common terminal phase therefore gives
\[
\begin{aligned}
|2\om(T)\ol\Theta_{0,L}(T)|
&\lesssim\nu^{(L+1)(\frac14-\frac{\beta_2}{2})}
+\nu^{\frac{L+1}4}\int_T^{2\nu^{-\beta_2}}\tau^{\frac{L-1}2}\,\md\tau\\
&\lesssim\nu^{(L+1)(\frac14-\frac{\beta_2}{2})}.
\end{aligned}
\]
This includes $L=0$ and introduces no logarithmic loss. Since
$\om\gtrsim ab>0$, differentiating the phase equation now gives,
for every $j\geq1$ and $m\in\N_0$,
\begin{equation*}
|\pa_T^m\ol\Theta_{j-1,L}|
\lesssim\nu^{(L+1)(\frac14-\frac{\beta_2}{2})}T^{-m}.
\end{equation*}
Returning to the full forcing, we obtain
\begin{equation*}
|\pa_T^m\pa_Y^nH_{j,L}|
\lesssim\nu^{(L+1)(\frac14-\frac{\beta_2}{2})}T^{-m}e^{-\mu|Y|},
\qquad \int_{\R}H_{j,L}\,\md Y=0.
\end{equation*}
Integrating \eqref{eq:M16-cut} once gives
\[
\pa_Y\ol Q_{j,L}
-\left(V_0^{\mathrm{shock}}-\frac{\dot{\mf S}}{\sqrt T}\right)\ol Q_{j,L}
=\int_{-\infty}^{Y}H_{j,L}(T,Y')\,\md Y'.
\]
The right-hand side decays at both infinities. The normalization
$\ol Q_{j,L}(T,0)=0$ follows from that of the profiles and from
$\chi=0$ near the origin. By \eqref{pro}, the integrating factor
is $\cosh^2(\om(T)Y/2)$, and hence
\[
\ol Q_{j,L}(T,Y)=\frac{1}{\cosh^2(\om(T)Y/2)}
\int_0^Y\cosh^2(\om(T)Y'/2)
\int_{-\infty}^{Y'}H_{j,L}(T,Y'')\,\md Y''\,\md Y'.
\]
The uniform lower bound for $\om$, together with the estimates above,
proves the first estimate in \eqref{4.34}, after decreasing $\mu$ if
necessary. Differentiating this formula and the equation proves the
same estimate for all indicated derivatives. Relabelling $j-1$ as
$j$ in the phase estimate proves the second estimate in \eqref{4.34}
and completes the induction.

\end{proof}

We introduce the rescaled profiles of the shock-profile expansion:
\begin{equation}\label{4.36}
	\begin{aligned}
	U_{j,\ell}^{\mathrm{SI}}(T,X)
	&=T^{\frac12-2j}\sum_{p=0}^{\ell}\frac1{p!}
	\sum_{\substack{\ell_0+\ell_1+\cdots+\ell_p=\ell\\
	\ell_0\ge0,\ \ell_1,\ldots,\ell_p\ge1}}
	(\pa_Y^pV_{j,\ell_0}^{\mathrm{SI}})
	(T,\sqrt{T}X+\Theta_0^{\mathrm{I}}(T))\\
	&	\times\left(\prod_{r=1}^p\Theta_{\ell_r}^{\mathrm{I}}(T)\right).
	\end{aligned}
\end{equation}
Then $U_{j,\ell}^{\mathrm{SI}}$ satisfies the equation:
\begin{prop}
	For $j,\ell\in \N_0$, with $U_{-1,\ell}^{\mathrm{SI}}=0$, we have
	\begin{equation}\label{4.43}
		\begin{aligned}
		\pa_X^2U_{j,\ell}^{\mathrm{SI}}
		&+T^{\frac12}\sum_{\ell_1+\ell_2=\ell-1}
		\mf{S}_{\ell_1}\pa_XU_{j,\ell_2}^{\mathrm{SI}}
		=\frac12\pa_X\sum_{\substack{j_1+j_2=j\\\ell_1+\ell_2=\ell}}
		U_{j_1,\ell_1}^{\mathrm{SI}}U_{j_2,\ell_2}^{\mathrm{SI}}
		+\pa_TU_{j-1,\ell}^{\mathrm{SI}}\\
		&\quad-\sum_{\ell_1+\ell_2=\ell}
		\pa_T\left(\frac{\Theta_{\ell_1}^{\mathrm{I}}}{T^{\frac12}}\right)
		\pa_XU_{j-1,\ell_2}^{\mathrm{SI}}
		+\sum_{\substack{j_1+j_2=j-1\\\ell_1+\ell_2=\ell}}
		\pa_T\left(\frac{\Theta_{j_1,\ell_1}}{T^{2j_1+\frac12}}\right)
		\pa_XU_{j_2,\ell_2}^{\mathrm{SI}}.
		\end{aligned}
	\end{equation}
\end{prop}
\begin{proof}
For $j=0$, the identity follows from \eqref{0si} and \eqref{4.36} by the Taylor product rule, using $\pa_X=T^{1/2}\pa_Y$. For $j=1$, the chain rule gives
\[
\begin{aligned}
&\pa_TU_0^{\mathrm{shock}}
-\pa_T\left(\frac{\Theta}{\sqrt T}\right)\pa_XU_0^{\mathrm{shock}}\\
&\quad=T^{-\frac12}\left[
T\pa_TV_0^{\mathrm{shock}}+\frac Y2\pa_YV_0^{\mathrm{shock}}
+\frac12V_0^{\mathrm{shock}}
\right]_{Y=\sqrt T X+\Theta}.
\end{aligned}
\]
Collecting the Taylor coefficients at fixed $(T,X)$ and using \eqref{4.8} with $j=1$ yields
\[
\begin{aligned}
\pa_X^2U_{1,\ell}^{\mathrm{SI}}
&+\sqrt T\sum_{\ell_1+\ell_2=\ell-1}\mf S_{\ell_1}\pa_XU_{1,\ell_2}^{\mathrm{SI}}
=\pa_X\sum_{\ell_1+\ell_2=\ell}U_{0,\ell_1}^{\mathrm{SI}}U_{1,\ell_2}^{\mathrm{SI}}+\pa_TU_{0,\ell}^{\mathrm{SI}}\\
&\quad-\sum_{\ell_1+\ell_2=\ell}
\pa_T\left(\frac{\Theta_{\ell_1}^{\mathrm I}-\Theta_{0,\ell_1}}{\sqrt T}\right)
\pa_XU_{0,\ell_2}^{\mathrm{SI}},
\end{aligned}
\]
which is \eqref{4.43} for $j=1$. The same calculation gives the general case.
\end{proof}
Here we explain why we introduce \eqref{4.36}. Indeed, we seek the expansion:
\begin{equation}\label{4.38}
	\begin{aligned}
	U_j^{\mathrm{shock}}(T,X)\sim \sum\limits_{\ell\geq0} \nu^{\frac{\ell}4}	U_{j,\ell}^{\mathrm{SI}}(T,X). 
	\end{aligned}
\end{equation}
Substituting it into \eqref{348} and collecting terms of the same order in $\nu$, we obtain the same equation for $U_{j,\ell}^{\mathrm{SI}}$ as in \eqref{4.43}. 

Define
\begin{equation}\label{eq:SI-cut}
	\begin{aligned}
	U_{j,\ell}^{\mathrm{cut}}(T,X)
	&=T^{\frac12-2j}\sum_{p=0}^{\ell}\frac1{p!}
	\sum_{\substack{\ell_0+\ell_1+\cdots+\ell_p=\ell\\
	\ell_0\ge0,\ \ell_1,\ldots,\ell_p\ge1}}
	(\pa_Y^pV_{j,\ell_0}^{\mathrm{cut}})
	(T,\sqrt{T}X+\Theta_0^{\mathrm{I}}(T))\\
	&	\times\left(\prod_{r=1}^p\Theta_{\ell_r}^{\mathrm{I}}(T)\right).
	\end{aligned}
\end{equation}
\begin{prop}\label{prop:SI-estimates}
Let $\beta_2\in(0,\frac1{100})$ and choose $\Theta_\ell^{\mathrm I}$ as in \eqref{5.43}.
There exists $\mu>0$ such that, for all $j,\ell,m,n\in\N_0$,
$T\in[1,2\nu^{-\beta_2}]$ and $X\in\R$,
\begin{equation}\label{eq:SI-estimates}
	\begin{aligned}
	|\pa_T^m\pa_X^n(U_{j,\ell}^{\mathrm{SI}}-U_{j,\ell}^{\mathrm{cut}})|
	&\lesssim \nu^{-\beta_2\ell/2}
	T^{\frac12-2j-m+\frac n2}e^{-\mu\sqrt T|X|},\\
	|\pa_T^m\pa_X^nU_{j,\ell}^{\mathrm{SI}}|
	&\lesssim \nu^{-\beta_2\ell/2}
	T^{\frac12-2j-m+\frac n2}(1+\sqrt T|X|)^j.
	\end{aligned}
\end{equation}
The implicit constants may depend on the indices and the fixed phase truncation, but not on $\nu,T,X$.
\end{prop}
\begin{proof}
By \eqref{4.14}, \eqref{5.43}, Proposition~\ref{prop4.2} and \eqref{4.15},
\[
|\pa_T^m\Theta_\ell^{\mathrm I}|
+|\pa_T^mV_{i,k,\ell}^{\pm}|
\lesssim \nu^{-\beta_2\ell/2}T^{-m}.
\]
Apply the chain rule to \eqref{4.36} and \eqref{eq:SI-cut}.
Each $X$ derivative contributes $T^{1/2}$; each $T$ derivative contributes
$T^{-1}$ and at most one factor $1+\sqrt T|X|$.
Since $|\Theta_0^{\mathrm I}|\lesssim1$, \eqref{4.14} gives the first estimate
after decreasing $\mu$ to absorb these polynomial factors.
For the second, $V_{j,\ell}^{\mathrm{cut}}$ is a polynomial of degree at most $j$
on each of $Y>1$ and $Y<-1$. Differentiating its composition does not increase
this degree, while derivatives of $\chi$ have bounded support.
The coefficient bounds above and the first estimate finish the proof.
\end{proof}

\begin{prop}\label{prop:SI-reconstruction}
Let $\beta_2\in(0,\frac1{100})$ and choose $\Theta_\ell^{\mathrm I}$ as in \eqref{5.43}.
For each fixed $j,L,m,n\in\N_0$, $T\in[1,2\nu^{-\beta_2}]$, $X\in\R$
and sufficiently small $\nu>0$, we have
\begin{equation}\label{eq:SI-reconstruction}
	\begin{aligned}
	&\left|\pa_T^m\pa_X^n\left(
	U_j^{\mathrm{shock}}-\sum_{\ell=0}^L\nu^{\frac\ell4}U_{j,\ell}^{\mathrm{SI}}
	\right)\right|\\
	&\qquad\lesssim \nu^{(L+1)(\frac14-\frac{\beta_2}{2})}
	T^{\frac12-2j-m+\frac n2}(1+\sqrt T|X|)^j.
	\end{aligned}
\end{equation}
The implicit constant is independent of $\nu,T,X$.
\end{prop}
\begin{proof}
Proposition~\ref{prop4.3}, \eqref{vik}, \eqref{4.15} and the phase estimate in
\eqref{4.34} control the finite products in $Q_{j,L}^{\mathrm{cut}}$.
Adding the first estimate in \eqref{4.34} gives
\[
|\pa_T^m\pa_Y^nQ_{j,L}|
\lesssim \nu^{(L+1)(\frac14-\frac{\beta_2}{2})}T^{-m}(1+|Y|)^j.
\]
First replace $V_j^{\mathrm{shock}}$ by $V_{j,[L]}^{\mathrm{SI}}$ in
$U_j^{\mathrm{shock}}$, and then replace $\Theta$ by
$\sum_{r=0}^L\nu^{r/4}\Theta_r^{\mathrm I}$.
The preceding estimate and \eqref{5.44} bound these errors.
For each $\ell\le L$, Taylor-expand $V_{j,\ell}^{\mathrm{SI}}$ at
$Y=\sqrt T X+\Theta_0^{\mathrm I}$ through order $L-\ell$.
The phase increment is $O(\nu^{\frac14-\frac{\beta_2}{2}})$, so its weighted
Taylor remainder is
$O(\nu^{(L+1)(\frac14-\frac{\beta_2}{2})}(1+\sqrt T|X|)^j)$.
Terms of total order greater than $L$ have the same bound; the retained terms,
after multiplication by $T^{\frac12-2j}$, are exactly the sum in
\eqref{eq:SI-reconstruction} by \eqref{4.36}.
Differentiating the integral Taylor remainder gives the factor $T^{-m+n/2}$
as in Proposition~\ref{prop:SI-estimates}, completing the proof.
\end{proof}

We also obtain the outer--shock grid matching estimate. Introduce
\begin{equation*}
	\begin{aligned}
		U_{[J],\ell}^{\mathrm{SI}}=\sum\limits_{j=0}^JU_{j,\ell}^{\mathrm{SI}},\quad 	U_{[K],\ell}^{\mathrm{OI}}=\sum\limits_{k=0}^KU_{k,\ell}^{\mathrm{OI}}.
	\end{aligned}
\end{equation*}
Following the proof of Proposition~\ref{prop3.9}, we have the following estimate:
\begin{prop}\label{prop4.9}
	Let $\beta_1\in (0,\frac1{100})$. In the region $\{(T,X):T\in[1,2\nu^{-\beta_2}],\ 2T^{-1/2}\leq |X|\leq4 T^{-\frac12+\beta_1}\}$ there exists a constant $\mu>0$ such that for all $m,n,J,K,\ell\in\N_0$, we have
	\begin{equation}\label{4.39}
		\begin{aligned}
			|\pa_T^m\pa_X^n&(U_{[J],\ell}^{\mathrm{SI}}-{U}_{[K],\ell}^{\mathrm{OI}})|\lesssim \nu^{-\beta_2\ell/2}|T|^{\frac12-m+\frac n2}\\
			&\times(e^{-\mu\sqrt{T}|X|}+(\frac{|X|+1}{T^{\frac32}})^{K+1}+(\frac{|X|+1}{T^{\frac32}})^{J+1}).
		\end{aligned}
	\end{equation}
\end{prop}

To end this section, we give a lemma that will be used in the next section and shows that the following asymptotic expansion
\begin{equation*}
	\begin{aligned}
	\Theta_{\ell}^{\mathrm{I}}\sim \sum\limits_{j\geq0} T^{-2j}\Theta_{j,\ell}
	\end{aligned}
\end{equation*}
holds.
\begin{lemma}\label{lem5.8}
	For all $\ell,m,J\in \N_0$ and $T\in [1,2\nu^{-\beta_2}]$, we have
	\begin{equation*}
		\begin{aligned}
			|\pa_T^m(\Theta_{\ell}^{\mathrm{I}}-\sum\limits_{j=0}^J T^{-2j}\Theta_{j,\ell})|\lesssim \nu^{-\beta_2\ell/2}T^{-2J-2-m}.
		\end{aligned}
	\end{equation*}
\end{lemma}
In fact, Remark~\ref{rem3.4} tells us that for some sufficiently large $J_*\in \N$, we define
	\begin{equation*}
	\begin{aligned}
	\Theta=\sum\limits_{j=0}^{J_*}T^{-2j}\Theta_{j}^{\mathrm{S}}.
	\end{aligned}
\end{equation*}
From \eqref{4.1}, for $\ell\in \N_0$ we can define
	\begin{equation}\label{5.43}
	\begin{aligned}
		\Theta_{\ell}^{\mathrm{I}}=\sum\limits_{j=0}^{J_*}T^{-2j}\Theta_{j,\ell}.
	\end{aligned}
\end{equation}
Hence Lemma~\ref{lem5.8} holds automatically provided $J\leq J_*$. For convenience as in Remark~\ref{rem3.4}, we omit the upper bound $J\leq J_*$. The following lemma shows that, under the choice of $\Theta_{\ell}^{\mathrm{I}}$ in \eqref{5.43}, the asymptotic expansion \eqref{4.1} holds.
\begin{lemma}\label{lem5.9}
	For all $m,L\in \N_0$ and $T\in [1,2\nu^{-\beta_2}]$, we have
		\begin{equation}\label{5.44}
		\begin{aligned}
			|\pa_T^m(\Theta-\sum\limits_{\ell=0}^L \nu^{\frac{\ell}4}\Theta_{\ell}^{\mathrm{I}})|\lesssim T^{-m} \nu^{(L+1)(\frac14-\frac{\beta_2}{2})}.
		\end{aligned}
	\end{equation}
\end{lemma}
\begin{proof}
	From Remark~\ref{rem3.4} and \eqref{5.43}, we have
		\begin{equation}\label{5.45}
		\begin{aligned}
			|\pa_T^m(\Theta-\sum\limits_{\ell=0}^L \nu^{\frac{\ell}4}\Theta_{\ell}^{\mathrm{I}})|\leq \sum\limits_{j=0}^{J_*}|\pa_T^m[T^{-2j}(\Theta_j^{\mathrm{S}}-\sum\limits_{\ell=0}^L \nu^{\frac{\ell}4}\Theta_{j,\ell})]|.
		\end{aligned}
	\end{equation}
	Then \eqref{5.44} comes from \eqref{4.34} and \eqref{5.45}.
\end{proof}

\section{Inner-Grid Matching}
In this section we construct the inner-grid approximation
$U_{\ell,0}^{\mathrm{in}}$ of the inner expansion $U_\ell^{\mathrm{in}}$,
with errors controlled by index functions. We work on the time interval
$[1,2\nu^{-\beta_2}]$ for some small $\beta_2>0$. Section~6.1 collects the
notation and states the main result of the section; Section~6.2 derives the
error equation for $P_{J,K,0}$ and establishes the three-zone estimate for
the residual $G_{J,K,0}$. The weighted energy estimates and the inductive
construction of $P_{J,K,\ell,0}$ and $U_{\ell,0}^{\mathrm{in}}$ are carried
out in the subsequent subsections.

\subsection{Setup and notation}

Let $U_{k,\ell}^{\mathrm{OI}}$ and $U_{j,\ell}^{\mathrm{SI}}$ be the
outer--inner and shock--inner grid functions introduced in Section~5.
With the initial data $U_\ell^{\mathrm{in}}(0,X)$ prescribed in \cite{CG},
we solve \eqref{4.4} for $0\le T\le1$. The following matching estimate
is obtained by an argument similar to the proof of
\cite[Proposition~4.2]{CG}; details are given in Appendix~\ref{app:prop61}.
\begin{lemma}\label{lem:52}
For $K,\ell,m,n\in\N_0$, we have
\begin{equation*}
|\pa_T^m\pa_X^n(U_\ell^{\mathrm{in}}-U_{[K],\ell}^{\mathrm{OI}})|\lesssim
\mf D^{\ell-3-4K-2m-3n},\qquad \text{for } |X|\ge1,\ T\in[0,1].
\end{equation*}
\end{lemma}

We write $\wh X=X+T^{-\frac12}\Theta(T)$ for the inner variable. Let
$\wh\zeta_1\in C_c^{\infty}(\R)$ be a cutoff with $\wh\zeta_1(z)=1$ on
$|z|\le\frac12$ and $\wh\zeta_1(z)=0$ for $|z|\ge1$, and let
$\wh\zeta_2\in C^{\infty}(\R)$ satisfy $\wh\zeta_2(z)=1$ for $z\ge1$ and
$\wh\zeta_2(z)=0$ for $z\leq\frac12$. For $\beta_1\in(0,\frac12]$, set
\begin{equation}\label{64}
\zeta(T,X)=\wh\zeta_1\bigl(XT^{\frac12-\beta_1}\bigr)\,\wh\zeta_2(T).
\end{equation}
Thus $\zeta=1$ in the shock zone, $\zeta=0$ in the outer zone, and
$0<\zeta<1$ in the matching zone, where $\pa_X\zeta$ is supported. We define
the combined grid
\begin{equation}\label{eq:SOI}
U_{j,k,\ell}^{\mathrm{SOI}}=\zeta U_{j,\ell}^{\mathrm{SI}}+(1-\zeta)U_{k,\ell}^{\mathrm{OI}},
\qquad
U_{[J],[K],\ell}^{\mathrm{SOI}}=\zeta U_{[J],\ell}^{\mathrm{SI}}+(1-\zeta)U_{[K],\ell}^{\mathrm{OI}}.
\end{equation}

\begin{definition}[Index functions]
A function $\psi:\N_0^2\to\R$ is an \emph{index function} if it is increasing
in both arguments, and if there exist $\mu>0$ and $C\ge0$ such that
$\psi(J,K)\ge\mu\min\{J,K\}-C$ for all $J,K\in\N_0$.
\end{definition}

The following is the main result of this section.
\begin{prop}
\label{thm:main}
Fix $L\in\N_0$. There exist smooth corrections $P_{J,K,\ell,0}$,
index functions $\psi_\ell$ and families
$\{a_{m,n}^{(\ell)}\}\subset\N_0$, with $a_{m,n}^{(\ell)}$ independent of
$\nu,J,K$, such that, for $0\le\ell\le L$, each fixed $m,n\in\N_0$
and all sufficiently large $J,K\in\N$, on $[1,2\nu^{-\beta_2}]\times\R$,
\begin{equation}\label{eq:Pellh}
|\pa_T^m\pa_X^nP_{J,K,\ell,0}|
\lesssim \nu^{-\frac{\beta_2}{2}\ell}
(T+X^2)^{-\psi_\ell(J,K)+a_{m,n}^{(\ell)}}.
\end{equation}
Moreover, the profiles
\[
U_{\ell,0}^{\mathrm{in}}
:=P_{J,K,\ell,0}+U_{[J],[K],\ell}^{\mathrm{SOI}}
\]
satisfy
\begin{equation}\label{eq:UellH}
\begin{aligned}
\pa_TU_{\ell,0}^{\mathrm{in}}
&-T^{\frac12}\sum_{\ell_1+\ell_2=\ell-1}
\mf S_{\ell_1}\pa_XU_{\ell_2,0}^{\mathrm{in}}+\frac12\pa_X\sum_{\ell_1+\ell_2=\ell}
U_{\ell_1,0}^{\mathrm{in}}U_{\ell_2,0}^{\mathrm{in}}
=\pa_X^2U_{\ell,0}^{\mathrm{in}},
\end{aligned}
\end{equation}
with initial data
\[
U_{\ell,0}^{\mathrm{in}}(1,X)
=U_\ell^{\mathrm{in}}(1,X)
+\ml O(\nu^{-\frac{\beta_2}{2}\ell})\rho(X),
\]
where $\rho\in C_c^\infty(\R)$ is fixed, nonnegative, independent of
$\nu,J,K$, and satisfies $\int_\R\rho\,\md X=1$.
The coefficient of $\rho$ is independent of $X$, and the implicit constants
are independent of $\nu,T,X$.
\end{prop}

\subsection{Proof of Proposition~\ref{thm:main} for $\ell=0$}

Set $P_{J,K,0}=U_0^{\mathrm{in}}-U_{[J],[K],0}^{\mathrm{SOI}}$.
Using \eqref{4.4} with $\ell=0$, we obtain
\begin{equation}\label{eq:P}
	\begin{aligned}
\pa_TP_{J,K,0}
+\pa_X[U_{[J],[K],0}^{\mathrm{SOI}}P_{J,K,0}]&+P_{J,K,0}\pa_XP_{J,K,0}
-\pa_X^2P_{J,K,0}\\
&=G_{J,K,0},
\end{aligned}
\end{equation}
where
\begin{equation}\label{eq:G0}
-G_{J,K,0}=\pa_TU_{[J],[K],0}^{\mathrm{SOI}}
+U_{[J],[K],0}^{\mathrm{SOI}}\pa_XU_{[J],[K],0}^{\mathrm{SOI}}
-\pa_X^2U_{[J],[K],0}^{\mathrm{SOI}}.
\end{equation}
Writing $U_{[J],[K],0}^{\mathrm{SOI}}=\zeta U_{[J],0}^{\mathrm{SI}}+(1-\zeta)U_{[K],0}^{\mathrm{OI}}$,
we expand the right-hand side of \eqref{eq:G0} as
\begin{equation}\label{eq:G0exp}
\begin{aligned}
-G_{J,K,0}
&=\zeta\pa_TU_{[J],0}^{\mathrm{SI}}+(1-\zeta)\pa_TU_{[K],0}^{\mathrm{OI}}
+\pa_T\zeta\,(U_{[J],0}^{\mathrm{SI}}-U_{[K],0}^{\mathrm{OI}})\\
&\quad+\zeta^2U_{[J],0}^{\mathrm{SI}}\pa_XU_{[J],0}^{\mathrm{SI}}
+(1-\zeta)^2U_{[K],0}^{\mathrm{OI}}\pa_XU_{[K],0}^{\mathrm{OI}}\\
&\quad+\zeta(1-\zeta)\bigl[U_{[J],0}^{\mathrm{SI}}\pa_XU_{[K],0}^{\mathrm{OI}}
+U_{[K],0}^{\mathrm{OI}}\pa_XU_{[J],0}^{\mathrm{SI}}\bigr]\\
&\quad+\zeta\pa_X\zeta\,|U_{[J],0}^{\mathrm{SI}}|^2
-(1-\zeta)\pa_X\zeta\,|U_{[K],0}^{\mathrm{OI}}|^2
+(1-2\zeta)\pa_X\zeta\,U_{[J],0}^{\mathrm{SI}}U_{[K],0}^{\mathrm{OI}}\\
&\quad-\Bigl[\zeta\pa_X^2U_{[J],0}^{\mathrm{SI}}+(1-\zeta)\pa_X^2U_{[K],0}^{\mathrm{OI}}
+2\pa_X\zeta\,\pa_X(U_{[J],0}^{\mathrm{SI}}-U_{[K],0}^{\mathrm{OI}})\\
&\quad+\pa_X^2\zeta\,(U_{[J],0}^{\mathrm{SI}}-U_{[K],0}^{\mathrm{OI}})\Bigr].
\end{aligned}
\end{equation}

\medskip
We first consider terms free of derivatives of $\zeta$.
For $J\ge1$, summing \eqref{4.43} with $\ell=0$ and grouping the phase terms gives
\begin{equation}\label{eq:M1}
\begin{aligned}
\pa_X^2U_{[J],0}^{\mathrm{SI}}
&=\frac12\pa_X\sum_{j_1+j_2\le J}
U_{j_1,0}^{\mathrm{SI}}U_{j_2,0}^{\mathrm{SI}}
+\pa_TU_{[J-1],0}^{\mathrm{SI}}\\
&\quad-\sum_{j_2=0}^{J-1}
\pa_T\left(\frac{\Theta_0^{\mathrm I}
-\sum_{j_1=0}^{J-1-j_2}T^{-2j_1}\Theta_{j_1,0}}{\sqrt T}\right)
\pa_XU_{j_2,0}^{\mathrm{SI}}.
\end{aligned}
\end{equation}
Consequently,
\begin{equation}\label{eq:M2}
\begin{aligned}
&\zeta\left(\pa_TU_{[J],0}^{\mathrm{SI}}
+\frac12\pa_X|U_{[J],0}^{\mathrm{SI}}|^2
-\pa_X^2U_{[J],0}^{\mathrm{SI}}\right)\\
&=\zeta\Biggl[\pa_TU_{J,0}^{\mathrm{SI}}
+\frac12\pa_X\sum_{\substack{j_1+j_2>J\\j_1,j_2\le J}}
U_{j_1,0}^{\mathrm{SI}}U_{j_2,0}^{\mathrm{SI}}\\
&\qquad+\sum_{j_2=0}^{J-1}
\pa_T\left(\frac{\Theta_0^{\mathrm I}
-\sum_{j_1=0}^{J-1-j_2}T^{-2j_1}\Theta_{j_1,0}}{\sqrt T}\right)
\pa_XU_{j_2,0}^{\mathrm{SI}}\Biggr].
\end{aligned}
\end{equation}
Similarly, summing \eqref{4.5} with $\ell=0$ gives
\begin{equation}\label{eq:M3}
\begin{aligned}
&(1-\zeta)\left(\pa_TU_{[K],0}^{\mathrm{OI}}
+\frac12\pa_X|U_{[K],0}^{\mathrm{OI}}|^2
-\pa_X^2U_{[K],0}^{\mathrm{OI}}\right)\\
&=(1-\zeta)\left[-\pa_X^2U_{K,0}^{\mathrm{OI}}
+\frac12\pa_X\sum_{\substack{k_1+k_2>K\\k_1,k_2\le K}}
U_{k_1,0}^{\mathrm{OI}}U_{k_2,0}^{\mathrm{OI}}\right].
\end{aligned}
\end{equation}

The remaining quadratic terms of \eqref{eq:G0exp} that do not involve
$\pa\zeta$ reduce to
\begin{equation}\label{eq:M5}
\begin{aligned}
&\zeta^2U_{[J],0}^{\mathrm{SI}}\pa_XU_{[J],0}^{\mathrm{SI}}
+(1-\zeta)^2U_{[K],0}^{\mathrm{OI}}\pa_XU_{[K],0}^{\mathrm{OI}}
+\zeta(1-\zeta)\bigl[U_{[J],0}^{\mathrm{SI}}\pa_XU_{[K],0}^{\mathrm{OI}}
\\
&+U_{[K],0}^{\mathrm{OI}}\pa_XU_{[J],0}^{\mathrm{SI}}\bigr]
\quad-\zeta U_{[J],0}^{\mathrm{SI}}\pa_XU_{[J],0}^{\mathrm{SI}}
-(1-\zeta)U_{[K],0}^{\mathrm{OI}}\pa_XU_{[K],0}^{\mathrm{OI}}
\\
&=-\frac12\zeta(1-\zeta)\pa_X\bigl(U_{[K],0}^{\mathrm{OI}}-U_{[J],0}^{\mathrm{SI}}\bigr)^2.
\end{aligned}
\end{equation}
Finally, the terms in \eqref{eq:G0exp} involving $\pa\zeta$ are given
explicitly by
\begin{equation}\label{eq:dxz}
\begin{aligned}
\pa_X\zeta\Bigl[&\zeta U_{[J],0}^{\mathrm{SI}}\bigl(U_{[J],0}^{\mathrm{SI}}-U_{[K],0}^{\mathrm{OI}}\bigr)
+(1-\zeta)U_{[K],0}^{\mathrm{OI}}\bigl(U_{[J],0}^{\mathrm{SI}}-U_{[K],0}^{\mathrm{OI}}\bigr)\\
&-2\pa_X\bigl(U_{[J],0}^{\mathrm{SI}}-U_{[K],0}^{\mathrm{OI}}\bigr)\Bigr]\\
&+\pa_T\zeta\bigl(U_{[J],0}^{\mathrm{SI}}-U_{[K],0}^{\mathrm{OI}}\bigr)-\pa_X^2\zeta\bigl(U_{[J],0}^{\mathrm{SI}}-U_{[K],0}^{\mathrm{OI}}\bigr).
\end{aligned}
\end{equation}
They are supported in the transition region, and they are controlled by the
common factor $U_{[J],0}^{\mathrm{SI}}-U_{[K],0}^{\mathrm{OI}}$ together with
the matching estimate of Section~5.

With $\beta_1$ as in Section~5, Lemma~\ref{lem5.8} yields
for $0\le j_2\le J-1$,
\[
\left|\pa_T^{m+1}\left(
\frac{\Theta_0^{\mathrm I}
-\sum_{j_1=0}^{J-1-j_2}T^{-2j_1}\Theta_{j_1,0}}{\sqrt T}\right)\right|
\lesssim T^{-2(J-j_2)-\frac32-m}.
\]
Thus Proposition~\ref{prop:SI-estimates} bounds the right-hand side of
\eqref{eq:M2} by $T^{-(2-\beta_1)J-\frac12+\beta_1}$ on the support of $\zeta$.
Proposition~\ref{prop4.2} bounds \eqref{eq:M3} by $\mf D^{-4K-5}$.
On the transition region, \eqref{4.39}, with uniform bounds for bounded $T$, yields
\[
\left|\pa_T^m\pa_X^n
\bigl(U_{[J],0}^{\mathrm{SI}}-U_{[K],0}^{\mathrm{OI}}\bigr)\right|
\lesssim T^{-1-\frac32\min\{J,K\}-m+\frac n2}.
\]
Here the exponential term is absorbed into any prescribed power of $T^{-1}$.
Using $|\pa_T\zeta|\lesssim T^{-1}$,
$|\pa_X^n\zeta|\lesssim T^{n(\frac12-\beta_1)}$ and the grid bounds,
we bound \eqref{eq:M5} and \eqref{eq:dxz} by $T^{-\frac32\min\{J,K\}}$.
Since $T+X^2\sim T$ on the support of $\zeta$ and
$\mf D^6\sim T^3+X^2\gtrsim T+X^2$, the residual estimate below follows with
\[
\psi_{0}(J,K)=\frac12\min\{J,K\}-3.
\]
At $T=1$, Lemma~\ref{lem:52} applies on $|X|\ge1$, where $\zeta$
and all its derivatives vanish. Using $\mf D(1,X)\sim(1+X^2)^{1/6}$
gives the initial estimate there.
On $|X|\le1$, the initial profile and grid derivatives are uniformly bounded.
Differentiating the preceding residual formulas introduces only fixed powers
depending on $m,n$, independently of $J,K,\nu$; enlarging $a_{m,n}$ covers
these losses, with $a_{0,0}=0$. This proves the following lemma:
\begin{lemma}[Three-zone residual estimate]
\label{lem:G0}
There exist an index function $\psi_{0}$ and a family $\{a_{m,n}\}\subset\N_0$ with $a_{0,0}=0$ such that, for each fixed $m,n\in\N_0$ and all sufficiently large $J,K$, on $[1,2\nu^{-\beta_2}]$,
\begin{equation*}
	\begin{aligned}
&|\pa_T^m\pa_X^nG_{J,K,0}|\lesssim
(T+X^2)^{-\psi_{0}(J,K)-3+a_{m,n}}\\
&|\pa_T^m\pa_X^nP_{J,K,0}(1,X)|\lesssim
(1+X^2)^{-\psi_{0}(J,K)-1+a_{m,n}}.
\end{aligned}
\end{equation*}
\end{lemma}

\subsection{Initial data correction}

The source $G_{J,K,0}$ and the initial discrepancy $P_{J,K,0}(1,\cdot)$
need not have zero spatial mean. We construct $P_{J,K,0,0}$ solving
\eqref{eq:P} with corrected initial data. The decomposition below gives
an equation for $\wh P^0$ with zero-mean forcing and initial data,
as required by the energy estimate.

Define $P_{J,K,0,0}$ to be the solution of
\begin{equation}\label{eq:P000}
\pa_TP_{J,K,0,0}+\pa_X[U_{[J],[K],0}^{\mathrm{SOI}}P_{J,K,0,0}]
+P_{J,K,0,0}\pa_XP_{J,K,0,0}-\pa_X^2P_{J,K,0,0}=G_{J,K,0},
\end{equation}
with initial data $P_{J,K,0,0}(1,X)$ to be chosen below. Let
$\rho\in C_c^{\infty}(\R)$ be non-negative with $\int_\R\rho\,\md X=1$, and set
$\rho_j(T,X)=\lam_j(T)\rho(X)$. We look for $P_{J,K,0,0}$ in the form
$P_{J,K,0,0}=\wh P^0+\rho_0$. Then $\wh P^0$ satisfies
\begin{equation}\label{eq:whP0}
\begin{aligned}
\pa_T\wh P^0&+\pa_X[(U_{[J],[K],0}^{\mathrm{SOI}}+\rho_0)\wh P^0]
+\wh P^0\pa_X\wh P^0-\pa_X^2\wh P^0\\
&=G_{J,K,0}-\bigl(\pa_T\rho_0+\pa_X[U_{[J],[K],0}^{\mathrm{SOI}}\rho_0]
+\rho_0\pa_X\rho_0-\pa_X^2\rho_0\bigr)\\
&=:\wh G^0.
\end{aligned}
\end{equation}
We choose $\lam_0$ so that $\int_\R\wh G^0\,\md X=0$, namely
\begin{equation}\label{eq:lam0}
\lam_0(T)=-\int_T^{2\nu^{-\beta_2}}\int_\R G_{J,K,0}\,\md X\,\md T,
\end{equation}
for then $\pa_T\lam_0=\int_\R G_{J,K,0}\,\md X$ and the spatial means of the
two sides of \eqref{eq:whP0} balance. We define the initial data by
\begin{equation}\label{eq:init0}
\begin{aligned}
P_{J,K,0,0}(1,X)&=P_{J,K,0}(1,X)-\mu_0\rho(X),\\
\wh P^0(1,X)&=P_{J,K,0}(1,X)-(\lam_0(1)+\mu_0)\rho(X),
\end{aligned}
\end{equation}
and fix
\begin{equation}\label{eq:mu0}
\mu_0=\int_\R P_{J,K,0}(1,X)\,\md X-\lam_0(1)=\ml O(1),
\end{equation}
so that $\int_\R\wh P^0(1,X)\,\md X=0$.

The principal part of $U_{[J],[K],0}^{\mathrm{SOI}}$ is $U_{0,0,0}^{\mathrm{SOI}}$.
 Combining the grid-profile estimates of Section 5, we
obtain, for $m\in \N_0$ and $n\in\{0,1\}$,
\begin{equation}\label{eq:ve}
	\begin{aligned}
\bigl|\pa_T^m\pa_X^n\bigl(U_{[J],[K],0}^{\mathrm{SOI}}-U_{0,0,0}^{\mathrm{SOI}}\bigr)\bigr|
\lesssim& T^{-m}\Bigl(\frac{|\pa_X^nX|+1}{T}+T^{\frac {n-3}2}e^{-\mu|X|\sqrt T}\Bigr)1_{\{(T,X):|X|T^{\frac12-\beta_1}\leq 1\}}
\\
&+\mf D^{-3-2m-3n}.
\end{aligned}
\end{equation}
Here $\mf d$ denotes the gauge distance of Section 3, and $\mf D$ is its
rescaling, $\mf d(\nu^{\frac12}T,\nu^{\frac34}X)=\nu^{\frac14}\allowbreak\mf D(T,X)$,
introduced in Section 4.2. We define the speed
\begin{equation}\label{eq:V}
\ml V=U_{0,0,0}^{\mathrm{SOI}},
\end{equation}
and write
$U_{[J],[K],0}^{\mathrm{SOI}}=\ml V+\ml V_e$, so that $\ml V_e$ satisfies
\eqref{eq:ve}.

Set
\begin{equation}\label{eq:varphi}
\vph(T,X)=T+\bigl(X+T^{-\frac12}\Theta_0^{\mathrm I}(T)\bigr)^2.
\end{equation}
The following weighted energy estimate is the main analytic tool of this
section.

\begin{prop}[Weighted energy estimate]\label{prop:energy}
Suppose that $P$ satisfies
\begin{equation}\label{eq:Peq}
\pa_TP+\pa_X[(\ml V+\ml V_e)P]+P\pa_XP-\pa_X^2P=G
\end{equation}
on $[1,2\nu^{-\beta_2}]$, with $\int_\R G(T,X)\,\md X=0$ for all $T\in[1,2\nu^{-\beta_2}]$,
$\int_\R P(1,X)\,\md X=0$, and with $\ml V_e$ satisfying \eqref{eq:ve} for $m\in\N_0,n\in \{0,1\}$.

 Suppose there exist a constant $N>2$ and a family
$\{a_{m,n}\}\subset\N_0$ with $a_{0,0}=0$ such that, for all $m,n\in\N_0$,
\begin{equation}\label{cod}
	\begin{aligned}
|\pa_T^m\pa_X^nG|&\lesssim(T+X^2)^{-N-3+a_{m,n}},\qquad
|(\pa_T^m\pa_X^nP)(1,X)|\lesssim(1+X^2)^{-N-1+a_{m,n}},
\end{aligned}
\end{equation}
and for $m\in \N_0,n\in \N$ with $n\geq 2$,
\begin{equation*}
	\begin{aligned}
		|\pa_T^m\pa_X^n\ml{V}_e|&\lesssim (T+X^2)^{a_{m,n}}.
	\end{aligned}
\end{equation*}

Set $\ml G=-\frac12\int_{-\infty}^XG(T,X')\,\md X'$.
Then there exists $R=R(m,n)\in\N$, depending only on $\{a_{m,n}\}$, such
that, for all $0\leq r\le N-2R(m,n)$,
\begin{equation}\label{eq:energy}
\begin{aligned}
&\sup_{T\in[1,2\nu^{-\beta_2}]}\int\vph^{2r}|\pa_T^m\pa_X^nP(T,\cdot)|^2\\
&+\int_1^{2\nu^{-\beta_2}}\int\bigl[\vph^{2r-1}|\pa_T^m\pa_X^nP|^2
+\vph^{2r}|\pa_T^m\pa_X^{n+1}P|^2\bigr]\md T\\
&\lesssim\left(\int_1^{2\nu^{-\beta_2}}
\bigl\|\vph^{r+R(m,n)}\pa_T^{\leq m}\pa_X^{\leq n+1}\ml G\bigr\|_{L^2_X}\,\md T\right)^2\\
&\quad+\bigl\|\langle X\rangle^{2r+2R(m,n)+1}
(\pa_T^{\leq m}\pa_X^{\leq n}P)(1,\cdot)\bigr\|_{L^2_X}^2.
\end{aligned}
\end{equation}
Initial time derivatives are determined by \eqref{eq:Peq}; the constants
are independent of $\nu$.
Moreover, there exists a family $\{b_{m,n}\}$, depending only on $\{a_{m,n}\}$,
such that, for $N\ge b_{m,n}$ and $T\in[1,2\nu^{-\beta_2}]$,
\begin{equation}\label{eq:Pdecay}
|\pa_T^m\pa_X^nP|\lesssim(T+X^2)^{-N+b_{m,n}}.
\end{equation}
\end{prop}

To prove Proposition \ref{prop:energy}, we first prove some auxiliary lemmas.

\begin{lemma}[Damping]\label{lem:damping}
For $\vph=T+(X+T^{-\frac12}\Theta_0^{\mathrm I})^2$ and real $r\ge1$, there exists $\tilde T>1$, independent of $\nu$, such that for $T\in[\tilde T,2\nu^{-\beta_2}]$ we have
\begin{equation}\label{eq:damping}
-\pa_X(\ml V\vph^{2r})\gtrsim\vph^{2r-1}\Bigl[|T^{\frac12}X+\Theta_0^{\mathrm I}(T)|
+\zeta T^2\operatorname{sech}^2\Bigl(\frac{ab}{2}\bigl(T^{\frac12}X+\Theta_0^{\mathrm I}(T)\bigr)\Bigr)\Bigr].
\end{equation}
\end{lemma}

\begin{proof}
Set $\mf y(\nu^{\frac12}T,\nu^{\frac34}X)=\nu^{\frac14}\mf Y(T,X)$. From the
leading-order profiles established in Sections 3 and 5, together with the
definition \eqref{eq:V}, we have
\begin{equation*}
\ml V=\zeta U_{0,0}^{\mathrm{SI}}+(1-\zeta)U_{0,0}^{\mathrm{OI}},
\end{equation*}
with
\begin{equation*}
U_{0,0}^{\mathrm{OI}}(T,X)=-a\mf Y(T,X),\qquad
U_{0,0}^{\mathrm{SI}}(T,X)=-abT^{\frac12}\tanh\Bigl(\frac{ab}{2}\bigl(T^{\frac12}X+\Theta_0^{\mathrm I}(T)\bigr)\Bigr).
\end{equation*}
By the rescaled gauge identities of Section~3,
\begin{equation*}
\pa_X\mf Y=\frac1a\mf D^{-2}>0,\qquad
\pm\mf Y(T,X)>bT^{\frac12}\quad\text{for }\pm X>0.
\end{equation*}
Differentiating $\ml V$, we obtain
\begin{equation}\label{eq:dxV}
\begin{aligned}
-\pa_X\ml V
&=(1-\zeta)\mf D^{-2}
+\zeta\frac{(ab)^2}{2}T\operatorname{sech}^2\Bigl(\frac{ab}{2}\bigl(T^{\frac12}X+\Theta_0^{\mathrm I}(T)\bigr)\Bigr)
-\pa_X\zeta\bigl(U_{0,0}^{\mathrm{SI}}-U_{0,0}^{\mathrm{OI}}\bigr)\\
&\geq(1-\zeta)\mf D^{-2}
+\zeta\frac{(ab)^2}{2}T\operatorname{sech}^2\Bigl(\frac{ab}{2}\bigl(T^{\frac12}X+\Theta_0^{\mathrm I}(T)\bigr)\Bigr),
\end{aligned}
\end{equation}
because $-\pa_X\zeta\,(U_{0,0}^{\mathrm{SI}}-U_{0,0}^{\mathrm{OI}})>0$ in the
matching zone and $\pa_X\zeta=0$ elsewhere. Since $|\Theta_0^{\mathrm I}|\lesssim1$ and $|\sqrt T X|\ge\frac12T^{\beta_1}$ wherever $1-\zeta>0$, taking $\tilde T$ large ensures $X(X+T^{-\frac12}\Theta_0^{\mathrm I})>0$ there for $T\ge\tilde T$. Combining \eqref{eq:dxV} with
$\pa_X\vph=2(X+T^{-\frac12}\Theta_0^{\mathrm I})$, we obtain
\begin{equation}\label{eq:dxVvph}
\begin{aligned}
-\pa_X(\ml V\vph^{2r})
&=-\vph^{2r-1}\bigl(2r\ml V\pa_X\vph+\pa_X\ml V\vph\bigr)\\
&\ge\vph^{2r-1}\Bigl[4rabT^{\frac12}\tanh\Bigl(\frac{ab}{2}\bigl(T^{\frac12}X+\Theta_0^{\mathrm I}(T)\bigr)\Bigr)
\bigl(X+T^{-\frac12}\Theta_0^{\mathrm I}\bigr)\\
&\qquad+\zeta\vph\frac{(ab)^2}{2}T\operatorname{sech}^2\Bigl(\frac{ab}{2}\bigl(T^{\frac12}X+\Theta_0^{\mathrm I}(T)\bigr)\Bigr)\Bigr]\\
&\gtrsim\vph^{2r-1}\Bigl[|T^{\frac12}X+\Theta_0^{\mathrm I}(T)|
+\zeta T^2\operatorname{sech}^2\Bigl(\frac{ab}{2}\bigl(T^{\frac12}X+\Theta_0^{\mathrm I}(T)\bigr)\Bigr)\Bigr].
\end{aligned}
\end{equation}
Here the last step uses $y\tanh(\frac{ab}{2}y)\geq c_{\delta}|y|$ for all $|y|\geq \delta$ and $\vph\gtrsim T$, where $c_{\delta}>0$ is a small constant depending on $\delta>0$.

\end{proof}
On $[1,\tilde T]$, Gr\"onwall's inequality yields the corresponding energy estimates with constants independent of $\nu$. For brevity, subsequent energy estimates treat Lemma~\ref{lem:damping} as valid throughout the time interval under consideration, with this finite-time argument understood.

The following lemma is used to ensure the regularity of the inverse of the Cole--Hopf transform.
\begin{lemma}
	Under the assumptions of Proposition~\ref{prop:energy}, there exists a constant $M>0$ such that for all $(T,X)\in [1,2\nu^{-\beta_2}]\times \R$,
	\begin{equation}\label{ich}
		\int_{-\infty}^X P(T,X')\md X'<M.
	\end{equation}
\end{lemma}
\begin{proof}
	Set $Q(T,X)=	\int_{-\infty}^X P(T,X')\md X'$ and $\ml G=-\frac12\int_{-\infty}^XG\,\md X'$. Integrating \eqref{eq:Peq} in $X$, we obtain
	\begin{equation*}
		\begin{aligned}
		\pa_TQ+(\ml{V}+\ml{V}_e)\pa_XQ+\frac12 |\pa_XQ|^2-\pa_X^2Q=-2\ml{G}.
		\end{aligned}
	\end{equation*}
From \eqref{cod}, we have $Q(1,X)\in L^{\infty}_X(\R)$. Set  $Q_{\mathrm{com}}(T,X)=C_1e^{-\frac{2}{\sqrt{T}}}$.  From \eqref{cod} and the assumption $\int G(T,X)\md X=0$, we have
	\begin{equation*}
		\begin{aligned}
			|\ml{G}|\lesssim (T+X^2)^{-N-2}\lesssim T^{-N-2}.
		\end{aligned}
	\end{equation*}
For $T\in [1,2\nu^{-\beta_2}]$, choose $C_1>0$ independent of $\nu$ with $Q_{\mathrm{com}}(1,\cdot)\ge Q(1,\cdot)$ and
		\begin{equation}\label{641}
		\begin{aligned}
		 	\pa_TQ_{\mathrm{com}}&+(\ml{V}+\ml{V}_e)\pa_XQ_{\mathrm{com}}+\frac12 |\pa_XQ_{\mathrm{com}}|^2-\pa_X^2Q_{\mathrm{com}}\\
		 	&=C_1e^{-\frac{2}{\sqrt{T}}}T^{-\frac32}>2|\ml{G}|.
		\end{aligned}
	\end{equation}
Then \eqref{ich} follows from \eqref{641} and the comparison principle.
	
\end{proof}

Now we are in a position to prove Proposition~\ref{prop:energy}.

\begin{proof}[Proof of Proposition \ref{prop:energy}]
Set $W=e^{-\frac12\int_{-\infty}^XP\,\md X'}-1$ and
$\ml G=-\frac12\int_{-\infty}^XG\,\md X'$. A direct computation shows that
$W$ satisfies the linear equation
\begin{equation}\label{eq:W}
\pa_TW+(\ml V+\ml V_e)\pa_XW-\ml G\,W-\pa_X^2W=\ml G.
\end{equation}
For $m=n=0$ and $r\geq0$, we use $\int_{\R}P(1,X)\,\md X=0$
to obtain, by a direct calculation,
\[
\begin{aligned}
\|\vph(1,\cdot)^rW(1,\cdot)\|_{L^2_X}
&\asymp\|\langle X\rangle^{2r}W(1,\cdot)\|_{L^2_X}\\
&\lesssim e^{\frac12\|P(1,\cdot)\|_{L^1_X}}
\|\langle X\rangle^{2r+1}P(1,\cdot)\|_{L^2_X}.
\end{aligned}
\]

For real $1\le r<N$, multiplying \eqref{eq:W} by $\vph^{2r}W$ and
integrating in $X$, we obtain
\begin{equation}\label{eq:energyid}
\frac12\frac{\md}{\md T}\int\vph^{2r}|W|^2
+\int H_0|W|^2+\int\vph^{2r}|\pa_XW|^2
=\int\vph^{2r}\ml G\,W,
\end{equation}
where
\begin{equation*}
H_0=-\frac12\pa_X[\vph^{2r}(\ml V+\ml V_e)]
-\ml G\vph^{2r}-\frac12\pa_X^2(\vph^{2r})-\frac12\pa_T(\vph^{2r}).
\end{equation*}
By \eqref{eq:ve} and the decay of $G$, direct computation gives
\begin{equation*}
\Bigl|-\frac12\pa_X(\ml V_e\vph^{2r})
-\ml G\vph^{2r}-\frac12\pa_X^2(\vph^{2r})-\frac12\pa_T(\vph^{2r})\Bigr|
\lesssim\vph^{2r-1}(1+\frac{|X|}{T}).
\end{equation*}
From Lemma \ref{lem:damping}, arguing separately in the regions
$|T^{\frac12}X+\Theta_0^{\mathrm I}(T)|\gtrsim1$ and $|T^{\frac12}X+\Theta_0^{\mathrm I}(T)|\lesssim1$,
we obtain, after enlarging $\tilde T$ if necessary, for all $T\in[\tilde T,2\nu^{-\beta_2}]$,
\begin{equation}\label{eq:G0lower}
H_0\gtrsim\vph^{2r-1}\Bigl[|T^{\frac12}X+\Theta_0^{\mathrm I}(T)|
+\zeta T^2\operatorname{sech}^2\Bigl(\frac{ab}{2}\bigl(T^{\frac12}X+\Theta_0^{\mathrm I}(T)\bigr)\Bigr)-C\Bigr]
\gg\vph^{2r-1}.
\end{equation}
On $[1,\tilde T]$, Gr\"onwall's inequality gives the corresponding energy bound.
Substituting this into \eqref{eq:energyid} and using the Cauchy--Schwarz
inequality, we obtain
\begin{equation*}
\begin{aligned}
\sup_{T\in[1,2\nu^{-\beta_2}]}\int\vph^{2r}|W(T,\cdot)|^2
&+\int_1^{2\nu^{-\beta_2}}\int\bigl[\vph^{2r-1}|W|^2+\vph^{2r}|\pa_XW|^2\bigr]\md T\\
&\lesssim\Bigl(\int_1^{2\nu^{-\beta_2}}\|\vph^r\ml G\|_{L^2_X}\md T\Bigr)^2+
\|\langle X\rangle^{2r}W(1,\cdot)\|_{L^2_X}^2.
\end{aligned}
\end{equation*}

For $m\in\N_0$ and $n\in\N_0\cup\{-1\}$, set $W_{m,n}=\pa_T^m\pa_X^{n+1}W$.
Differentiating \eqref{eq:W} gives
\begin{equation*}
\begin{aligned}
\pa_TW_{m,n}&+(\ml V+\ml V_e)\pa_XW_{m,n}
+(n+1)\pa_X(\ml V+\ml V_e)W_{m,n}-\ml G\,W_{m,n}-\pa_X^2W_{m,n}\\
&=\pa_T^m\pa_X^{n+1}\ml G\\
&\quad-\sum_{\substack{m_1+m_2=m,\ n_1+n_2=n\\m_1,m_2,n_1\geq0,\ n_2\geq-1\\(m_1,n_1)\notin\{(0,0),(0,1)\}}}
C_m^{m_1}C_{n+1}^{n_1}\pa_T^{m_1}\pa_X^{n_1}(\ml V+\ml V_e)\pa_XW_{m_2,n_2}\\
&\quad+\sum_{\substack{m_1+m_2=m,\ n_1+n_2=n\\m_1,m_2,n_1\geq0,\ n_2\geq-1\\(m_1,n_1)\neq(0,0)}}
C_m^{m_1}C_{n+1}^{n_1}\pa_T^{m_1}\pa_X^{n_1}\ml G\,W_{m_2,n_2}.
\end{aligned}
\end{equation*}
For $n=-1$, multiplication by $\vph^{2r}W_{m,-1}$ and induction on $m$
give, for $1\le r\le N-2R(m,0)$ and sufficiently large $R(m,0)$,
\begin{equation*}
\begin{aligned}
&\sup_{T\in[1,2\nu^{-\beta_2}]}\int\vph^{2r}|W_{m,-1}(T,\cdot)|^2\\
&+\int_1^{2\nu^{-\beta_2}}\int\bigl[\vph^{2r-1}|W_{m,-1}|^2
+\vph^{2r}|\pa_XW_{m,-1}|^2\bigr]\md T\\
&\lesssim\left(\int_1^{2\nu^{-\beta_2}}
\|\vph^{r+R(m,0)}\pa_T^{\leq m}\ml G\|_{L^2_X}\,\md T\right)^2
+\|\langle X\rangle^{2r+2R(m,0)}(\pa_T^{\leq m}W)(1,\cdot)\|_{L^2_X}^2.
\end{aligned}
\end{equation*}
Indeed, after we enlarge $a_{m_1,n_1}$ for $(m_1,n_1)\ne(0,0)$,
the commutator coefficients satisfy
\[
|\pa_T^{m_1}\pa_X^{n_1}(\ml V+\ml V_e)|
+|\pa_T^{m_1}\pa_X^{n_1}\ml G|\lesssim\vph^{a_{m_1,n_1}}.
\]
Using $\vph\ge T\ge1$ and $\int_1^{2\nu^{-\beta_2}}T^{-2}\,\md T\le1$, we have
\begin{align*}
&\left(\int_1^{2\nu^{-\beta_2}}
\|\vph^r\pa_T^{m_1}\pa_X^{n_1}(\ml V+\ml V_e)
\pa_XW_{m_2,n_2}\|_{L^2_X}\,\md T\right)^2\displaybreak[1]\\
&\lesssim\int_1^{2\nu^{-\beta_2}}\int
\vph^{2r+2a_{m_1,n_1}+2}|\pa_XW_{m_2,n_2}|^2\,\md T.
\end{align*}
The $\ml G$ commutators are bounded in the same way, using the lower-order
damping integral at the fixed weight $r+a_{m_1,n_1}+\frac32$.
All weight losses are included in $R$.
For $n\in\N_0$, the term $(n+1)\pa_X(\ml V+\ml V_e)W_{m,n}$ requires a
separate treatment. Let $\Sigma_j=\{X:|X+T^{-\frac12}\Theta(T)|\sim2^j\sqrt T\}$ for $j\ge1$,
$\Sigma_0=\{X:|X+T^{-\frac12}\Theta(T)|\le2\sqrt{T}\}$, and set $\theta=\frac{n+1}{n+2}$. Using
$|\pa_X\ml V|\lesssim T$, we obtain
\begin{equation}\label{eq:dyadic}
\begin{aligned}
\Bigl|\int\vph^{2r}\pa_X\ml V\,|W_{m,n}|^2\Bigr|
&\lesssim\int\vph^{2r+1}|W_{m,n}|^2
\lesssim\sum_{j\ge0}(4^jT)^{2r+1}\|W_{m,n}\|_{L^2_X(\Sigma_j)}^2\\
&\lesssim\vep\|\vph^r\pa_XW_{m,n}\|_{L^2_X}^2
+C(\vep)\|\vph^{r+n+2}W_{m,-1}\|_{L^2_X}^2.
\end{aligned}
\end{equation}
Here localized interpolation includes the cutoff derivatives, whose terms
are absorbed by the same bound. The last term is the damping term for
$W_{m,-1}$ at weight $r+n+\frac52$. The $\ml V_e$ contribution satisfies
the same estimate by \eqref{eq:ve}. Induction first on $m$ and then on $n$
therefore closes after enlarging $R(m,n)$; monotonicity of the weights also
covers $0\le r<1$.
Differentiating the Cole--Hopf transform and using the weighted primitive
estimate bounds the initial terms by the $P$ initial norm in \eqref{eq:energy}.
By \eqref{ich}, $1+W\ge e^{-M/2}>0$, so $P=-2\pa_XW/(1+W)$ gives
\eqref{eq:energy}. Weighted Sobolev embedding applied at spatial orders
$n,n+1$ gives \eqref{eq:Pdecay}, with
$b_{m,n}\ge2\max\{R(m,n),R(m,n+1)\}$.
\end{proof}

Using Proposition \ref{prop:energy}, we obtain the following estimate for
$P_{J,K,0,0}$.

\begin{prop}\label{prop:P000}
There exists a family $\{a_{m,n}^{(0,0)}\}$ such that, for each fixed
$m,n\in\N_0$ and all sufficiently large $J,K\in\N$,
\begin{equation*}
|\pa_T^m\pa_X^nP_{J,K,0,0}|\lesssim(T+X^2)^{-\psi_{0}(J,K)+a_{m,n}^{(0,0)}}.
\end{equation*}
\end{prop}

\begin{proof}
By \eqref{eq:lam0}, \eqref{eq:whP0} and Lemma~\ref{lem:G0}, enlarging
$a_{m,n}$ for $(m,n)\ne(0,0)$ gives
\[
\begin{aligned}
|\lam_0(T)|&\lesssim T^{-\psi_{0}(J,K)-\frac32},\\
|\pa_T^m\pa_X^n\wh G^0|&\lesssim
(T+X^2)^{-\psi_{0}(J,K)-\frac12+a_{m,n}}.
\end{aligned}
\]
The compatible initial traces satisfy \eqref{cod} by \eqref{eq:init0}
and \eqref{eq:whP0}. Apply Proposition~\ref{prop:energy} to $\wh P^0$ with
$\ml V_e+\rho_0$ and
$N=\psi_{0}(J,K)-\frac52>\max\{2,b_{m,n}\}$.
For sufficiently large $J,K$, the required coefficient bounds hold for
all derivatives used, and \eqref{eq:lam0}--\eqref{eq:mu0} give the zero means.
Since $P_{J,K,0,0}=\wh P^0+\rho_0$, \eqref{eq:Pdecay} yields the claim with
$a_{m,n}^{(0,0)}\ge b_{m,n}+\frac52$.
\end{proof}

In subsequent applications, decrease $\psi_{\ell}$ so that the forcing
bound in \eqref{cod} holds before the mean correction with
$N=\psi_{\ell}+\frac52$, and the initial bound with $N=\psi_{\ell}$.
After correction use $N=\psi_{\ell}>\max\{2,b_{m,n}\}$; fixed decreases
preserve the index-function property.

\subsection{Construction of $U_{\ell,0}^{\mathrm{in}}$ for $\ell\in\N$}

We construct $U_{\ell,0}^{\mathrm{in}}$ by induction on $\ell$. Suppose
that $U_{\ell',0}^{\mathrm{in}}$ and $P_{J,K,\ell',0}$ have been
constructed for all $0\le\ell'<\ell$, satisfying \eqref{eq:UellH} and,
for some index functions $\psi_{\ell'}$ and families $\{a_{m,n}^{(\ell')}\}$,
\begin{equation}\label{eq:indhyp}
\begin{aligned}
&|\pa_T^m\pa_X^nP_{J,K,\ell',0}|\lesssim
\nu^{-\frac{\beta_2}{2}\ell'}(T+X^2)^{-\psi_{\ell'}(J,K)+a_{m,n}^{(\ell')}},\\
&U_{\ell',0}^{\mathrm{in}}(1,X)=U_{\ell'}^{\mathrm{in}}(1,X)+\ml O(\nu^{-\frac{\beta_2}2\ell'})\rho(X).
\end{aligned}
\end{equation}
We construct $U_{\ell,0}^{\mathrm{in}}$ to satisfy
\begin{equation}\label{eq:target}
\begin{aligned}
\pa_TU_{\ell,0}^{\mathrm{in}}
&-T^{\frac12}\sum_{\ell_1+\ell_2=\ell-1}\mf S_{\ell_1}\pa_XU_{\ell_2,0}^{\mathrm{in}}
+\frac12\pa_X\sum_{\ell_1+\ell_2=\ell}U_{\ell_1,0}^{\mathrm{in}}U_{\ell_2,0}^{\mathrm{in}}
=\pa_X^2U_{\ell,0}^{\mathrm{in}}.
\end{aligned}
\end{equation}
Set
\begin{equation}\label{eq:decomp}
U_{\ell,0}^{\mathrm{in}}=P_{J,K,\ell,0}+U_{[J],[K],\ell}^{\mathrm{SOI}}.
\end{equation}
For this calculation, abbreviate
$B_\ell=U_{[J],[K],\ell}^{\mathrm{SOI}}$ and $P_\ell=P_{J,K,\ell,0}$.
Then
\[
\begin{aligned}
\frac12\pa_X\sum_{\ell_1+\ell_2=\ell}
&\bigl(U_{\ell_1,0}^{\mathrm{in}}U_{\ell_2,0}^{\mathrm{in}}
-B_{\ell_1}B_{\ell_2}\bigr)
=\pa_X\bigl(U_{0,0}^{\mathrm{in}}P_\ell\bigr)\\
&\quad+\pa_X\sum_{\substack{\ell_1+\ell_2=\ell\\\ell_2<\ell}}
B_{\ell_1}P_{\ell_2}+\frac12\pa_X\sum_{\substack{\ell_1+\ell_2=\ell\\\ell_1,\ell_2>0}}
P_{\ell_1}P_{\ell_2}.
\end{aligned}
\]
The last two terms contain only previously constructed corrections.
Substituting \eqref{eq:decomp} into \eqref{eq:target}, we obtain
\begin{align}
\pa_TP_\ell&+\pa_X(U_{0,0}^{\mathrm{in}}P_\ell)-\pa_X^2P_\ell=G_\ell,\notag\displaybreak[1]\\
G_\ell:=-\Biggl[&\pa_TB_\ell-\pa_X^2B_\ell
-T^{\frac12}\sum_{\ell_1+\ell_2=\ell-1}\mf S_{\ell_1}\pa_XB_{\ell_2}\notag\displaybreak[1]\\
&+\frac12\pa_X\sum_{\ell_1+\ell_2=\ell}B_{\ell_1}B_{\ell_2}
-T^{\frac12}\sum_{\ell_1+\ell_2=\ell-1}\mf S_{\ell_1}\pa_XP_{\ell_2}\notag\displaybreak[1]\\
&+\pa_X\sum_{\substack{\ell_1+\ell_2=\ell\\\ell_2<\ell}}
B_{\ell_1}P_{\ell_2}+\frac12\pa_X\sum_{\substack{\ell_1+\ell_2=\ell\\\ell_1,\ell_2>0}}
P_{\ell_1}P_{\ell_2}\Biggr].
\label{eq:derived}
\end{align}
The grid residual is estimated as in Subsection~6.2, using
\eqref{4.5}, \eqref{4.43}, Lemma~\ref{lem5.8} and \eqref{4.39}.
The remaining terms follow from \eqref{eq:indhyp} and the grid bounds.
Thus, decreasing $\psi_\ell$ as above, for some family
$\{a_{m,n}\}\subset\N_0$ independent of $\nu,J,K$, with $a_{0,0}=0$,
\begin{equation}\label{eq:Gell}
|\pa_T^m\pa_X^nG_\ell|\lesssim
\nu^{-\frac{\beta_2}{2}\ell}(T+X^2)^{-\psi_\ell(J,K)-\frac{11}{2}+a_{m,n}}.
\end{equation}

As in Subsection~6.3, choose
\[
\begin{aligned}
\lam_\ell(T)&=-\int_T^{2\nu^{-\beta_2}}\int_\R G_\ell(T',X)\,\md X\,\md T',\\
\mu_\ell&=\int_\R\bigl(U_\ell^{\mathrm{in}}(1,X)-B_\ell(1,X)\bigr)\,\md X-\lam_\ell(1),
\end{aligned}
\]
and solve
\begin{equation}\label{eq:Pell0}
\begin{aligned}
\pa_TP_{J,K,\ell,0}+\pa_X(U_{0,0}^{\mathrm{in}}P_{J,K,\ell,0})
&-\pa_X^2P_{J,K,\ell,0}=G_\ell,\\
P_{J,K,\ell,0}(1,X)&=U_\ell^{\mathrm{in}}(1,X)-B_\ell(1,X)-\mu_\ell\rho(X).
\end{aligned}
\end{equation}
With $\rho_\ell=\lam_\ell\rho$, the function $P_{J,K,\ell,0}-\rho_\ell$
has zero initial mean and satisfies the same linear equation with zero-mean source
\[
G_\ell-\pa_T\rho_\ell-\pa_X(U_{0,0}^{\mathrm{in}}\rho_\ell)+\pa_X^2\rho_\ell.
\]

\begin{prop}[Energy estimate for the linearized equation]\label{prop:linenergy}
Suppose that $P$ satisfies
\begin{equation*}
\pa_TP+\pa_X[(\ml V+\ml V_e)P]-\pa_X^2P=G
\end{equation*}
on $[1,2\nu^{-\beta_2}]$, with $\int_\R G(T,X)\,\md X=0$ for all $T\in[1,2\nu^{-\beta_2}]$,
$\int_\R P(1,X)\,\md X=0$, and with $\ml V_e$ satisfying the coefficient bounds of Proposition~\ref{prop:energy}.
Suppose there exist a constant $N>2$ and a family $\{a_{m,n}\}\subset\N_0$ with $a_{0,0}=0$
such that, for all $m,n\in\N_0$,
\begin{equation*}
|\pa_T^m\pa_X^nG|\lesssim(T+X^2)^{-N-3+a_{m,n}},\qquad
|(\pa_T^m\pa_X^nP)(1,X)|\lesssim(1+X^2)^{-N-1+a_{m,n}}.
\end{equation*}
Then \eqref{eq:energy} and \eqref{eq:Pdecay} hold, with initial time
derivatives determined by the linearized equation. The proof uses
$W=-\frac12\int_{-\infty}^XP(T,X')\,\md X'$ and the same weighted induction.
\end{prop}

\begin{prop}\label{prop:Plellh}
There exists a family $\{a_{m,n}^{(\ell)}\}\subset\N_0$, independent of
$\nu,J,K$, such that, for each fixed $m,n\in\N_0$ and all sufficiently large $J,K\in\N$,
\begin{equation*}
|\pa_T^m\pa_X^nP_{J,K,\ell,0}|\lesssim
\nu^{-\frac{\beta_2}{2}\ell}(T+X^2)^{-\psi_\ell(J,K)+a_{m,n}^{(\ell)}}.
\end{equation*}
\end{prop}
\begin{proof}
By \eqref{eq:Gell} and Lemma~\ref{lem:52},
\[
|\lam_\ell(T)|\lesssim\nu^{-\frac{\beta_2}{2}\ell}T^{-\psi_\ell(J,K)-4},
\qquad |\mu_\ell|\lesssim\nu^{-\frac{\beta_2}{2}\ell}.
\]
By \eqref{eq:ve} and Proposition~\ref{prop:P000}, the coefficient
$U_{0,0}^{\mathrm{in}}-\ml V$ satisfies the required bounds for large $J,K$.
The compact support of $\rho$ and \eqref{eq:Gell} give \eqref{cod}
for the normalized corrected source, with $N=\psi_\ell(J,K)$.
Lemma~\ref{lem:52} and \eqref{eq:Pell0} give the compatible initial traces;
outside a fixed compact set they equal those of $U_\ell^{\mathrm{in}}-B_\ell$.
Apply Proposition~\ref{prop:linenergy} to
$\nu^{\frac{\beta_2}{2}\ell}(P_{J,K,\ell,0}-\rho_\ell)$ with
$N>\max\{2,b_{m,n}\}$, and add back $\rho_\ell$.
\end{proof}

\begin{prop}\label{prop:UellH}
Let $U_{\ell,0}^{\mathrm{in}}$ be given by \eqref{eq:decomp}. Then it
satisfies \eqref{eq:target}, with initial data
\[
U_{\ell,0}^{\mathrm{in}}(1,X)=U_\ell^{\mathrm{in}}(1,X)-\mu_\ell\rho(X),
\qquad |\mu_\ell|\lesssim\nu^{-\frac{\beta_2}{2}\ell}.
\]
\end{prop}
\begin{proof}
This follows from \eqref{eq:derived} and \eqref{eq:Pell0}.
\end{proof}

Together with Subsection~6.3, this completes the induction and proves
Proposition~\ref{thm:main}. For any fixed $r\ge0$ and finitely many
derivative orders, choosing $\psi_\ell(J,K)>r+\max a_{m,n}^{(\ell)}$ gives
\[
\nu^{\frac{\beta_2}{2}\ell}|\pa_T^m\pa_X^nP_{J,K,\ell,0}|
\lesssim(T+X^2)^{-r}
\]
uniformly in $\nu$ for all those derivatives.

\section{Approximation of $U^{(\nu)}$}
In this section we assemble the three-layer approximation
$U^{\mathrm{app}}_{J,K,L}$ to $U^{(\nu)}$ from the outer, shock-profile,
and inner-grid expansions of Sections~4--6, and we estimate the error
$W=U^{(\nu)}-U^{\mathrm{app}}_{J,K,L}$ on the whole interval
$[0,T_0]$, with $T_0=\ol\vep\nu^{-1/2}$. Section~7.1 treats the
short-time post-formation regime $T\in[0,1]$: we approximate $U^{(\nu)}$
by $U_{K,L}^{\mathrm{app}}$ and establish the weighted estimate in
Proposition~\ref{prop6.1}. Section~7.2 treats the
shock-development regime $T\in[1,T_0]$: we introduce the three zones and
the three-layer ansatz $U^{\mathrm{app}}_{J,K,L}$, derive the equation
for the error $W$ with residual $E_1$, control its initial data and
residual, and prove the pointwise estimate of Proposition~\ref{prop6.4} by weighted
energy estimates. The section closes with the proof of the main theorem,
which combines the two regimes and yields part (ii) of
Theorem~\ref{thm}.
	Recall that, for $T\in[0,T_0]$, $U^{(\nu)}$ satisfies
\begin{equation*}
	\begin{aligned}
		\pa_TU^{(\nu)}-\dot{\mf{S}}\pa_XU^{(\nu)}+U^{(\nu)}\pa_XU^{(\nu)}=\pa_X^2U^{(\nu)}.
	\end{aligned}
\end{equation*}
\subsection{Approximation of $U^{(\nu)}$ at shock formation}
Before shock formation, the viscous Burgers equation reads
\begin{equation*}
	\begin{aligned}
		\pa_tu^{(\nu)}+u^{(\nu)}\pa_xu^{(\nu)}=\nu\pa_x^2u^{(\nu)}.
	\end{aligned}
\end{equation*}
Following \cite{CG}, we define
\begin{equation*}
	\begin{aligned}
	u_{K,L}^{\mathrm{app}}:=(1-\xi)u_{[K]}^{\mathrm{out}}+\xi u_{[L]}^{\mathrm{in}}.
	\end{aligned}
\end{equation*}
See \eqref{0.20} and  \cite{CG} for more details.
Set $v=u^{(\nu)}-u_{K,L}^{\mathrm{app}}$. Then $v$ satisfies
\begin{equation*}
	\begin{aligned}
	&\pa_tv+\pa_x(u_{K,L}^{\mathrm{app}}v)+v\pa_xv-\nu\pa_x^2v=-(\pa_tu_{K,L}^{\mathrm{app}}+u_{K,L}^{\mathrm{app}}\pa_xu_{K,L}^{\mathrm{app}}-\nu\pa_x^2u_{K,L}^{\mathrm{app}}) :=E_0\\
	&v(t_{\mathrm{in}},x)=0.
	\end{aligned}
\end{equation*}
For convenience, introduce $U_{K,L}^{\mathrm{app}}(T,X)=\nu^{-\frac14}u_{K,L}^{\mathrm{app}}(\nu^{\frac12}T,\nu^{\frac34} X)$, where $u_{K,L}^{\mathrm{app}}$ is defined in the $(t,\x)$ coordinates.
Adapting the proof in \cite[Section~5]{CG} yields the following estimate;
the proof is deferred to Appendix~C.
\begin{prop}\label{prop6.1}
	For $\vep<\frac16$ and $\beta\in(\frac1{2+\vep},\frac12)$, there exist constants $C_1>0$ and $0<\nu_0\le\min\{1,\ol\vep^2\}$ such that, for all $K,L\in\N_0$ and $0<\nu\le\nu_0$, we have
	\begin{equation*}
		\begin{aligned}
			\left\|(1+X^2)\bigl(U^{(\nu)}-U_{K,L}^{\mathrm{app}}\bigr)\right\|_{L^{\infty}_TH^3_X([0,1]\times \R)}
\lesssim \nu^{\min\{(1-2\beta)K,L\beta/3\}-C_1}.
		\end{aligned}
	\end{equation*}
\end{prop}
\subsection{Approximation of $U^{(\nu)}$ in shock development}
In the shock-development regime we use the following three zones:
\begin{equation}\label{6.7}
	\begin{aligned}
	O_{\mathrm{shock}}&=\{(T,X):\nu^{-\beta_2}< T\leq T_0,\ |X|T^{\frac12-\beta_1}< 2\},\\
	O_{\mathrm{in}}&=\{(T,X):0\leq T< 2\nu^{-\beta_2},\ |X|< 2\nu^{-\beta_3}\},\\
	O_{\mathrm{out}}&=\{(T,X):0\leq T< 2\nu^{-\beta_2},\ |X|>\nu^{-\beta_3}\}\\
	&\qquad\cup\{(T,X):\nu^{-\beta_2}< T\leq T_0,\ |X|T^{\frac12-\beta_1}> 1\},
	\end{aligned}
\end{equation}
where $\beta_1,\beta_2,\beta_3\in (0,\frac1{100}).$ Let $\varsigma_o+\varsigma_i+\varsigma_s=1$ be the associated partition of unity, even in $X$.
For small $\nu$, choose $\varsigma_s(1,\cdot)=0$ and $\varsigma_i(1,\cdot)$
to agree with the rescaled cutoff defining $U_{K,L}^{\mathrm{app}}$,
with $\varsigma_i(1,\cdot)=1$ on $\operatorname{supp}\rho$. Set
\begin{equation*}
	\begin{aligned}
	U^{\mathrm{app}}_{J,K,L}=\varsigma_oU_{[K]}^{\mathrm{out}}+\varsigma_sU_{[J]}^{\mathrm{shock}}+\varsigma_iU_{[L],0}^{\mathrm{in}},
	\end{aligned}
\end{equation*}
where
\begin{equation*}
	\begin{aligned}
	U_{[L],0}^{\mathrm{in}}:=\sum\limits_{\ell=0}^L\nu^{\frac{\ell}4}U_{\ell,0}^{\mathrm{in}}.
	\end{aligned}
\end{equation*}
Set $W=U^{(\nu)}-U^{\mathrm{app}}_{J,K,L}$. Then $W$ satisfies
\begin{equation}\label{6.8}
	\begin{aligned}
		\pa_TW-\dot{\mf{S}}\pa_XW+\pa_X(U^{\mathrm{app}}_{J,K,L}W)+W\pa_XW-\pa_X^2W=E_1,
	\end{aligned}
\end{equation}
where
\begin{equation*}
	\begin{aligned}
	E_1=-[\pa_TU^{\mathrm{app}}_{J,K,L}-\dot{\mf{S}}\pa_XU^{\mathrm{app}}_{J,K,L}+U^{\mathrm{app}}_{J,K,L}\pa_XU^{\mathrm{app}}_{J,K,L}-\pa_X^2U^{\mathrm{app}}_{J,K,L}].
	\end{aligned}
\end{equation*}
In what follows, we assume $\beta=\frac12-\frac23 \beta_3$, where $\beta$ is chosen in Proposition~\ref{prop6.1}.
\begin{lemma} For any $M>1$ and sufficiently large $J,K,L$, there exists $\ol{\mu}(\nu)\in \R$  such that
	\begin{equation}\label{6.10}
		\begin{aligned}
			\left\|(1+X^2)\bigl[W(1,X)-\ol{\mu}(\nu)\rho(X)\bigr]\right\|_{H^3_X(\R)}\lesssim \nu^{M},
		\end{aligned}
	\end{equation}
where $\rho\in C_c^{\infty}(\R)$ is chosen in Proposition~{\ref{thm:main}}, and $|\ol{\mu}(\nu)|\lesssim1$ uniformly as $\nu\to0$.
\end{lemma}
\begin{proof}
Writing
\begin{equation}\label{6.11}
	\begin{aligned}
	W&=U^{(\nu)}-U^{\mathrm{app}}_{J,K,L}\\
	&=U^{(\nu)}-U_{K,L}^{\mathrm{app}}+U_{K,L}^{\mathrm{app}}-U^{\mathrm{app}}_{J,K,L},
	\end{aligned}
\end{equation}
choose $\ol\mu(\nu)=\sum_{\ell=0}^L\nu^{\ell/4}\mu_\ell$, where $\mu_\ell$
is the initial correction from Section~6. The cutoff choice and
Proposition~\ref{prop:UellH}, together with \eqref{eq:init0}, give
\[
U_{K,L}^{\mathrm{app}}(1,X)-U^{\mathrm{app}}_{J,K,L}(1,X)
=\ol\mu(\nu)\rho(X),\qquad
|\ol\mu(\nu)|\lesssim\sum_{\ell=0}^L\nu^{\ell(\frac14-\frac{\beta_2}{2})}\lesssim1.
\]
Thus \eqref{6.10} follows from \eqref{6.11} and Proposition~\ref{prop6.1}.
\end{proof}

\begin{lemma}
	For any $M>0$, there exist sufficiently large $J,K,L$ such that for $T\in [1,T_0]$,
		\begin{equation}\label{6.13}
		\begin{aligned}
			|\pa_T^{\leq 4}\pa_X^{\leq 4}E_1|\lesssim \nu^{M}(T+X^2)^{-M}.
		\end{aligned}
	\end{equation}
\end{lemma}

\begin{proof}
	By Propositions~\ref{prop3.1}, \ref{shockeq} and \ref{prop:UellH}, the truncated expansions satisfy
\begin{equation*}
	\begin{aligned}
	&	\pa_TU_{[K]}^{\mathrm{out}}-\dot{\mf{S}}\pa_XU_{[K]}^{\mathrm{out}}
		+U_{[K]}^{\mathrm{out}}\pa_XU_{[K]}^{\mathrm{out}}-\pa_X^2U_{[K]}^{\mathrm{out}}=\frac 12\pa_X\sum\limits_{\substack{k_1+k_2>K\\k_1,k_2\leq K}}U_{k_1}^{\mathrm{out}}U_{k_2}^{\mathrm{out}}-\pa_X^2U_{K}^{\mathrm{out}},\\
	&\pa_TU_{[J]}^{\mathrm{shock}}-\dot{\mf{S}}\pa_XU_{[J]}^{\mathrm{shock}}+U_{[J]}^{\mathrm{shock}}\pa_XU_{[J]}^{\mathrm{shock}}-\pa_X^2U_{[J]}^{\mathrm{shock}}\\
	&\quad =\pa_TU_{J}^{\mathrm{shock}}+\frac12\pa_X\sum_{\substack{j_1+j_2>J\\j_1,j_2\leq J}}U^{\mathrm{shock}}_{j_1}U^{\mathrm{shock}}_{j_2}+\sum_{j_1+j_2=J-1}\pa_T\Big(\tfrac{\Theta-\wh{\Theta}^{\mathrm{S}}_{[j_1]}}{T^{\frac12}}\Big)\pa_XU_{j_2}^{\mathrm{shock}},
		\end{aligned}
	\end{equation*}
	and
\begin{equation*}
	\begin{aligned}	
&\pa_TU_{[L],0}^{\mathrm{in}}-\dot{\mf S}\pa_XU_{[L],0}^{\mathrm{in}}
+U_{[L],0}^{\mathrm{in}}\pa_XU_{[L],0}^{\mathrm{in}}-\pa_X^2U_{[L],0}^{\mathrm{in}}\\
&\quad =\frac12\pa_X\sum_{\substack{\ell_1+\ell_2>L\\\ell_1,\ell_2\le L}}
\nu^{\frac{\ell_1+\ell_2}{4}}U_{\ell_1,0}^{\mathrm{in}}U_{\ell_2,0}^{\mathrm{in}}-\sum_{\ell_1+\ell_2=L}\nu^{\frac{\ell_2}{4}}
\bigl(\dot{\mf S}-T^{\frac12}\mf S_{[\ell_1]}\bigr)\pa_XU_{\ell_2,0}^{\mathrm{in}}.
\end{aligned}
\end{equation*}
Notice the definition of $\wh{\Theta}^{\mathrm{S}}_{[j_1]}$ in \eqref{3.61}.
The terms involving derivatives of the partition are controlled by the
matching estimates in Propositions~\ref{prop3.9}, \ref{prop4.3},
\ref{prop:SI-reconstruction} and \ref{prop4.9}.
On these overlaps, $T+X^2\gtrsim\nu^{-\min\{\beta_2,2\beta_3\}}$, so
Proposition~\ref{thm:main} makes the correction terms arbitrarily small.
Choose $L$ large for the inner and speed truncation tails, and then $J,K$
large for the outer and shock tails and the index-function decay.
Together with Propositions~\ref{prop3.8}, \ref{prop4.2},
\ref{prop:SI-estimates} and Lemma~\ref{lem5.8}, these bounds give
\eqref{6.13}, after absorbing the fixed derivative and weight losses.

\end{proof}
\begin{prop}\label{prop6.4}
There exists a constant $\delta>0$ such that, for all  $\beta_1,\beta_2,\beta_3\in(0,\frac1{100})$, we have for sufficiently large $J,K,L$ and all $T\in[1,T_0]$, 
	\begin{equation}\label{6.18}
		\begin{aligned}
			|W(T,X)|\lesssim 1+\sqrt{T}e^{-\delta T^{\frac12}|X|}.
		\end{aligned}
	\end{equation}
\end{prop}
The proof uses the following explicit shock-profile correction.
\begin{lemma}\label{lem6.5}
	Let $F(y)=-\mathsf{c}\tanh\frac{\mathsf{c}y}{2}$ for a constant $\mathsf{c}>0$. For $r>0$, define
		\begin{equation*}
		\begin{aligned}
			\vph_r(\mathsf{c},y)=\frac{2\mathsf{c}(r-1)e^{\mathsf{c}y}}{(e^{\mathsf{c}y}+1)(e^{\mathsf{c}y}+r)}.
		\end{aligned}
	\end{equation*}
	Then we have
	\begin{equation*}
		\begin{aligned}
		\pa_y(F\vph_r)+\vph_r\pa_y\vph_r-\pa_y^2\vph_r=0.
		\end{aligned}
	\end{equation*}
Moreover, for all $r\in(0,+\infty)$ and $y\in\R$, we have
	\begin{equation*}
	\begin{aligned}
		\pa_y\vph_r+\pa_yF=-2\frac{\mathsf{c}^2re^{\mathsf{c}y}}{(e^{\mathsf{c}y}+r)^2},
	\end{aligned}
\end{equation*}
and
\begin{equation}\label{6.22}
	\begin{aligned}
	\int_{\R} \vph_r(\mathsf{c},y)\md y=2\ln r.
	\end{aligned}
\end{equation}
\end{lemma}
\begin{proof}
The identities follow by direct differentiation and integration.
\end{proof}

\begin{proof}[Proof of Proposition~\ref{prop6.4}]
	We first prove the estimate on $T\in[1,\nu^{-\beta_2}]$. We decompose
	$W=W_1+W_2+W_3$. First, we choose a function $W_1$ defined on $[1,T_0]\times\R$ such that
	$|\pa_T^m\pa_X^nW_1(T,X)|\lesssim\nu^M(T+X^2)^{-M}$ for $m,n\leq4$ and
	sufficiently large $M$, and
		\begin{equation}\label{6.23}
			\begin{aligned}
				\int \pa_TW_1=\int E_1.
			\end{aligned}
		\end{equation}
	Next, we choose
	$W_2(T,X)=\vph_r(abT^{\frac12},\wh X)$ and
	$F(T,X)=-abT^{\frac12}\tanh(\frac{abT^{\frac12}}{2}\wh X)$ for some $r$,
	such that
		\begin{equation*}
		\begin{aligned}
		\int W_2(1,X)\md X=\int W(1,X)\md X-\int W_1(1,X)\md X.
		\end{aligned}
	\end{equation*}
	Here $\wh X=X+T^{-\frac12}\Theta(T)$.
	From \eqref{6.22}, we obtain
		\begin{equation}\label{6.25}
		\begin{aligned}
			\int \pa_TW_2(T,X)\md X=0.
		\end{aligned}
	\end{equation}
Then $W_3$ satisfies
	\begin{equation}\label{6.26}
		\begin{aligned}
			\pa_TW_3+\pa_X(\ml{V}W_3)+W_3\pa_XW_3-\pa_X^2W_3=G_1+G_2,
		\end{aligned}
	\end{equation}
where 
	\begin{equation*}
		\begin{aligned}
		\ml{V}&=U^{\mathrm{app}}_{J,K,L}-\dot{\mf{S}}+W_1+W_2\\
		G_1&=E_1-\pa_TW_1-\pa_X(U^{\mathrm{app}}_{J,K,L}W_1)+\dot{\mf{S}}\pa_XW_1-\pa_X(W_1W_2)-W_1\pa_XW_1+\pa_X^2W_1\\
		G_2&=-\pa_TW_2-\pa_X(FW_2)-W_2\pa_XW_2+\pa_X^2W_2-\pa_X[(U^{\mathrm{app}}_{J,K,L}-F-\dot{\mf{S}})W_2]\\
		&=-\pa_TW_2-\pa_X[(U^{\mathrm{app}}_{J,K,L}-F-\dot{\mf{S}})W_2].\\
		\end{aligned}
	\end{equation*}
We note that 
 \begin{equation}\label{6.28}
 	\begin{aligned}
 		\dot{\mf{S}}&=\ml{O}(\nu^{\frac14}T)\\
 		\varsigma_iU_{[L],0}^{\mathrm{in}}&=\varsigma_iU_{0,0,0}^{\mathrm{SOI}}
+\ml{O}((T+X^2)^{\beta_1-\frac32})+\ml{O}(\mf D^{-3})\\
&\quad+\ml{O}(\nu^{\frac14}\mf D^2)+\ml{O}(\nu^{\frac14-\beta_2})\\
 		\varsigma_oU_{[K]}^{\mathrm{out}}&=\varsigma_oU_{0,0}^{\mathrm{OI}}+\ml{O}(\min\{\nu^{\frac14}\mf{D}^2,\mf{D}\})+\ml{O}(\mf{D}^{-3})\\
 		\varsigma_sU_{[J]}^{\mathrm{shock}}&=\varsigma_sU_{0,0}^{\mathrm{SI}}+\ml{O}(\nu^{\frac14-\beta_2})+\ml{O}((T+X^2)^{\beta_1-\frac32}),
 	\end{aligned}
 \end{equation}
 and
 \begin{equation}\label{6.29}
 	\begin{aligned}
 		\pa_X[\varsigma_i(U_{[L],0}^{\mathrm{in}}-U_{0,0,0}^{\mathrm{SOI}})]
&=\ml{O}((T+X^2)^{-1})+\ml{O}(\mf D^{-6})\\
&\quad+\ml{O}(\nu^{\frac14}\mf D^{-1})+\ml{O}(\nu^{\frac14-\frac32\beta_2})\\
 		\pa_X[\varsigma_o(U_{[K]}^{\mathrm{out}}-U_{0,0}^{\mathrm{OI}})]&=\ml{O}(\nu^{\frac14}\mf{D}^{-1})+\ml{O}(\mf{D}^{-6})\\
 		\pa_X[\varsigma_s(U_{[J]}^{\mathrm{shock}}-U_{0,0}^{\mathrm{SI}})]&=\ml{O}(\nu^{\frac14-\frac32\beta_2})+\ml{O}((T+X^2)^{-1}).
 	\end{aligned}
 \end{equation}
The inner bounds follow from \eqref{eq:SOI}, Proposition~\ref{prop4.2},
Proposition~\ref{prop:SI-estimates} and Proposition~\ref{thm:main}.
On $\operatorname{supp}\varsigma_i$, $\nu^{1/4}\mf D\ll1$; the
outer-grid terms with $\ell\ge1$ sum to $\ml O(\nu^{1/4}\mf D^2)$,
and their $X$ derivatives to $\ml O(\nu^{1/4}\mf D^{-1})$. Introduce 
$$\wh{X}_a=\nu^{-\frac34}\arctan(\nu^{\frac34}\wh{X}).$$
 From \eqref{6.10}, \eqref{6.23}--\eqref{6.26}, \eqref{6.28}--\eqref{6.29}, \eqref{64} and \eqref{eq:SOI}, we obtain that, for $\beta_1,\beta_2\in (0,\frac1{100})$, there exists a constant $\delta>0$ such that
		\begin{equation}\label{6.30}
		\begin{aligned}
		\int_{\R}W_3(1,X)\,\md X&=0,\\
		\left\|W_3(1,\cdot)\right\|_{H^1_X(\R)}+\left\|(\pa_TW_3)(1,\cdot)\right\|_{L^2_X(\R)}&\lesssim1,\\
		\sum_{m=0}^2\left\|\int_{-\infty}^X(\pa_T^mW_3)(1,X')\,\md X'\right\|_{L^2_X(\R)}&\lesssim1,\\
		\int G_i(T,X)&=0,\quad \text{ for }i=1,2,\\
		|\pa_T^m\pa_X^nG_1|&\lesssim  \nu^M(T+X^2)^{-M},\quad \text{for } m,n \leq 2,\\
		|\pa_T^m\pa_X^nG_2|&\lesssim  T^{\frac{n-1}{2}-m}e^{-\delta |T^{\frac12}\wh{X}|},\quad \text{for } m,n \leq 2,\\
		|\pa_T^m\int_{-\infty}^X G_2(T,X')\md X'|&\lesssim T^{-1-m}e^{-\delta |T^{\frac12}\wh{X}|}\quad \text{for } m\leq 2.\\
		\end{aligned}
	\end{equation}
Let $G=G_1+G_2$,
$\ml W(T,X)=\int_{-\infty}^XW_3(T,X)\md X$ and
$\ml G_i=\int_{-\infty}^XG_i(T,X')\md X'$ for $i=1,2$ with
$\ml G=\ml G_1+\ml G_2$. Then $\ml W$ satisfies
		\begin{equation}\label{eqw}
		\begin{aligned}
		\pa_T\ml{W}+\ml{V}\pa_X\ml{W}+\frac12|\pa_X\ml{W}|^2-\pa_X^2\ml{W}=\ml{G}.
		\end{aligned}
	\end{equation}
Introduce $\Phi(z)=\cosh(\kappa z)-1\geq 0$, where $\kappa>0$ will be chosen later. For $s\in(0,1]$, we multiply \eqref{eqw} by $(T+\wh{X}_a^2)^{2s}\Phi'(\ml{W})$ to obtain
\begin{equation}\label{7.30}
\begin{aligned}
\frac{\md}{\md T}&\int (T+\wh X_a^2)^{2s}\Phi(\ml W)
+\int H_0\Phi(\ml W)\\
&+\int (T+\wh X_a^2)^{2s}
\left[\Phi''(\ml W)+\frac12\Phi'(\ml W)\right]|\pa_X\ml W|^2=\int (T+\wh X_a^2)^{2s}\ml G\Phi'(\ml W),
\end{aligned}
\end{equation}
where
\begin{equation*}
\begin{aligned}
H_0=-\pa_X\bigl[\ml V(T+\wh X_a^2)^{2s}\bigr]
-\pa_T(T+\wh X_a^2)^{2s}
-\pa_X^2(T+\wh X_a^2)^{2s}.
\end{aligned}
\end{equation*}
Fix $\kappa\ge1$ and $0<s\le\frac34$. Then
$\Phi''+\frac12\Phi'\ge\frac12\Phi''$,
$|\Phi'|\lesssim\Phi+\sqrt\Phi$, and $(\Phi')^2\le2\Phi\Phi''$.
Let $\Gamma=-\pa_X(F+W_2)$. By Lemma~\ref{lem6.5} and the proof of
\eqref{eq:G0lower}, with the bounded-time contribution included in
$CT^{-2}$, we have on $[1,2\nu^{-\beta_2}]$
\begin{equation}\label{6.5}
\begin{aligned}
H_0\ge c\Gamma(T+\wh X_a^2)^{2s}
-C\bigl(\nu^{\beta_2}+T^{-2}\bigr)(T+\wh X_a^2)^{2s}.
\end{aligned}
\end{equation}
By \eqref{6.30} and $\Gamma\gtrsim Te^{-C|\sqrt T\wh X|}$ from
Lemma~\ref{lem6.5}, splitting at
$|\sqrt T\wh X|=C^{-1}\log(2+T)$ gives, after decreasing $\delta>0$,
\begin{equation*}
\begin{aligned}
\|\ml G_1\|_{L^\infty_X}
+\|(T+\wh X_a^2)^s\ml G_1\|_{L^2_X}&\lesssim T^{-2},\\
|\ml G_2|&\le\vep\Gamma+C(\vep)T^{-1-\delta}.
\end{aligned}
\end{equation*}
Using $\Gamma\gtrsim T$ on $|\sqrt T\wh X|\le1$, apply the fundamental
theorem of calculus to $\sinh(\kappa\ml W/2)$. Then Cauchy--Schwarz,
\eqref{6.30} and Young's inequality give
\begin{equation*}
\begin{aligned}
&\left|\int(T+\wh X_a^2)^{2s}\ml G_2\Phi'(\ml W)\,\md X\right|\\
&\quad\le\vep\int(T+\wh X_a^2)^{2s}
\left[\Gamma\Phi(\ml W)+\Phi''(\ml W)|W_3|^2\right]\md X\\
&\qquad+C(\vep)T^{-1-\delta}\int(T+\wh X_a^2)^{2s}\Phi(\ml W)\,\md X
+C(\vep)T^{2s-7/2}.
\end{aligned}
\end{equation*}
Combining these bounds with \eqref{7.30}, choosing $\vep$ sufficiently
small and using $2s-7/2\le-2$, we obtain
\begin{equation*}
\begin{aligned}
\frac{\md}{\md T}&\int(T+\wh X_a^2)^{2s}\Phi(\ml W)\,\md X\\
&+c\int(T+\wh X_a^2)^{2s}
\left[\Gamma\Phi(\ml W)+\Phi''(\ml W)|W_3|^2\right]\md X\\
&\le C\bigl(\nu^{\beta_2}+T^{-1-\delta}+T^{-2}\bigr)
\int(T+\wh X_a^2)^{2s}\Phi(\ml W)\,\md X+CT^{-2}.
\end{aligned}
\end{equation*}
The coefficient on the right has a uniformly bounded time integral.
By \eqref{6.10} for the initial weighted bound (the extra weight on its small remainder costs at most $\nu^{-3/8}$), Gr\"onwall's inequality gives
\begin{equation}\label{7.38}
\begin{aligned}
&\sup_{1\le T\le\nu^{-\beta_2}}
\int(T+\wh X_a^2)^{2s}\Phi(\ml W(T,X))\,\md X\\
&\quad+\int_1^{2\nu^{-\beta_2}}\int(T+\wh X_a^2)^{2s}
\left[\Gamma\Phi(\ml W)+\Phi''(\ml W)|W_3|^2\right]\md X\,\md T
\lesssim1.
\end{aligned}
\end{equation}
Since $\Phi(z)\ge\kappa^2z^2/2$ and $\Phi''(z)\ge\kappa^2$,
\eqref{6.37} follows for $s\in[0,\frac34]$.
We now estimate $W_3$. From \eqref{6.26}, we have
\begin{equation}\label{6.43}
	\begin{aligned}
		\frac12\frac{\md}{\md T}\int |W_3|^2+\frac12\int\pa_X\ml{V}|W_3|^2+\int |\pa_XW_3|^2&=\int \pa_X\ml{G}W_3\\
		&\leq \left\|\ml{G}\right\|_{L^2}\left\|\pa_XW_3\right\|_{L^2}.
	\end{aligned}
\end{equation}
To bound the $\pa_X\ml V$ term, we notice that
\begin{equation*}
	\begin{aligned}
		\int|\pa_X\ml{V}||W_3|^2&\lesssim \int\Gamma|W_3|^2+T^{-1}\int |W_3|^2\md X.
	\end{aligned}
\end{equation*}
To treat $\int\Gamma|W_3|^2$, we decompose $\R=\bigcup\limits_{n\in \Z} (\frac{n-\frac12}{\sqrt{T}},\frac{n+\frac32}{\sqrt{T}}):=\bigcup\limits_{n\in \Z} I_n$ and introduce an associated partition of unity $1=\sum\limits_{n\in \Z}\chi_n$, whose cutoff functions have derivatives of order $\ml{O}(\sqrt{T})$ uniformly in $n$. Then we obtain for sufficiently small $\vep>0$,
\begin{equation*}
	\begin{aligned}
	\int\Gamma|W_3|^2&\lesssim \sum\limits_{n\in \Z}\Gamma(T,\frac{n}{\sqrt{T}})\int|\chi_nW_3|^2\\
	&\lesssim  \sum\limits_{n\in \Z}\Gamma(T,\frac{n}{\sqrt{T}})\left\|\pa_X(\chi_nW_3)\right\|_{L^2(I_n)}\left\|\ml{W}\right\|_{L^2(I_n)}\\
	&\lesssim  \sum\limits_{n\in \Z}\Gamma(T,\frac{n}{\sqrt{T}})\left\|\pa_XW_3\right\|_{L^2(I_n)}\left\|\ml{W}\right\|_{L^2(I_n)}\\
	&+\sum\limits_{n\in \Z}\Gamma(T,\frac{n}{\sqrt{T}})\sqrt{T}\left\|W_3\right\|_{L^2(I_n)}\left\|\ml{W}\right\|_{L^2(I_n)}\\
	&\lesssim \vep\left\|\pa_XW_3\right\|_{L^2(\R)}^2+ \vep \left\|\Gamma^{\frac12}W_3\right\|_{L^2}^2+ C(\vep)\left\|\Gamma^{\frac12}(T+\wh{X}_a^2)^{\frac12}\ml{W}\right\|_{L^2(\R)}^2.
	\end{aligned}
\end{equation*}
Hence
\begin{equation*}
	\begin{aligned}
		\int|\pa_X\ml{V}||W_3|^2
		&\lesssim \vep\left\|\pa_XW_3\right\|_{L^2(\R)}^2+C(\vep)\left\|\Gamma^{\frac12}(T+\wh{X}_a^2)^{\frac12}\ml{W}\right\|_{L^2}^2\\
		&+T^{-1}\int |\pa_X\ml{W}|^2\md X.\\
	\end{aligned}
\end{equation*}

Note that
\begin{equation}\label{7.41}
	\begin{aligned}
	\left\|\ml{W}\right\|_{L^{\infty}}\lesssim T^{-\frac s2}\left\|(T+\wh{X}_a^2)^{s}\ml{W}\right\|_{L^{2}}^{\frac12}\left\|{W}_3\right\|_{L^{2}}^{\frac12}.
	\end{aligned}
\end{equation}
From \eqref{7.38}--\eqref{7.41}, we obtain for $T\in [1,\nu^{-\beta_2}]$ and $s\in[0,\frac34]$,
\begin{equation}\label{6.37}
	\begin{aligned}
		\left\|(T+\wh{X}_a^2)^s\ml{W}(T,\cdot)\right\|_{L^2_X}^2&+\int_1^{T}\left[\left\|(T'+\wh{X}_a^2)^s\Gamma^{\frac12}\ml{W}\right\|_{L^2_X}^2+\left\|(T'+\wh{X}_a^2)^sW_3\right\|_{L^2_X}^2\right]\md T'\\
		&\lesssim 1;
	\end{aligned}
\end{equation}
\begin{equation}\label{6.45}
	\begin{aligned}
		\left\|W_3(T,\cdot)\right\|_{L^2_X}^2&+\int_1^{T}\left\|\pa_XW_3\right\|_{L^2_X}^2\md T'\lesssim 1.
	\end{aligned}
\end{equation}

Using \eqref{6.37} with $s=\frac34$ to control $(\pa_T\ml V)W_3$, we obtain the estimates for $\pa_T\ml W$ and $\pa_TW_3$ similarly, with $s\in[0,\frac12]$:
\begin{equation}\label{6.38}
	\begin{aligned}
		\left\|(T+\wh{X}_a^2)^s\pa_T\ml{W}(T,\cdot)\right\|_{L^2_X}^2&+\int_1^{T}[\left\|(T'+\wh{X}_a^2)^s\Gamma^{\frac12}\pa_T\ml{W}\right\|_{L^2_X}^2\\&+\left\|(T'+\wh{X}_a^2)^s\pa_X\pa_T\ml{W}\right\|_{L^2_X}^2]\md T'\lesssim 1,
	\end{aligned}
\end{equation}
and 
\begin{equation*}
	\begin{aligned}
		\left\|\pa_TW_3(T,\cdot)\right\|_{L^2_X}^2&+\int_1^{T}\left\|\pa_X\pa_TW_3\right\|_{L^2_X}^2\md T'\lesssim 1.
	\end{aligned}
\end{equation*}

By \eqref{6.30} and the preceding energy bounds, the Cauchy--Schwarz inequality gives
\begin{equation}\label{7.47}
\begin{aligned}
\|\pa_XW_3(T,\cdot)\|_{L^2_X}^2
&=\|\pa_XW_3(1,\cdot)\|_{L^2_X}^2\\
&\quad+2\int_1^T\int_{\R}\pa_XW_3(T',X)\pa_T\pa_XW_3(T',X)\,\md X\,\md T'\lesssim1.
\end{aligned}
\end{equation}
The one-dimensional Sobolev inequality and \eqref{6.45} therefore give
\begin{equation}\label{7.49}
\|W_3(T,\cdot)\|_{L^\infty_X}^2
\lesssim\|W_3(T,\cdot)\|_{L^2_X}\|\pa_XW_3(T,\cdot)\|_{L^2_X}\lesssim1.
\end{equation}
Combining the estimates for $W_1$, $W_2$ and $W_3$, we obtain the desired
result on $T\in[1,\nu^{-\beta_2}]$. 

The proof of the estimate for $T\in[\nu^{-\beta_2},T_0]$ follows in a similar fashion; here we give a sketch. Denote by $\vph_r^1$ the function given in Lemma~\ref{lem6.5} with $F(T,X)=-abT^{\frac12}\tanh(\frac{abT^{\frac12}}{2}\wh{X})$, and by $\vph_r^2$ the function given in Lemma~\ref{lem6.5} with $F(T,X)=U_0^{\mathrm{shock}}-\dot{\mf{S}}$.
By \eqref{3.19} and \eqref{3.21}, $\vph_r^2$ is well-defined. Introduce an even cutoff function $\xi\in C_c^{\infty}$ such that $\xi(z)=1$ for $|z|\leq 1$ and $\xi=0$ for $|z|\geq 2$, and set $W_2(T,X)=\xi(\nu^{\beta_2}T)\vph_r^1(ab\sqrt{T},\wh{X})+[1-\xi(\nu^{\beta_2}T)]\vph_r^2(\om(T)\sqrt{T},\wh{X})$. Introduce $W_3=W-W_1-W_2$. By Proposition~\ref{prop4.7}, the additional error produced by $\xi$ on $T\in [\nu^{-\beta_2},2\nu^{-\beta_2}]$ is negligible provided $\beta_2\in (0,\frac1{100})$; hence the estimate \eqref{7.49} for $W_3$ extends to $T\in [1,2\nu^{-\beta_2}]$. For $T\in[2\nu^{-\beta_2},T_0]$, we decompose the coefficient $\ml{V}$ in \eqref{6.26} as
\begin{equation*}
	\begin{aligned}
		\ml{V}&=\ml{V}_0+\ml{V}_e\\
		\ml{V}_0&=\varsigma_oU_0^{\mathrm{out}}+\varsigma_{s}U_0^{\mathrm{shock}}-\dot{\mf{S}}+W_2\\
		\ml{V}_e&=\ml{O}((T+X^2)^{\beta_1-\frac32})+\ml{O}(\mf{D}^{-3})\\
		\pa_X\ml{V}_e&=\ml{O}((T+X^2)^{-1})+\ml{O}(\mf{D}^{-6}).\\
	\end{aligned}
\end{equation*}
Estimating $W_3$ as in \eqref{6.26}, we obtain \eqref{7.30} with
\begin{equation*}
	\begin{aligned}
		H_0\gtrsim \frac14\Gamma (T+\wh{X}_a^2)^{2s}-C\nu^{\frac12}(T+\wh{X}_a^2)^{2s},
	\end{aligned}
\end{equation*}
where the negative term $-C\nu^{\frac12}(T+\wh{X}_a^2)^{2s}$ comes from the error involving $U_{[K]}^{\mathrm{out}}$ for $|X|\gtrsim \nu^{-\frac34}$. This extends \eqref{7.49} to $T\in [1,T_0]$. We omit the details.

\end{proof}
	
We are now in a position to prove the main theorem.
	
\begin{proof}[Proof of Theorem~\ref{thm}]
	By the maximum principle, it suffices to consider sufficiently small $\nu>0$.
	Set $T_0=\ol\vep\nu^{-1/2}$ and work in the inner coordinates
	\eqref{inner} on $[0,T_0]\times\R$ and fix $\beta_i\in (0,\frac1{100})$ for $i=1,2,3$.
	We first fix sufficiently large truncation orders
	$J,K,L$; the implicit constants below may depend on these fixed orders,
	but are independent of $\nu$ and $T$. We establish the estimate
	\begin{equation}\label{eq:thminner}
		|U^{(\nu)}(T,X)-U^{(0)}(T,X)|\lesssim
		1+\sqrt T\,e^{-\delta\sqrt T|X|}
		\qquad (0\le T\le T_0,\ X\in\R).
	\end{equation}
	Multiplying by $\nu^{\frac14}$ and using \eqref{inner} yields the
	conclusion of the theorem.

	\emph{Short times ($0\le T\le1$).}
	Choose $K,L$ so large that the exponent in Proposition~\ref{prop6.1} is
	positive. On $|X|\le2$, the boundedness of the finitely many inner
	profiles and the local cusp description of $U^{(0)}$ give a uniform bound
	for $U_{K,L}^{\mathrm{app}}-U^{(0)}$. On $|X|\ge1$, the same bound follows
	from Lemma~\ref{lem:52} and the outer estimates in
	Propositions~\ref{prop3.2} and \ref{prop4.2}. Hence
	\[
	|U_{K,L}^{\mathrm{app}}(T,X)-U^{(0)}(T,X)|\lesssim1,
	\qquad 0\le T\le1.
	\]
	Together with Proposition~\ref{prop6.1}, this yields
	\begin{equation*}
		|U^{(\nu)}(T,X)-U^{(0)}(T,X)|\lesssim1\lesssim
		1+\sqrt T\,e^{-\delta\sqrt T|X|}
		\qquad (0\le T\le1).
	\end{equation*}

	\emph{Reduction on $1\le T\le T_0$.}
	By Proposition~\ref{prop6.4},
	\begin{equation*}
		\begin{aligned}
		|U^{(\nu)}-U^{(0)}|&\le
		|U^{(\nu)}-U^{\mathrm{app}}_{J,K,L}|
		+|U^{\mathrm{app}}_{J,K,L}-U^{(0)}|\\
		&\lesssim1+\sqrt T\,e^{-\delta\sqrt T|X|}
		+|U^{\mathrm{app}}_{J,K,L}-U^{(0)}|.
		\end{aligned}
	\end{equation*}
	Thus it remains to show
	\begin{equation}\label{eq:app0}
		|U^{\mathrm{app}}_{J,K,L}-U^{(0)}|\lesssim
		1+\sqrt T\,e^{-\delta\sqrt T|X|}
		\qquad (1\le T\le T_0).
	\end{equation}

	\emph{Zone decomposition.}
	Let $\varsigma_o+\varsigma_i+\varsigma_s=1$ be the partition of unity
	subordinate to the covering \eqref{6.7}, and write
	\begin{equation}\label{eq:dec0}
		U^{\mathrm{app}}_{J,K,L}-U^{(0)}
		=\varsigma_o\bigl(U_{[K]}^{\mathrm{out}}-U^{(0)}\bigr)
		+\varsigma_s\bigl(U_{[J]}^{\mathrm{shock}}-U^{(0)}\bigr)
		+\varsigma_i\bigl(U_{[L],0}^{\mathrm{in}}-U^{(0)}\bigr).
	\end{equation}
	We show that each of the three differences on the right-hand side of
	\eqref{eq:dec0} is bounded by the right-hand side of \eqref{eq:app0} on
	the support of its cutoff, after increasing $J,K,L$ if necessary. The
	supports of the cutoffs are contained in the corresponding zones
	\eqref{6.7} up to the matching regions, so it suffices to combine, for
	each zone, the zone estimate with the matching estimates valid on the
	two adjacent overlaps.

	The three zone estimates are as follows. On $O_{\mathrm{out}}$, by
	Propositions~\ref{prop3.1}--\ref{prop3.2} and the scaling \eqref{3.60},
	\begin{equation}\label{eq:zout}
		|U_{[K]}^{\mathrm{out}}-U^{(0)}|\lesssim1,
	\end{equation}
	for all $K$ sufficiently large. On $O_{\mathrm{shock}}$, the profile
	expansion interpolates between the two inviscid states across the shock
	layer of width $T^{-1/2}$; by Propositions~\ref{prop3.8}--\ref{prop3.9}
	and \eqref{eq:zout},
	\begin{equation}\label{eq:zshock}
		|U_{[J]}^{\mathrm{shock}}-U^{(0)}|\lesssim
		1+\sqrt T\,e^{-\delta\sqrt T|X|},
	\end{equation}
	for $J$ sufficiently large. On $O_{\mathrm{in}}$, the lowest-order inner
	correction must be retained. To make this contribution explicit, set
	\begin{equation}\label{eq:Ptrunc}
	\begin{aligned}
	P_{J,K,[L],0}&:=\sum_{\ell=0}^L\nu^{\frac\ell4}
	P_{J,K,\ell,0},\\
	U_{[J],[L]}^{\mathrm{SI}}&:=\sum_{\ell=0}^L\nu^{\frac\ell4}
	U_{[J],\ell}^{\mathrm{SI}},\qquad
	U_{[K],[L]}^{\mathrm{OI}}:=\sum_{\ell=0}^L\nu^{\frac\ell4}
	U_{[K],\ell}^{\mathrm{OI}}.
	\end{aligned}
	\end{equation}
	By \eqref{eq:decomp} and \eqref{eq:SOI}, the exact decomposition is
	\begin{equation}\label{eq:innertrunc}
	U_{[L],0}^{\mathrm{in}}
	=P_{J,K,[L],0}+\zeta U_{[J],[L]}^{\mathrm{SI}}
	+(1-\zeta)U_{[K],[L]}^{\mathrm{OI}}.
	\end{equation}
	Fix $q>0$. Since the index functions in Proposition~\ref{thm:main}
	tend to infinity with $J,K$, we may choose $J,K$ so large that
	$\psi_\ell(J,K)-a_{0,0}^{(\ell)}\ge q$ for all $0\le\ell\le L$.
	Summing \eqref{eq:Pellh} with the weights in \eqref{eq:Ptrunc} gives
	\begin{equation}\label{eq:Ptruncbd}
	|P_{J,K,[L],0}|\lesssim (T+X^2)^{-q},
	\qquad 1\le T\le2\nu^{-\beta_2},
	\end{equation}
	since $\sum_{\ell=0}^L\nu^{\ell(\frac14-\frac{\beta_2}{2})}\lesssim1$.
	In particular, \eqref{eq:Ptruncbd} is only $O(1)$ at a fixed inner point; no positive
	power of $\nu$ is asserted there. Propositions~\ref{prop4.3},
	\ref{prop:SI-reconstruction}, and \ref{prop4.9} compare the grid part of
	\eqref{eq:innertrunc} with the
	outer and shock-profile expansions. The viscous outer corrections
	must also be retained: with
	$\sigma=\max\{\beta_2/2,\beta_3/3\}<1/4$, the homogeneous gauge satisfies
	$1\lesssim\mathfrak D\lesssim\nu^{-\sigma}$ on
	$O_{\mathrm{in}}\cap\{T\ge1\}$. Since $U_0^{\mathrm{out}}=U^{(0)}$,
	Propositions~\ref{prop3.2} and \ref{prop4.3} give, for fixed $K$,
	\[
	\left|U_{[K],[L]}^{\mathrm{OI}}-U^{(0)}\right|
	\lesssim\mathfrak D^{-3}
	+\nu^{\frac{L+1}{4}}\mathfrak D^{L+2}
	\lesssim1,
	\]
	provided $(L+1)/4-\sigma(L+2)>0$. Thus the algebraic tail is
	uniformly bounded. In the core
	$|X|\lesssim T^{-1/2}$ the layer term is of size $O(\sqrt T)$; outside
	the core the outer--shock matching applies. Thus, after choosing $L$ large,
	\[
	\left|\zeta U_{[J],[L]}^{\mathrm{SI}}
	+(1-\zeta)U_{[K],[L]}^{\mathrm{OI}}-U^{(0)}\right|
	\lesssim1+\sqrt T\,e^{-\delta\sqrt T|X|}
	\]
	on $O_{\mathrm{in}}$. Combining this estimate with
	\eqref{eq:Ptruncbd} gives
	\begin{equation}\label{eq:zin}
	|U_{[L],0}^{\mathrm{in}}-U^{(0)}|
	\lesssim (T+X^2)^{-q}+1
	+\sqrt T\,e^{-\delta\sqrt T|X|}
	\lesssim1+\sqrt T\,e^{-\delta\sqrt T|X|}
	\end{equation}
	on $O_{\mathrm{in}}$.

	The matching estimates on the overlaps are as follows. On
	$O_{\mathrm{out}}\cap O_{\mathrm{shock}}$, we have $T\ge\nu^{-\beta_2}$
	and $1\le|X|T^{\frac12-\beta_1}\le2$, hence
	$|X|/T^{\frac32}\le2T^{\beta_1-2}\le2\nu^{\beta_2(2-\beta_1)}$, and
	Proposition~\ref{prop3.9} yields
	\begin{equation}\label{eq:mosh}
		\begin{aligned}
		|U_{[J]}^{\mathrm{shock}}-U_{[K]}^{\mathrm{out}}|
		&\lesssim T^{\frac12}\Bigl(e^{-\frac{\om(T)\sqrt T}{4}|X|}
		+\bigl(\tfrac{|X|}{T^{\frac32}}\bigr)^{K+1}
		+\bigl(\tfrac{|X|}{T^{\frac32}}\bigr)^{J+1}\Bigr)\\
		&\lesssim\sqrt T\,e^{-\delta\sqrt T|X|}+1,
		\end{aligned}
	\end{equation}
	where we decrease $\delta$ if necessary and choose $J,K$ large. On
	$O_{\mathrm{out}}\cap O_{\mathrm{in}}$, combining
	\eqref{eq:innertrunc} and Proposition~\ref{thm:main} with the outer--inner grid
	matching of Proposition~\ref{prop4.3} and the grid matching of
	Proposition~\ref{prop4.9}, we obtain
	\begin{equation}\label{eq:moi}
		\begin{aligned}
		|U_{[L],0}^{\mathrm{in}}-U_{[K]}^{\mathrm{out}}|
		&\le|P_{J,K,[L],0}|
		+|U_{[K],[L]}^{\mathrm{OI}}-U_{[K]}^{\mathrm{out}}|\\
		&\qquad+\zeta|U_{[J],[L]}^{\mathrm{SI}}-U_{[K],[L]}^{\mathrm{OI}}|\\
		&\lesssim1+\sqrt T\,e^{-\delta\sqrt T|X|}.
		\end{aligned}
	\end{equation}
	Indeed, on this overlap either $|X|\ge\nu^{-\beta_3}$ or
	$T\ge\nu^{-\beta_2}$; hence
	$T+X^2\gtrsim\nu^{-\min\{\beta_2,2\beta_3\}}$, and \eqref{eq:Ptruncbd}
	gives $|P_{J,K,[L],0}|\lesssim\nu^{q\min\{\beta_2,2\beta_3\}}$.
	On $O_{\mathrm{shock}}\cap O_{\mathrm{in}}$, the same decomposition,
	Propositions~\ref{prop:SI-estimates}, \ref{prop:SI-reconstruction} and
	\ref{prop4.9} give, since $T\ge\nu^{-\beta_2}$ there,
	\begin{equation}\label{eq:msi}
		|U_{[L],0}^{\mathrm{in}}-U_{[J]}^{\mathrm{shock}}|
		\lesssim1+\sqrt T\,e^{-\delta\sqrt T|X|}.
	\end{equation}
	Here \eqref{eq:Ptruncbd} contributes at most
	$\nu^{\beta_2q}\lesssim1$ for $q>0$ and $0<\nu\le1$.

	We now combine these estimates. On $O_{\mathrm{out}}$, the identity
	\eqref{eq:dec0} together with \eqref{eq:zout}, \eqref{eq:mosh} and
	\eqref{eq:moi} gives
	\begin{equation*}
		\begin{aligned}
		|U^{\mathrm{app}}_{J,K,L}-U^{(0)}|
		\le|U_{[K]}^{\mathrm{out}}-U^{(0)}|
		+\varsigma_s|U_{[J]}^{\mathrm{shock}}-U_{[K]}^{\mathrm{out}}|
		+\varsigma_i|U_{[L],0}^{\mathrm{in}}-U_{[K]}^{\mathrm{out}}|\\
		\lesssim1+\sqrt T\,e^{-\delta\sqrt T|X|}.
		\end{aligned}
	\end{equation*}
	Arguing identically on $O_{\mathrm{shock}}$ and $O_{\mathrm{in}}$ with
	\eqref{eq:zshock}, \eqref{eq:zin} and the matching estimates symmetric
	to \eqref{eq:mosh}--\eqref{eq:msi}, we obtain the same bound. Since the
	three zones \eqref{6.7} cover $[1,T_0]\times\R$, this proves
	\eqref{eq:app0}, and with it \eqref{eq:thminner} on $[0,T_0]\times\R$.

	Finally, \eqref{eq:thminner} and the scaling \eqref{inner} yield, for all
	$t\in[0,\ol\vep]$ and $x\in\R$,
	\begin{equation*}
		|u^{(\nu)}(t,x)-u^{(0)}(t,x)|\lesssim
		\nu^{\frac14}
		+\sqrt t\,e^{-\delta\frac{|x-\mf{s}(t)|\sqrt t}{\nu}},
	\end{equation*}
	which is the assertion of part (ii).
	This completes the proof.
\end{proof}

\appendix
\section{Proof of Proposition~\ref{prop3.6}}\label{app:prop46}

In this appendix we prove Proposition~\ref{prop3.6} for general $j\in\N$ by
induction, as deferred from Section~4.4. Assuming that the profiles
$V_{<j}^{\mathrm{shock}}$ and the phase coefficients $\Theta_{<j-1}^{\mathrm{S}}$
have been constructed with the estimates of Proposition~\ref{prop3.6}, we
decompose $V_j^{\mathrm{shock}}=\ol{V}_j+V_j^{\mathrm{cut}}$ and derive the
equation satisfied by $\ol{V}_j$. Lemma~\ref{lem3.5} shows that the
pure-cutoff terms cancel on $|Y|\ge1$; the phase $\Theta_{j-1}^{\mathrm{S}}$ is
then chosen so that the boundary contribution at $Y=+\infty$ vanishes, which
yields the exponentially decaying bound for $\ol{V}_j$ and, similarly, for its
derivatives.

\begin{proof}[Proof of Proposition~\ref{prop3.6}]
	Assume that Proposition~\ref{prop3.6} holds for $V_{<j}^{\mathrm{shock}}$ with the exponential decay rate $e^{-(\frac{1}{2}+\eta')\om(T)|Y|}$ for some $\eta'>0$, and for the phase coefficients $\Theta_{<j-1}^{\mathrm{S}}$. We show that $V_j^{\mathrm{shock}}$ satisfies the analogous estimate with the decay rate $e^{-(\frac{1}{2}+\eta)\om(T)|Y|}$ for some $0<\eta<\eta'$, and we determine $\Theta_{j-1}^{\mathrm{S}}$ accordingly. Define
	\begin{equation}\label{3.28}
		\begin{aligned}
			\ol{V}_j=V_j^{\mathrm{shock}}-V_j^{\mathrm{cut}}.
		\end{aligned}
	\end{equation}
	From \eqref{3.12} and \eqref{3.28} we derive the following equation for $\ol{V}_j$:
	\begin{align}
&\pa_Y^2\ol{V}_j-\pa_Y[(V_0^{\mathrm{shock}}-\frac{\dot{\mf{S}}}{\sqrt{T}})\ol{V}_j]=-(2j-\frac52)\ol{V}_{j-1}+\frac Y2\pa_Y\ol{V}_{j-1}+T\pa_T\ol{V}_{j-1}\notag\displaybreak[1]\\
		&+\sum\limits_{\substack{j_1+j_2=j-1\\j_2>0}}[T\pa_T\Theta_{j_1}^{\mathrm{S}}-(2j_1+\frac12)\Theta_{j_1}^{\mathrm{S}}]\pa_Y\ol{V}_{j_2}+\frac12\pa_Y\sum\limits_{\substack{j_1+j_2=j\\j_1,j_2>0}}\ol{V}_{j_1}\ol{V}_{j_2}\notag\displaybreak[1]\\
		&+\pa_Y\sum\limits_{\substack{j_1+j_2=j\\j_2>0}}\ol{V}_{j_1}{V}_{j_2}^{\mathrm{cut}}
		+\frac12\pa_Y\sum\limits_{j_1+j_2=j}V_{j_1}^{\mathrm{cut}}V_{j_2}^{\mathrm{cut}}+\frac Y2\pa_YV_{j-1}^{\mathrm{cut}}+T\pa_TV_{j-1}^{\mathrm{cut}}-(2j-\frac 52)V_{j-1}^{\mathrm{cut}}\notag\displaybreak[1]\\
		&-\frac{\dot{\mf{S}}}{\sqrt{T}}\pa_YV_j^{\mathrm{cut}}
		+\sum\limits_{\substack{j_1+j_2=j-1\\j_2>0}}[T\pa_T\Theta_{j_1}^{\mathrm{S}}-(2j_1+\frac 12)\Theta_{j_1}^{\mathrm{S}}]\pa_YV_{j_2}^{\mathrm{cut}}\notag\displaybreak[1]\\
		&+[T\pa_T\Theta_{j-1}^{\mathrm{S}}-(2j-\frac 32)\Theta_{j-1}^{\mathrm{S}}]\pa_YV_{0}^{\mathrm{shock}}
		-\pa_Y^2V_{j}^{\mathrm{cut}}\notag\displaybreak[1]\\
		&:=g_{\geq 1}+g_0+[T\pa_T\Theta_{j-1}^{\mathrm{S}}-(2j-\frac 32)\Theta_{j-1}^{\mathrm{S}}]\pa_YV_{0}^{\mathrm{shock}},
\label{3.38}
\end{align}
	where $g_{\geq1}$ collects all terms containing at least one factor $\ol{V}_{j'}$ with $j'<j$, and $g_0$ collects the remaining terms, which involve only the cutoff profiles $V^{\mathrm{cut}}$. We extract the $\pa_YV_0^{\mathrm{shock}}$ term, and choose $\Theta_{j-1}^{\mathrm{S}}$ appropriately later to ensure solvability. By the inductive hypothesis, $g_{\geq1}$ is exponentially decaying. We shall show that $g_0$ can be balanced by Lemma~\ref{lem3.5}. To this end, we decompose $g_0$ as follows:
		\begin{align*}
g_0=I_1+I_2+I_3+I_4+I_5,
\end{align*}
where
			\begin{align*}
		I_1&=\frac12\pa_Y\sum\limits_{j_1+j_2=j}V_{j_1}^{\mathrm{cut}}V_{j_2}^{\mathrm{cut}},\quad
		I_2=\frac Y2\pa_YV_{j-1}^{\mathrm{cut}}+T\pa_TV_{j-1}^{\mathrm{cut}}-(2j-\frac 52)V_{j-1}^{\mathrm{cut}}\notag\displaybreak[1]\\
		I_3&=\sum\limits_{\substack{j_1+j_2=j-1\\j_2>0}}[T\pa_T\Theta_{j_1}^{\mathrm{S}}-(2j_1+\frac 12)\Theta_{j_1}^{\mathrm{S}}]\pa_YV_{j_2}^{\mathrm{cut}},\quad
		I_4=-\frac{\dot{\mf{S}}}{\sqrt{T}}\pa_YV_j^{\mathrm{cut}},\quad
		I_5=-\pa_Y^2V_{j}^{\mathrm{cut}}.
\end{align*}
	In fact, we rewrite \eqref{3.27} as 
	\begin{equation}\label{3.39}
		\begin{aligned}
		V_j^{\mathrm{cut}}=\chi(Y)\sum\limits_{p=0}^j \sum\limits_{\substack{i\geq p,k\geq 0\\ h_1,...,h_p\geq 0\\i+k+h_1+...+h_p=j}}(-1)^{p}C_{i}^{p}Y^{i-p}{V}_{i,k}^{\pm}\prod_{r=1}^p\Theta_{h_r}^{\mathrm{S}}.
		\end{aligned}
	\end{equation}
	We now compute the relevant terms explicitly. From \eqref{3.27} we write $V_j^{\mathrm{cut}}=\chi\widetilde{V}_j^{\pm}$, where
	\[
	\begin{aligned}
	\widetilde{V}_j^{\pm}(Y)&=\sum\limits_{p=0}^j\sum\limits_{\substack{i\geq p,k\geq 0\\h_1,\dots,h_p\geq 0\\i+k+h_1+\cdots+h_p=j}}(-1)^{p}C_i^{p}Y^{i-p}V_{i,k}^{\pm}\prod_{r=1}^p\Theta_{h_r}^{\mathrm{S}},\\[4pt]
	\pa_Y\widetilde{V}_j^{\pm}&=\sum\limits_{p=0}^j\sum\limits_{\substack{i\geq p,k\geq 0\\\sum h_r=j-i-k}}(-1)^{p}C_i^{p}(i-p)Y^{i-p-1}V_{i,k}^{\pm}\prod_{r=1}^p\Theta_{h_r}^{\mathrm{S}}.
	\end{aligned}
	\]
	In what follows we write $f(Y)\asymp g(Y)$ to denote that $f(Y)=g(Y)$ for $|Y|\geq1$. Then
	\begin{equation*}
	\begin{aligned}
	\frac12\pa_Y\sum_{j_1+j_2=j}V_{j_1}^{\mathrm{cut}}V_{j_2}^{\mathrm{cut}}
	&\asymp\frac12\pa_Y\bigg[\sum_{j_1+j_2=j}\widetilde{V}_{j_1}^{\pm}\widetilde{V}_{j_2}^{\pm}\bigg]\\[2pt]
	&=\frac12\sum_{j_1+j_2=j}
	\bigl(\pa_Y\widetilde{V}_{j_1}^{\pm}\!\cdot\!\widetilde{V}_{j_2}^{\pm}
	+\widetilde{V}_{j_1}^{\pm}\!\cdot\!\pa_Y\widetilde{V}_{j_2}^{\pm}\bigr).
	\end{aligned}
	\end{equation*}
	Substituting the explicit formulas for $\widetilde{V}_j^{\pm}$ and its $Y$-derivative, we obtain the unmerged expansion
	\begin{align*}
\frac12&\sum_{j_1+j_2=j}
	\bigl(\pa_Y\widetilde{V}_{j_1}^{\pm}\!\cdot\!\widetilde{V}_{j_2}^{\pm}
	+\widetilde{V}_{j_1}^{\pm}\!\cdot\!\pa_Y\widetilde{V}_{j_2}^{\pm}\bigr)\notag\displaybreak[1]\\
	&\asymp\frac12\sum_{\substack{j_1+j_2=j\\0\leq p_1\leq j_1\\0\leq p_2\leq j_2}}
	\sum_{\substack{i_1\geq p_1,\,k_1\geq0\\i_2\geq p_2,\,k_2\geq0}}
	\sum_{h_{1,r},h_{2,s}\geq0} (-1)^{p_1+p_2}C_{i_1}^{p_1}C_{i_2}^{p_2}
	(i_1+i_2-p_1-p_2)Y^{i_1+i_2-p_1-p_2-1}\notag\displaybreak[1]\\
	&\qquad\times V_{i_1,k_1}^{\pm}V_{i_2,k_2}^{\pm}
	\prod_{r=1}^{p_1}\Theta_{h_{1,r}}^{\mathrm{S}}\prod_{s=1}^{p_2}\Theta_{h_{2,s}}^{\mathrm{S}},
\end{align*}
	where the inner sums are restricted to $i_1+k_1+\sum_{r=1}^{p_1}h_{1,r}=j_1$ and $i_2+k_2+\sum_{s=1}^{p_2}h_{2,s}=j_2$. With $p=p_1+p_2$, $i=i_1+i_2$, and $k=k_1+k_2$, the full expression simplifies to
	\begin{equation*}
	\begin{aligned}
	\frac12\pa_Y\sum_{j_1+j_2=j}V_{j_1}^{\mathrm{cut}}V_{j_2}^{\mathrm{cut}}
	&\asymp\frac12\sum\limits_{p=0}^{j}\sum\limits_{\substack{i\geq p,k\geq 0\\ i+k+\sum h_s=j}}
	(-1)^{p}(i-p)Y^{i-p-1}\\[-2pt]
	&\quad\times\Bigl(\prod_{r=1}^{p}\Theta_{h_r}^{\mathrm{S}}\Bigr)\,
	\sum_{p_1+p_2=p}\;
	\sum_{\substack{i_1+i_2=i\\i_1\geq p_1,\,i_2\geq p_2}}
	C_{i_1}^{p_1}C_{i_2}^{p_2}
	\sum_{k_1+k_2=k}V_{i_1,k_1}^{\pm}V_{i_2,k_2}^{\pm}.
	\end{aligned}
	\end{equation*}
	The quadratic term $I_1$ has now been simplified. It remains to expand the linear terms $I_2$ and $I_3$:
		\begin{align}
&I_2=\frac Y2\pa_YV_{j-1}^{\mathrm{cut}}+T\pa_TV_{j-1}^{\mathrm{cut}}-(2j-\frac 52)V_{j-1}^{\mathrm{cut}}\notag\displaybreak[1]\\
			&\asymp\sum\limits_{p=0}^{j-1} \sum\limits_{\substack{i\geq p,k\geq 0\\ h_1,...,h_p\geq 0\\i+k+h_1+...+h_p=j-1}}(-1)^{p}C_{i}^{p}\frac{i-p-4j+5}{2}Y^{i-p}{V}_{i,k}^{\pm}\prod_{r=1}^p\Theta_{h_r}^{\mathrm{S}}\notag\displaybreak[1]\\
			&+\sum\limits_{p=0}^{j-1}\sum\limits_{\substack{i\geq p,k\geq 0\\h_1,\dots,h_p\geq 0\\i+k+h_1+\cdots+h_p=j-1}}(-1)^{p}C_i^{p}Y^{i-p}T\pa_TV_{i,k}^{\pm}\prod_{r=1}^p\Theta_{h_r}^{\mathrm{S}}\notag\displaybreak[1]\\
			&+\sum\limits_{p=0}^{j-1}\sum\limits_{\substack{i\geq p,k\geq 0\\h_1,\dots,h_p\geq 0\\i+k+h_1+\cdots+h_p=j-1}}(-1)^{p}C_i^{p}Y^{i-p}V_{i,k}^{\pm}T\pa_T(\prod_{r=1}^p\Theta_{h_r}^{\mathrm{S}})\notag\\*
			&:=I_{2,a}+I_{2,b}+I_{2,c}.
\label{3.35}
\end{align}
	Similarly, for the term $I_3$ we have
		\begin{align*}
&I_3=\sum\limits_{\substack{j_1+j_2=j-1\\j_2>0}}[T\pa_T\Theta_{j_1}^{\mathrm{S}}-(2j_1+\frac 12)\Theta_{j_1}^{\mathrm{S}}]\pa_YV_{j_2}^{\mathrm{cut}}\notag\displaybreak[1]\\
		&\asymp\sum\limits_{\substack{j_1+j_2=j-1}}\sum\limits_{p=0}^{j_2}\sum\limits_{\substack{i\geq p,k\geq 0\\\sum h_r=j_2-i-k}}(-1)^{p}C_i^{p}(i-p)Y^{i-p-1}V_{i,k}^{\pm}\prod_{r=1}^p\Theta_{h_r}^{\mathrm{S}}T\pa_T\Theta_{j_1}^{\mathrm{S}}\notag\displaybreak[1]\\
		&-\sum\limits_{\substack{j_1+j_2=j-1}}(2j_1+\frac12)\sum\limits_{p=0}^{j_2}\sum\limits_{\substack{i\geq p,k\geq 0\\\sum h_r=j_2-i-k}}(-1)^{p}C_i^{p}(i-p)Y^{i-p-1}V_{i,k}^{\pm}\prod_{r=1}^p\Theta_{h_r}^{\mathrm{S}}\Theta_{j_1}^{\mathrm{S}}\notag\\*
		&:=I_{3,a}+I_{3,b}.
\end{align*}
	
	It remains to evaluate the terms involving $T\pa_T\Theta$. We first compute $I_{2,c}$:
		\begin{align*}
&I_{2,c}=\sum\limits_{p=0}^{j-1}\sum\limits_{\substack{i\geq p,k\geq 0\\h_1,\dots,h_p\geq 0\\i+k+h_1+\cdots+h_p=j-1}}(-1)^{p}C_i^{p}Y^{i-p}V_{i,k}^{\pm}T\pa_T(\prod_{r=1}^p\Theta_{h_r}^{\mathrm{S}})\notag\displaybreak[1]\\
		&=\sum\limits_{p=1}^{j-1}\sum\limits_{\substack{i\geq p,k\geq 0\\h_1,\dots,h_p\geq 0\\i+k+h_1+\cdots+h_p=j-1}}p(-1)^{p}C_i^{p}Y^{i-p}V_{i,k}^{\pm}T\pa_T\Theta_{h_1}^{\mathrm{S}}(\prod_{r=2}^p\Theta_{h_r}^{\mathrm{S}})\notag\displaybreak[1]\\
		&=\sum\limits_{p=1}^{j-1}\sum\limits_{\substack{i\geq p,k\geq 0\\h_1,\dots,h_p\geq 0\\i+k+h_1+\cdots+h_p=j-1}}\frac{i!}{(p-1)!(i-p)!}(-1)^{p}Y^{i-p}V_{i,k}^{\pm}T\pa_T\Theta_{h_p}^{\mathrm{S}}(\prod_{r=1}^{p-1}\Theta_{h_r}^{\mathrm{S}})\notag\displaybreak[1]\\
		&=\sum\limits_{p=1}^{j-1}\sum\limits_{\substack{i\geq p,k\geq 0\\h_1,\dots,h_p\geq 0\\i+k+h_1+\cdots+h_p=j-1}}(i-p+1)C_{i}^{p-1}(-1)^{p}Y^{i-p}V_{i,k}^{\pm}T\pa_T\Theta_{h_p}^{\mathrm{S}}(\prod_{r=1}^{p-1}\Theta_{h_r}^{\mathrm{S}})\notag\displaybreak[1]\\
		&=\sum\limits_{q=0}^{j-2}\sum\limits_{\substack{i\geq q,k\geq 0\\h_1,\dots,h_{q+1}\geq 0\\i+k+h_1+\cdots+h_{q+1}=j-1}}(i-q)C_{i}^{q}(-1)^{q+1}Y^{i-q-1}V_{i,k}^{\pm}T\pa_T\Theta_{h_{q+1}}^{\mathrm{S}}(\prod_{r=1}^{q}\Theta_{h_r}^{\mathrm{S}}),
\end{align*}
	and, for $I_{3,a}$, we obtain
	\begin{align*}
&I_{3,a}=\sum\limits_{\substack{j_1+j_2=j-1}}\sum\limits_{p=0}^{j_2}\sum\limits_{\substack{i\geq p,k\geq 0\\\sum h_r=j_2-i-k}}(-1)^{p}C_i^{p}(i-p)Y^{i-p-1}V_{i,k}^{\pm}\prod_{r=1}^p\Theta_{h_r}^{\mathrm{S}}T\pa_T\Theta_{j_1}^{\mathrm{S}}\notag\displaybreak[1]\\
	&=\sum\limits_{p=0}^{j-1}\sum\limits_{\substack{j_1+j_2=j-1\\j_2\geq p}}\sum\limits_{\substack{i\geq p,k\geq 0\\ i+k+h_1+...+h_p=j_2}}(-1)^{p}C_i^{p}(i-p)Y^{i-p-1}V_{i,k}^{\pm}\prod_{r=1}^p\Theta_{h_r}^{\mathrm{S}}T\pa_T\Theta_{j_1}^{\mathrm{S}}\notag\displaybreak[1]\\
	&=\sum\limits_{p=0}^{j-1}\sum\limits_{\substack{i\geq p,k\geq 0\\ i+k+h_1+...+h_p+h_{p+1}=j-1\\h_{p+1}\leq j-1-p}}(-1)^{p}C_i^{p}(i-p)Y^{i-p-1}V_{i,k}^{\pm}\prod_{r=1}^p\Theta_{h_r}^{\mathrm{S}}T\pa_T\Theta_{h_{p+1}}^{\mathrm{S}}\notag\displaybreak[1]\\
	&=\sum\limits_{p=0}^{j-2}\sum\limits_{\substack{i\geq p,k\geq 0\\ i+k+h_1+...+h_p+h_{p+1}=j-1}}(-1)^{p}C_i^{p}(i-p)Y^{i-p-1}V_{i,k}^{\pm}\prod_{r=1}^p\Theta_{h_r}^{\mathrm{S}}T\pa_T\Theta_{h_{p+1}}^{\mathrm{S}}\notag\\*
	&=-I_{2,c}.
\end{align*}
	For $I_{3,b}$, we use the rotation symmetry of the summation indices:
		\begin{align}
&-I_{3,b}=\sum\limits_{\substack{j_1+j_2=j-1}}(2j_1+\frac12)\sum\limits_{p=0}^{j_2}\sum\limits_{\substack{i\geq p,k\geq 0\\\sum h_r=j_2-i-k}}(-1)^{p}C_i^{p}(i-p)Y^{i-p-1}V_{i,k}^{\pm}\prod_{r=1}^p\Theta_{h_r}^{\mathrm{S}}\Theta_{j_1}^{\mathrm{S}}\notag\displaybreak[1]\\
		&=\sum\limits_{p=0}^{j-2}\sum\limits_{\substack{i\geq p,k\geq 0\\ i+k+h_1+...+h_p+h_{p+1}=j-1}}(2h_{p+1}+\frac12)(-1)^{p}C_i^{p}(i-p)Y^{i-p-1}V_{i,k}^{\pm}\prod_{r=1}^{p+1}\Theta_{h_r}^{\mathrm{S}}\notag\displaybreak[1]\\
		&=\sum\limits_{q=1}^{j-1}\sum\limits_{\substack{i\geq q-1,k\geq 0\\ i+k+h_1+...+h_{q}=j-1}}(2h_{q}+\frac12)(-1)^{q-1}C_i^{q-1}(i-q+1)Y^{i-q}V_{i,k}^{\pm}\prod_{r=1}^{q}\Theta_{h_r}^{\mathrm{S}}\notag\displaybreak[1]\\
		&=\sum\limits_{q=1}^{j-1}\sum\limits_{\substack{i\geq q,k\geq 0\\ i+k+h_1+...+h_{q}=j-1}}q(2h_{q}+\frac12)(-1)^{q-1}C_i^{q}Y^{i-q}V_{i,k}^{\pm}\prod_{r=1}^{q}\Theta_{h_r}^{\mathrm{S}}\notag\displaybreak[1]\\
		&=\sum\limits_{q=1}^{j-1}\sum\limits_{\substack{i\geq q,k\geq 0\\ i+k+h_1+...+h_{q}=j-1}}[\frac{q}{2}+2(j-1-i-k)](-1)^{q-1}C_i^{q}Y^{i-q}V_{i,k}^{\pm}\prod_{r=1}^{q}\Theta_{h_r}^{\mathrm{S}}.
\label{2j+0.5}
\end{align}
	Combining the expressions for $I_{2,a}$ and $I_{3,b}$, we obtain
	\begin{align*}
I_{2,a}+I_{3,b}&=\sum\limits_{p=0}^{j-1} \sum\limits_{\substack{i\geq p,k\geq 0\\ h_1,...,h_p\geq 0\\i+k+h_1+...+h_p=j-1}}(-1)^{p}C_{i}^{p}\frac{i-p-4j+5}{2}Y^{i-p}{V}_{i,k}^{\pm}\prod_{r=1}^p\Theta_{h_r}^{\mathrm{S}}\notag\displaybreak[1]\\
		&+\sum\limits_{p=1}^{j-1}\sum\limits_{\substack{i\geq p,k\geq 0\\ i+k+h_1+...+h_{p}=j-1}}[\frac{p}{2}+2(j-1-i-k)](-1)^{p}C_i^{p}Y^{i-p}V_{i,k}^{\pm}\prod_{r=1}^{p}\Theta_{h_r}^{\mathrm{S}}\notag\displaybreak[1]\\
		&=\sum\limits_{p=0}^{j-1}\sum\limits_{\substack{i\geq p,k\geq 0\\ i+k+h_1+...+h_{p}=j-1}}(-1)^{p}C_{i}^{p}\frac{1-4k-3i}{2}Y^{i-p}V_{i,k}^{\pm}\prod_{r=1}^{p}\Theta_{h_r}^{\mathrm{S}}.
\end{align*}
The main nonlinear term $I_1$ simplifies as follows:
	\begin{align*}
&I_1:=\frac12\sum\limits_{p=0}^{j}\sum\limits_{\substack{i\geq p,k\geq 0\\ i+k+\sum h_s=j}}
		(-1)^{p}(i-p)Y^{i-p-1}\prod_{r=1}^{p}\Theta_{h_r}^{\mathrm{S}}\notag\displaybreak[1]\\
		&\times\sum_{p_1+p_2=p}
		\sum_{\substack{i_1+i_2=i\\i_1\geq p_1,\,i_2\geq p_2}}
		C_{i_1}^{p_1}C_{i_2}^{p_2}
		\sum_{k_1+k_2=k}V_{i_1,k_1}^{\pm}V_{i_2,k_2}^{\pm}\notag\displaybreak[1]\\
		&=\frac12\sum\limits_{p=0}^{j-1}\sum\limits_{\substack{a\geq p,k\geq 0\\ a+k+\sum h_s=j-1}}
		(-1)^{p}(a-p+1)Y^{a-p}\prod_{r=1}^{p}\Theta_{h_r}^{\mathrm{S}}\notag\displaybreak[1]\\
		&\times\sum_{p_1+p_2=p}
		\sum_{\substack{i_1+i_2=a+1\\i_1\geq p_1,\,i_2\geq p_2\\k_1+k_2=k}}
		C_{i_1}^{p_1}C_{i_2}^{p_2}
		V_{i_1,k_1}^{\pm}V_{i_2,k_2}^{\pm}\notag\displaybreak[1]\\
		&=\frac12\sum\limits_{p=0}^{j-1}\sum\limits_{\substack{a\geq p,k\geq 0\\ a+k+\sum h_s=j-1}}
		(-1)^{p}(a-p+1)Y^{a-p}\prod_{r=1}^{p}\Theta_{h_r}^{\mathrm{S}}\notag\displaybreak[1]\\
		&\times
		\sum_{\substack{i_1+i_2=a+1\\k_1+k_2=k}}
		V_{i_1,k_1}^{\pm}V_{i_2,k_2}^{\pm}\sum_{\substack{p_1+p_2=p\\p_1\leq i_1,p_2\leq i_2}}	C_{i_1}^{p_1}C_{i_2}^{p_2}\notag\displaybreak[1]\\
		&=\frac12\sum\limits_{p=0}^{j-1}\sum\limits_{\substack{a\geq p,k\geq 0\\ a+k+\sum h_s=j-1}}
		(-1)^{p}(a+1)C_a^pY^{a-p}\prod_{r=1}^{p}\Theta_{h_r}^{\mathrm{S}}
		\sum_{\substack{i_1+i_2=a+1\\k_1+k_2=k}}
		V_{i_1,k_1}^{\pm}V_{i_2,k_2}^{\pm}.
\end{align*}
	We also record the expansions of the remaining terms $I_4$ and $I_5$:
	\begin{align*}
-\frac{\dot{\mf{S}}}{\sqrt{T}}\pa_YV_j^{\mathrm{cut}}
		&\asymp\frac{\dot{\mf{S}}}{\sqrt{T}}\sum\limits_{p=0}^j \sum\limits_{\substack{i\geq p,k\geq 0\\ h_1,...,h_p\geq 0\\i+k+h_1+...+h_p=j}}(-1)^{p+1}C_{i}^{p}(i-p)Y^{i-p-1}{V}_{i,k}^{\pm}\prod_{r=1}^p\Theta_{h_r}^{\mathrm{S}}\notag\displaybreak[1]\\
		&=\frac{\dot{\mf{S}}}{\sqrt{T}}\sum\limits_{p=0}^{j-1} \sum\limits_{\substack{i\geq p+1,k\geq 0\\ h_1,...,h_p\geq 0\\i+k+h_1+...+h_p=j}}(-1)^{p+1}iC_{i-1}^{p}Y^{i-p-1}{V}_{i,k}^{\pm}\prod_{r=1}^p\Theta_{h_r}^{\mathrm{S}}\notag\displaybreak[1]\\
		&=-\frac{\dot{\mf{S}}}{\sqrt{T}}\sum\limits_{p=0}^{j-1} \sum\limits_{\substack{a\geq p,k\geq 0\\ h_1,...,h_p\geq 0\\a+k+h_1+...+h_p=j-1}}(a+1)(-1)^{p}C_{a}^{p}Y^{a-p}{V}_{a+1,k}^{\pm}\prod_{r=1}^p\Theta_{h_r}^{\mathrm{S}},
\end{align*}
and, for $I_5$, we have
		\begin{align*}
-\pa_Y^2V_{j}^{\mathrm{cut}}
			&\asymp\sum\limits_{p=0}^j \sum\limits_{\substack{i\geq p,k\geq 0\\ h_1,...,h_p\geq 0\\i+k+h_1+...+h_p=j}}(-1)^{p+1}C_{i}^{p}(i-p)(i-p-1)Y^{i-p-2}{V}_{i,k}^{\pm}\prod_{r=1}^p\Theta_{h_r}^{\mathrm{S}}\notag\displaybreak[1]\\
			&=\sum\limits_{p=0}^{j-1} \sum\limits_{\substack{i\geq p+2,k\geq 0\\ h_1,...,h_p\geq 0\\i+k+h_1+...+h_p=j}}(-1)^{p+1}C_{i-2}^{p}i(i-1)Y^{i-p-2}{V}_{i,k}^{\pm}\prod_{r=1}^p\Theta_{h_r}^{\mathrm{S}}\notag\displaybreak[1]\\
			&=\sum\limits_{p=0}^{j-1} \sum\limits_{\substack{a\geq p,b\geq 1\\ h_1,...,h_p\geq 0\\a+b+h_1+...+h_p=j-1}}(-1)^{p+1}C_{a}^{p}(a+2)(a+1)Y^{a-p}{V}_{a+2,b-1}^{\pm}\prod_{r=1}^p\Theta_{h_r}^{\mathrm{S}}.
\end{align*}
	Combining the quadratic term, the linear terms, and the main parts of $I_4$ and $I_5$, and applying Lemma~\ref{lem3.5}, we obtain, for $|Y|\geq1$,
	\begin{equation}\label{combined_sec3}
		\begin{aligned}
			g_0&\asymp\sum\limits_{p=0}^{j-1}\sum\limits_{\substack{i\geq p,k\geq 0\\ i+k+h_1+...+h_{p}=j-1}}(-1)^{p}C_{i}^{p}Y^{i-p}\prod_{r=1}^{p}\Theta_{h_r}^{\mathrm{S}}\Bigl[\frac{1-4k-3i}{2}V_{i,k}^{\pm}+T\pa_TV_{i,k}^{\pm}\\
			&+\frac{i+1}{2}\sum_{\substack{i_1+i_2=i+1\\k_1+k_2=k}}V_{i_1,k_1}^{\pm}V_{i_2,k_2}^{\pm}-(i+2)(i+1)V_{i+2,k-1}^{\pm}-\dot{\mf{S}}(i+1)T^{-\frac12}V_{i+1,k}^{\pm}\Bigr]=0,
		\end{aligned}
	\end{equation}
	where the last two terms in the bracket are the main parts of $I_5$ and $I_4$, respectively.
By \eqref{combined_sec3}, $g_0$ vanishes on $|Y|\geq1$; let $\ml{F}_j:=g_{\geq 1}+g_0$ denote the right-hand side of \eqref{3.38} with the term involving $\Theta_{j-1}^{\mathrm{S}}$ removed. Integrating \eqref{3.38} once in $Y$, we obtain
		\begin{equation}\label{3.53}
		\begin{aligned}
		\pa_Y\ol{V}_j-(V_0^{\mathrm{shock}}-\frac{\dot{\mf{S}}}{\sqrt{T}})\ol{V}_j&=\int_{-\infty}^{Y}\ml{F}_{j}(T,Y')\md Y'\\
		&+[T\pa_T\Theta_{j-1}^{\mathrm{S}}-(2j-\frac 32)\Theta_{j-1}^{\mathrm{S}}](V_0^{\mathrm{shock}}-V_{0,0}^-).
		\end{aligned}
		\end{equation}
We need to extract the contribution of $\Theta_{j-1}^{\mathrm{S}}$ in the $\int \ml{F}_j$ term. From \eqref{3.39}, the only term in $V_j^{\mathrm{cut}}$ involving $\Theta_{j-1}^{\mathrm{S}}$ is
\[
-\chi(Y)V_{1,0}^{\pm}(T)\Theta_{j-1}^{\mathrm{S}}(T).
\]
Consequently,
\[
\int_{\R}\ml{F}_j(T,Y')\md Y'
=\int_{\R}\ml{F}_j|_{\Theta_{j-1}^{\mathrm{S}}=0}(T,Y')\md Y'
+\om(T)(V_{1,0}^++V_{1,0}^-)\Theta_{j-1}^{\mathrm{S}}(T).
\]
To cancel the boundary contribution at $Y=+\infty$ (recall that $V_0^{\mathrm{shock}}\to V_{0,0}^{+}$ as $Y\to+\infty$ and $V_{0,0}^{-}-V_{0,0}^{+}=2\om(T)$), we choose $\Theta_{j-1}^{\mathrm{S}}$ and use \eqref{430} to ensure
\[
\begin{aligned}
0&=\int_{\R}\ml{F}_j(T,Y')\md Y'-2\om(T)[T\pa_T\Theta_{j-1}^{\mathrm{S}}-(2j-\tfrac32)\Theta_{j-1}^{\mathrm{S}}]\\
&=\int_{\R}\ml{F}_j|_{\Theta_{j-1}^{\mathrm{S}}=0}(T,Y')\md Y'-2\om(T)\left[T\pa_T\Theta_{j-1}^{\mathrm{S}}-(2j-\tfrac32)\Theta_{j-1}^{\mathrm{S}}-\frac{V_{1,0}^++V_{1,0}^-}{2}\Theta_{j-1}^{\mathrm{S}}\right]\\
&=\int_{\R}\ml{F}_j|_{\Theta_{j-1}^{\mathrm{S}}=0}(T,Y')\md Y'-2T\pa_T(\om\Theta_{j-1}^{\mathrm{S}})+2(2j-2)\om\Theta_{j-1}^{\mathrm{S}}.
\end{aligned}
\]
Setting $\wh{\ml F}_j(T):=\int_{\R}\ml{F}_j|_{\Theta_{j-1}^{\mathrm{S}}=0}(T,Y')\md Y'$, and imposing $\Theta_{j-1}^{\mathrm{S}}(T_0)=0$, we solve the resulting linear first-order differential equation to obtain
		\begin{equation}\label{3.55}
		\begin{aligned}
			\Theta_{j-1}^{\mathrm{S}}(T)=-\frac{T^{2j-2}}{2\om(T)}\int_T^{T_0}\frac{\wh{\ml{F}}_j(\tau)}{\tau^{2j-1}}\md \tau.
		\end{aligned}
		\end{equation}
With $\Theta_{j-1}^{\mathrm{S}}$ thus determined, we have $|\Theta_{j-1}^{\mathrm{S}}|\lesssim 1$ for $T\in[1,T_0]$, and the right-hand side of \eqref{3.53} decays exponentially as $|Y|\to\infty$.  Setting $\ol{V}_j(T,0)=0$ and choosing $\eta\in(0,\eta')$, we obtain
	\begin{equation*}
		\begin{aligned}
		|\ol{V}_j|\lesssim e^{-(\frac{1}{2}+\eta)\om(T)|Y|}.
		\end{aligned}
	\end{equation*}
The estimates for the derivatives follow in the same way.
\end{proof}
\section{Proof of Proposition~\ref{prop4.5}}
In this appendix, we prove Proposition~\ref{prop4.5} for general $V_{j,\ell}^{\mathrm{SI}}$ and $\Theta_{j-1,\ell}$ by induction, where $j\in \N$ and $\ell \in \N_0$. We shall use the following lemma, which follows directly from Lemma~\ref{lem3.5}:
\begin{lemma}\label{lem4.6}
	For $k,i,\ell\in\N_0$, the scaled Taylor coefficients $V_{i,k,\ell}^{\pm}(T)$ satisfy
	\begin{equation*}
		\begin{aligned}
			\Big(\frac{1-4k-3i}{2}&V_{i,k,\ell}^{\pm}+T\pa_TV_{i,k,\ell}^{\pm}\Big)+\frac{i+1}{2}\sum\limits_{\substack{k_1+k_2=k\\i_1+i_2=i+1\\\ell_1+\ell_2=\ell}}V_{i_1,k_1,\ell_1}^{\pm}V_{i_2,k_2,\ell_2}^{\pm}\\
			&=(i+2)(i+1)V_{i+2,k-1,\ell}^{\pm}
			+(i+1)\sum\limits_{\ell_1+\ell_2=\ell-1}\mf{S}_{\ell_1}V_{i+1,k,\ell_2}^{\pm}.
		\end{aligned}
	\end{equation*}
\end{lemma}
Now we give the proof.
\begin{proof}[Proof of Proposition~\ref{prop4.5}]
	We argue by induction. The estimate for $V_{0,0}^{\mathrm{SI}}$ is straightforward. Assume that Proposition~\ref{prop4.5} holds for $V_{j',\ell'}^{\mathrm{SI}}$ and $\Theta_{j'-1,\ell'}$ with $j'<j$, or with $j'=j$ and $\ell'<\ell$, and that these profiles satisfy \eqref{4.14} with the exponential factor $e^{-\mu'|Y|}$ for some $\mu'>\mu$. We now construct $V_{j,\ell}^{\mathrm{SI}}$ and $\Theta_{j-1,\ell}$ satisfying the same estimates.\par
	Writing $V_{j,\ell}^{\mathrm{SI}}=\ol{V}_{j,\ell}+V_{j,\ell}^{\mathrm{cut}}$, we derive from Proposition~\ref{prop4.4} the following equation for $\ol{V}_{j,\ell}$:
	\begin{equation*}
		\begin{aligned}
			\pa_Y^2\ol{V}_{j,\ell}-\pa_Y(V_{0,0}^{\mathrm{SI}}\ol{V}_{j,\ell})=g_{\geq 1}+g_0+[T\pa_T\Theta_{j-1,\ell}-(2j-\frac 32)\Theta_{j-1,\ell}]\pa_YV_{0,0}^{\mathrm{SI}}.
		\end{aligned}
	\end{equation*}
	By the induction hypothesis, every term in $g_{\geq 1}$ contains a factor $\ol{V}_{j',\ell'}$ and hence decays exponentially; such terms can be bounded as in Section~4. It remains to analyze $g_0$:
	\begin{align}
g_0&=I_1+I_2+I_3+I_4+I_5\notag\displaybreak[1]\\
			I_1&=\frac12\pa_Y\sum_{\substack{j_1+j_2=j\\\ell_1+\ell_2=\ell}}V_{j_1,\ell_1}^{\mathrm{cut}}V_{j_2,\ell_2}^{\mathrm{cut}}\notag\displaybreak[1]\\
			I_2&=\frac Y2\pa_YV_{j-1,\ell}^{\mathrm{cut}}+T\pa_TV_{j-1,\ell}^{\mathrm{cut}}-(2j-\frac52)V_{j-1,\ell}^{\mathrm{cut}}\notag\displaybreak[1]\\
			I_3&=\sum_{\substack{j_1+j_2=j-1\\(j_2,\ell_2)\ne(0,0)\\\ell_1+\ell_2=\ell}}[T\pa_T\Theta_{j_1,\ell_1}-(2j_1+\frac 12)\Theta_{j_1,\ell_1}]\pa_YV_{j_2,\ell_2}^{\mathrm{cut}}\notag\displaybreak[1]\\
			I_4&=-\sum\limits_{\ell_1+\ell_2=\ell-1}\mf{S}_{\ell_1}\pa_YV_{j,\ell_2}^{\mathrm{cut}}\notag\\*
			I_5&=-\pa_Y^2V_{j,\ell}^{\mathrm{cut}}.
\label{4.17}
\end{align}
	To analyze $g_0$, we write $V_{j,\ell}^{\mathrm{cut}}=\chi\widetilde{V}_{j,\ell}^{\pm}$, where
	\[
	\begin{aligned}
		\widetilde{V}_{j,\ell}^{\pm}(Y)&=\sum_{p=0}^{j}\sum_{\substack{i\geq p,k\geq0\\h_1,\dots,h_p\geq0\\\ell_0,\dots,\ell_p\geq0\\
				i+k+\sum h_r=j\\\ell_0+\sum\ell_r=\ell}}
		(-1)^{p}C_i^{p}Y^{i-p}V_{i,k,\ell_0}^{\pm}\prod_{r=1}^{p}\Theta_{h_r,\ell_r},\\[4pt]
		\pa_Y\widetilde{V}_{j,\ell}^{\pm}&=\sum_{p=0}^{j}\sum_{\substack{i\geq p,k\geq0\\\sum h_r=j-i-k\\\ell_0+\sum\ell_r=\ell}}
		(-1)^{p}C_i^{p}(i-p)Y^{i-p-1}V_{i,k,\ell_0}^{\pm}\prod_{r=1}^{p}\Theta_{h_r,\ell_r}.
	\end{aligned}
	\]
	We then compute $I_1$ as follows:
	\begin{align}
\frac12\pa_Y&\sum_{\substack{j_1+j_2=j\\\ell_1+\ell_2=\ell}}V_{j_1,\ell_1}^{\mathrm{cut}}V_{j_2,\ell_2}^{\mathrm{cut}}
			\asymp
			\frac12\sum_{\substack{j_1+j_2=j\\\ell_1+\ell_2=\ell}}
			\bigl(\pa_Y\widetilde{V}_{j_1,\ell_1}^{\pm}\widetilde{V}_{j_2,\ell_2}^{\pm}
			+\widetilde{V}_{j_1,\ell_1}^{\pm}\pa_Y\widetilde{V}_{j_2,\ell_2}^{\pm}\bigr)\notag\displaybreak[1]\\
			&=\frac12\sum_{p=0}^{j-1}\sum_{\substack{a\geq p,k\geq0\\a+k+\sum h_r=j-1}}
			\sum_{\substack{\ell_0,\dots,\ell_p\geq0\\\ell_0+\sum\ell_r=\ell}}
			(-1)^{p}(a+1)C_a^{p}Y^{a-p}
			\prod_{r=1}^{p}\Theta_{h_r,\ell_r}\notag\displaybreak[1]\\
			&\times
			\sum_{\substack{i_1+i_2=a+1\\k_1+k_2=k\\\ell_0'+\ell_0''=\ell_0}}
			V_{i_1,k_1,\ell_0'}^{\pm}V_{i_2,k_2,\ell_0''}^{\pm}.
\label{nonlin_l}
\end{align}
	
	We next expand $I_2$:
	\begin{align*}
&\frac Y2\pa_YV_{j-1,\ell}^{\mathrm{cut}}+T\pa_TV_{j-1,\ell}^{\mathrm{cut}}-(2j-\frac52)V_{j-1,\ell}^{\mathrm{cut}}\notag\displaybreak[1]\\
			&\asymp\sum_{p=0}^{j-1}\sum_{\substack{i\geq p,k\geq0\\\sum h_r=j-1-i-k\\\ell_0+\sum\ell_r=\ell}}
			(-1)^{p}C_i^{p}\frac{i-p-4j+5}{2}Y^{i-p}
			V_{i,k,\ell_0}^{\pm}\prod_{r=1}^{p}\Theta_{h_r,\ell_r}\notag\displaybreak[1]\\
			&+\sum_{p=0}^{j-1}\sum_{\substack{i\geq p,k\geq0\\\sum h_r=j-1-i-k\\\ell_0+\sum\ell_r=\ell}}
			(-1)^{p}C_i^{p}Y^{i-p}\,
			T\pa_TV_{i,k,\ell_0}^{\pm}\prod_{r=1}^{p}\Theta_{h_r,\ell_r}\notag\displaybreak[1]\\
			&+\sum_{p=0}^{j-1}\sum_{\substack{i\geq p,k\geq0\\\sum h_r=j-1-i-k\\\ell_0+\sum\ell_r=\ell}}
			(-1)^{p}C_i^{p}Y^{i-p}
			V_{i,k,\ell_0}^{\pm}\,T\pa_T\Bigl(\prod_{r=1}^{p}\Theta_{h_r,\ell_r}\Bigr).
\end{align*}
	The term $I_3$ expands as follows:
	\begin{align*}
&\sum_{\substack{j_1+j_2=j-1\\(j_2,\ell_2)\ne(0,0)\\\ell_1+\ell_2=\ell}}[T\pa_T\Theta_{j_1,\ell_1}-(2j_1+\tfrac12)\Theta_{j_1,\ell_1}]\pa_YV_{j_2,\ell_2}^{\mathrm{cut}}\notag\displaybreak[1]\\
			&\asymp\sum_{\substack{j_1+j_2=j-1\\j_2>0\\\ell_1+\ell_2=\ell}}
			\sum_{p=0}^{j_2}\sum_{\substack{i\geq p,k\geq0\\\sum h_r=j_2-i-k\\\ell_0+\sum\ell_r=\ell_2}}
			(-1)^{p}C_i^{p}(i-p)Y^{i-p-1}
			V_{i,k,\ell_0}^{\pm}\prod_{r=1}^{p}\Theta_{h_r,\ell_r} T\pa_T\Theta_{j_1,\ell_1}\notag\displaybreak[1]\\
			&-\sum_{\substack{j_1+j_2=j-1\\j_2>0\\\ell_1+\ell_2=\ell}}
			(2j_1+\tfrac12)\sum_{p=0}^{j_2}\sum_{\substack{i\geq p,k\geq0\\\sum h_r=j_2-i-k\\\ell_0+\sum\ell_r=\ell_2}}
			(-1)^{p}C_i^{p}(i-p)Y^{i-p-1}
			V_{i,k,\ell_0}^{\pm}\notag\\*
			&\times\prod_{r=1}^{p}\Theta_{h_r,\ell_r}\;\Theta_{j_1,\ell_1}.
\end{align*}
	Note that the $j_2=0$ term in the expansion above vanishes identically, since it forces $p=i=k=0$ and hence the coefficient $(i-p)Y^{i-p-1}$ is zero. Therefore the restriction $j_2>0$ may be dropped in the relabeled sums below.
	
	We now compute the $T\pa_T\Theta$ terms. The term involving $T\pa_T(\prod\Theta)$ in the linear part is $\asymp I$, where $I$ is defined by
	\[
	I:=\sum_{q=0}^{j-2}\sum_{\substack{i\geq q,k\geq0\\\sum h_r=j-1-i-k\\\ell_0+\sum\ell_r=\ell}}
	(i-q)C_i^{q}(-1)^{q+1}Y^{i-q-1}
	V_{i,k,\ell_0}^{\pm}\,T\pa_T\Theta_{h_{q+1},\ell_{q+1}}\prod_{r=1}^{q}\Theta_{h_r,\ell_r}.
	\]
	Meanwhile, relabeling $j_1=h_{p+1}$ and $\ell_1=\ell_{p+1}$ in the $T\pa_T\Theta_{j_1,\ell_1}$ term of the $\Theta$-sum, we obtain
	\begin{equation*}
		\begin{aligned}
			\sum_{p=0}^{j-2}\sum_{\substack{i\geq p,k\geq0\\\sum h_r=j-1-i-k\\\ell_0+\sum\ell_r=\ell}}
			(-1)^{p}C_i^{p}(i-p)Y^{i-p-1}
			V_{i,k,\ell_0}^{\pm}\,T\pa_T\Theta_{h_{p+1},\ell_{p+1}}\prod_{r=1}^{p}\Theta_{h_r,\ell_r}
			\asymp-I.
		\end{aligned}
	\end{equation*}
	Thus the two $T\pa_T\Theta$ contributions cancel on $|Y|\geq1$.
	
	For the $(2j_1+\frac12)\Theta$ part, we relabel $j_1=h_{p+1}$ and $\ell_1=\ell_{p+1}$.
	For fixed $q\ge1$, $i,k$ and $0\le\ell_0\le\ell$, set $S:=j-1-i-k\ge0$.
	The product and the joint summation domain are invariant under simultaneous
	permutations of the pairs $(h_r,\ell_r)$, $r=1,\dots,q$. Averaging over
	these $q$ positions therefore gives
	\begin{equation}\label{rot_sym}
		\sum_{\substack{h_1,\dots,h_q\geq0\\\ell_1,\dots,\ell_q\geq0\\h_1+\cdots+h_q=S\\\ell_1+\cdots+\ell_q=\ell-\ell_0}}q\,h_q\prod_{r=1}^{q}\Theta_{h_r,\ell_r}
		=S\sum_{\substack{h_1,\dots,h_q\geq0\\\ell_1,\dots,\ell_q\geq0\\h_1+\cdots+h_q=S\\\ell_1+\cdots+\ell_q=\ell-\ell_0}}\prod_{r=1}^{q}\Theta_{h_r,\ell_r},
	\end{equation}
	so that, after this joint summation, the factor $q(2h_q+\frac12)$ is replaced by $2S+\frac q2$. Hence
	\begin{equation*}
		\begin{aligned}
			&\sum_{q=1}^{j-1}\sum_{\substack{i\geq q,k\geq0\\\sum h_r=j-1-i-k\\\ell_0+\sum\ell_r=\ell}}
			q\bigl(2h_q+\tfrac12\bigr)(-1)^{q-1}C_i^{q}Y^{i-q}
			V_{i,k,\ell_0}^{\pm}\prod_{r=1}^{q}\Theta_{h_r,\ell_r}\\
			&=\sum_{q=1}^{j-1}\sum_{\substack{i\geq q,k\geq0\\\sum h_r=j-1-i-k\\\ell_0+\sum\ell_r=\ell}}
			\bigl[\tfrac{q}{2}+2(j-1-i-k)\bigr](-1)^{q-1}C_i^{q}Y^{i-q}
			V_{i,k,\ell_0}^{\pm}\prod_{r=1}^{q}\Theta_{h_r,\ell_r}.
		\end{aligned}
	\end{equation*}
	Combining this with the linear part yields the combined coefficient
	\begin{equation*}
		\begin{aligned}
			&\frac{i-p-4j+5}{2}+\frac{p}{2}+2(j-1-i-k)
			=\frac{1-4k-3i}{2}.
		\end{aligned}
	\end{equation*}
	The term $I_5=-\pa_Y^2V_{j,\ell}^{\mathrm{cut}}$ expands, in the same way, as
	\begin{align}
-\pa_Y^2V_{j,\ell}^{\mathrm{cut}}
			&\asymp\sum_{p=0}^{j}\sum_{\substack{i\geq p,k\geq0\\\sum h_r=j-i-k\\\ell_0+\sum\ell_r=\ell}}(-1)^{p+1}C_i^{p}(i-p)(i-p-1)Y^{i-p-2}V_{i,k,\ell_0}^{\pm}\prod_{r=1}^{p}\Theta_{h_r,\ell_r}\notag\displaybreak[1]\\
			&=\sum_{p=0}^{j-1}\sum_{\substack{i\geq p+2,k\geq0\\\sum h_r=j-i-k\\\ell_0+\sum\ell_r=\ell}}(-1)^{p+1}C_{i-2}^{p}i(i-1)Y^{i-p-2}V_{i,k,\ell_0}^{\pm}\prod_{r=1}^{p}\Theta_{h_r,\ell_r}\notag\displaybreak[1]\\
			&=\sum_{p=0}^{j-1}\sum_{\substack{a\geq p,b\geq1\\a+b+\sum h_r=j-1\\\ell_0+\sum\ell_r=\ell}}(-1)^{p+1}C_{a}^{p}(a+2)(a+1)Y^{a-p}V_{a+2,b-1,\ell_0}^{\pm}\prod_{r=1}^{p}\Theta_{h_r,\ell_r},
\label{P2cut_l}
\end{align}
	where we used $C_i^{p}(i-p)(i-p-1)=i(i-1)C_{i-2}^{p}$ and the shift $(i,k)\mapsto(a+2,b-1)$.
	
	The term $I_4$ expands, in the same way, as
	\begin{equation}\label{S_cancel}
		\begin{aligned}
			-\sum_{\substack{\ell_1+\ell_2=\ell-1}}\mf{S}_{\ell_1}\pa_YV_{j,\ell_2}^{\mathrm{cut}}
			&\asymp-\sum_{p=0}^{j-1}\sum_{\substack{a\geq p,k\geq0\\a+k+\sum h_r=j-1\\\ell_1+\ell_0+\sum\ell_r=\ell-1}}(-1)^{p}(a+1)C_a^{p}Y^{a-p}\mf{S}_{\ell_1}V_{a+1,k,\ell_0}^{\pm}\\
			&\times\prod_{r=1}^{p}\Theta_{h_r,\ell_r},
		\end{aligned}
	\end{equation}
	where we used $C_{a+1}^{p}(a+1-p)=(a+1)C_a^{p}$ after the shift $i\mapsto a+1$.
	
	Thus, after the cancellations among the $T\pa_T\Theta$ and $(2j_1+\frac12)\Theta$ parts, the quadratic term~\eqref{nonlin_l} and the remaining linear terms combine by Lemma~\ref{lem4.6}, and $g_0$ reduces, in the sense of $\asymp$, to
	\begin{align}
g_0&\asymp\sum_{p=0}^{j-1}\sum_{\substack{i\geq p,k\geq0\\\sum h_r=j-1-i-k\\\ell_0+\sum\ell_r=\ell}}
			(-1)^{p}C_i^{p}Y^{i-p}\prod_{r=1}^{p}\Theta_{h_r,\ell_r}
			\Bigl[\frac{1-4k-3i}{2}V_{i,k,\ell_0}^{\pm}+T\pa_TV_{i,k,\ell_0}^{\pm}\notag\displaybreak[1]\\
			&+\frac{i+1}{2}\sum_{\substack{i_1+i_2=i+1\\k_1+k_2=k\\\ell_0'+\ell_0''=\ell_0}}
			V_{i_1,k_1,\ell_0'}^{\pm}V_{i_2,k_2,\ell_0''}^{\pm}
			-(i+2)(i+1)V_{i+2,k-1,\ell_0}^{\pm}\notag\displaybreak[1]\\
			&-(i+1)\sum_{\ell_1+\ell_2=\ell_0-1}\mf{S}_{\ell_1}V_{i+1,k,\ell_2}^{\pm}\Bigr]\notag\\*
			&=0,
\label{combined_l}
\end{align}
	where the term $-(i+2)(i+1)V_{i+2,k-1,\ell_0}^{\pm}$ in the bracket is the main part of $I_5$ (cf. \eqref{P2cut_l}, with $a=i$, $b=k$), and the $\mf{S}$-term of $I_4$ in the bracket coincides with \eqref{S_cancel} (with $a=i$ and $(\ell_1,\ell_2)=(\ell_1,\ell_0)$).
	The cutoff remainders omitted in the $\asymp$ calculations, including those from $I_5$, are supported in $|Y|\leq1$ and retained in $g_0$.
	
	Consequently, by \eqref{combined_l}, all pure $V^{\mathrm{cut}}$ polynomial terms in $g_0$ cancel on $|Y|\geq1$. Setting $\ml{F}_{j,\ell}:=g_0+g_{\geq1}$ and integrating the $\ol{V}_{j,\ell}$ equation once, we obtain
	\begin{equation}\label{int_ol}
		\pa_Y\ol{V}_{j,\ell}-V_{0,0}^{\mathrm{SI}}\ol{V}_{j,\ell}
		=\int_{-\infty}^{Y}\ml{F}_{j,\ell}(T,Y')\,\md Y'
		+\bigl[T\pa_T\Theta_{j-1,\ell}-(2j-\tfrac32)\Theta_{j-1,\ell}\bigr](V_{0,0}^{\mathrm{SI}}-V_{0,0,0}^{-}).
	\end{equation}
	Set $[[V_{0,0}^{\mathrm{SI}}]]:=V_{0,0}^{\mathrm{SI}}(T,-\infty)-V_{0,0}^{\mathrm{SI}}(T,+\infty)$. For $j=1$, set $\delta_{0,\ell}=1$ for $\ell=0$ and $\delta_{0,\ell}=0$ for $\ell\in\N$, and define
	\[
	\wh{\ml F}_{1,\ell}(T):=\int_{\R}\ml{F}_{1,\ell}|_{\Theta_{0,\ell}=0}(T,Y')\md Y'.
	\]
	Then
	\[
	\Theta_{0,\ell}(T)=\frac{1}{[[V_{0,0}^{\mathrm{SI}}]](T)}\left[\delta_{0,\ell}[[V_{0,0}^{\mathrm{SI}}]](2\nu^{-\beta_2})\Theta_0^{\mathrm{S}}(2\nu^{-\beta_2})-\int_T^{2\nu^{-\beta_2}}\frac{\wh{\ml F}_{1,\ell}(\tau)}{\tau}\md\tau\right].
	\]
	The calculation for $\Theta_0^{\mathrm{S}}$ in Section~4.4, applied coefficient by coefficient, gives
\begin{align*}
\wh{\ml F}_{1,\ell}
={}&-\frac12\sum_{\ell_1+\ell_2=\ell}
(V_{0,0,\ell_1}^--V_{0,0,\ell_1}^+)
(V_{0,1,\ell_2}^++V_{0,1,\ell_2}^-)\\
&+V_{1,0,\ell}^--V_{1,0,\ell}^+\\
&-T\pa_T\sum_{\substack{\ell_1+\ell_2=\ell\\\ell_1>0}}
(V_{0,0,\ell_1}^--V_{0,0,\ell_1}^+)\Theta_{0,\ell_2}.
\end{align*}
By Proposition~\ref{prop4.2} and \eqref{4.15}, the terms without $\Theta$ are $O(T^{\ell/2})$; for $\ell=0$ they vanish by the leading-order parity.
Since $[[V_{0,0}^{\mathrm{SI}}]]=2ab$, we obtain $\Theta_{0,0}(T)=\Theta_0^{\mathrm{S}}(2\nu^{-\beta_2})$.
Integrating the total derivative exactly and using the terminal data, Proposition~\ref{prop3.6} and induction on $\ell$, we obtain, for $\ell\geq1$,
\[
|\Theta_{0,\ell}(T)|\lesssim\nu^{-\beta_2\ell/2}
+\int_T^{2\nu^{-\beta_2}}\tau^{\ell/2-1}\,\md\tau
\lesssim\nu^{-\beta_2\ell/2}.
\]
Differentiation gives $|\pa_T^m\Theta_{0,\ell}|\lesssim\nu^{-\beta_2\ell/2}T^{-m}$. Hence, assume $j\geq2$. The only term in $V_{j,\ell}^{\mathrm{cut}}$ involving $\Theta_{j-1,\ell}$ is
	\[
	-\chi(Y)V_{1,0,0}^{\pm}(T)\Theta_{j-1,\ell}(T).
	\]
	Consequently,
	\[
	\int_{\R}\ml{F}_{j,\ell}(T,Y')\md Y'
	=\int_{\R}\ml{F}_{j,\ell}|_{\Theta_{j-1,\ell}=0}(T,Y')\md Y'
	+\frac{[[V_{0,0}^{\mathrm{SI}}]]}{2}(V_{1,0,0}^++V_{1,0,0}^-)\Theta_{j-1,\ell}(T).
	\]
	Using Lemma~\ref{lem4.6} with $(i,k,\ell)=(0,0,0)$, we have
	\[
	V_{1,0,0}^++V_{1,0,0}^-=-1-2T\frac{\pa_T[[V_{0,0}^{\mathrm{SI}}]]}{[[V_{0,0}^{\mathrm{SI}}]]}.
	\]
	To cancel the boundary contribution at $Y=+\infty$, we choose $\Theta_{j-1,\ell}$ so that
	\[
	\begin{aligned}
	0&=\int_{\R}\ml{F}_{j,\ell}(T,Y')\md Y'-[[V_{0,0}^{\mathrm{SI}}]][T\pa_T\Theta_{j-1,\ell}-(2j-\tfrac32)\Theta_{j-1,\ell}]\\
	&=\int_{\R}\ml{F}_{j,\ell}|_{\Theta_{j-1,\ell}=0}(T,Y')\md Y'
	\\
	&\quad-[[V_{0,0}^{\mathrm{SI}}]]\left[T\pa_T\Theta_{j-1,\ell}-(2j-\tfrac32)\Theta_{j-1,\ell}-\frac{V_{1,0,0}^++V_{1,0,0}^-}{2}\Theta_{j-1,\ell}\right]\\
	&=\int_{\R}\ml{F}_{j,\ell}|_{\Theta_{j-1,\ell}=0}(T,Y')\md Y'
	\\
	&\quad-T\pa_T\bigl([[V_{0,0}^{\mathrm{SI}}]]\Theta_{j-1,\ell}\bigr)+(2j-2)[[V_{0,0}^{\mathrm{SI}}]]\Theta_{j-1,\ell}.
	\end{aligned}
	\]
	Setting $\wh{\ml F}_{j,\ell}(T):=\int_{\R}\ml{F}_{j,\ell}|_{\Theta_{j-1,\ell}=0}(T,Y')\md Y'$, we solve this linear first-order differential equation with $\Theta_{j-1,\ell}(2\nu^{-\beta_2})=\delta_{0,\ell}\Theta_{j-1}^{\mathrm{S}}(2\nu^{-\beta_2})$ to obtain
	\begin{equation}\label{4.41}
\begin{aligned}
\Theta_{j-1,\ell}(T)&=\frac{T^{2j-2}}{[[V_{0,0}^{\mathrm{SI}}]](T)}\Biggl[\delta_{0,\ell}(2\nu^{-\beta_2})^{2-2j}[[V_{0,0}^{\mathrm{SI}}]](2\nu^{-\beta_2})\Theta_{j-1}^{\mathrm{S}}(2\nu^{-\beta_2})\\
&\qquad-\int_{T}^{2\nu^{-\beta_2}}\frac{\wh{\ml F}_{j,\ell}(\tau)}{\tau^{2j-1}}\md \tau\Biggr].
\end{aligned}
\end{equation}
	Then the right-hand side of \eqref{int_ol} decays exponentially as $|Y|\to\infty$. Setting $\ol{V}_{j,\ell}(T,0)=0$ and choosing $\mu''\in(\mu,\mu')$, we obtain
	\[
	|\ol{V}_{j,\ell}|\lesssim \nu^{-\beta_2\ell/2}e^{-\mu''|Y|},
	\]
	which completes the induction. Derivative estimates follow in the same way.
	
\end{proof}

\section{Proof of Lemma ~\ref{lem:52} and Proposition~\ref{prop6.1}}\label{app:prop61}

This appendix proves the short-time inner--outer matching estimate in
Lemma~\ref{lem:52} and the weighted initial error estimate in
Proposition~\ref{prop6.1}.

\begin{proof}[Proof of Lemma~\ref{lem:52}]
All time derivatives at $T=0$ are right-hand derivatives.
At fixed $\x\ne0$, Proposition~\ref{prop2.12} and \eqref{2.48}
give smooth right-hand time traces for $\mf s$. Hence \eqref{2.60}
implies that $\sqrt T\mf S_{\ell_1}$ is a polynomial of degree at most
$(\ell_1+2)/2$ for even $\ell_1$, and vanishes for odd $\ell_1$.

Use the cutoff $\chi$ from \eqref{3.24}, now evaluated at $X$, and set
\begin{equation}\label{eq:Ccutoff}
\begin{aligned}
&\chi(X)=0\quad(|X|\le\tfrac12),\qquad
\chi(X)=1\quad(|X|\ge1),\\
&P_{K,\ell}:=U_\ell^{\mathrm{in}}-\chi U_{[K],\ell}^{\mathrm{OI}}.
\end{aligned}
\end{equation}
The cutoff products are extended by zero near $X=0$ and are smooth.
Take the Cauchy solutions of \eqref{4.4} with polynomially growing spatial
derivatives, as in \cite[proof of Proposition~4.2, after (4.40)]{CG}.
For $\ell=0$, the argument of \cite[(4.22)--(4.23)]{CG}, using the cubic
profile at time $T-2$, gives on $[0,1]\times\R$
\begin{equation}\label{eq:Cleading}
\begin{aligned}
|U_0^{\mathrm{in}}|\lesssim(1+X^2)^{1/6},\qquad
|\pa_X^nU_0^{\mathrm{in}}|\lesssim1\quad(n\ge1).
\end{aligned}
\end{equation}
By \cite[Proposition~4.2]{CG}, the initial matching holds for all spatial
derivatives. Recursively evaluating \eqref{4.4} and \eqref{4.5} at $T=0^+$,
and using smoothness on $|X|\le1$, yields
\begin{equation}\label{eq:Cgridinitial}
\begin{aligned}
|\pa_T^m\pa_X^nP_{K,\ell}(0,X)|
\lesssim(1+X^2)^{(\ell-3-4K-2m-3n)/6},\qquad X\in\R.
\end{aligned}
\end{equation}

Substitution into \eqref{4.4} gives, on $[0,1]\times\R$,
\begin{equation}\label{eq:Cgridrem}
\begin{aligned}
&\pa_TP_{K,0}-\pa_X^2P_{K,0}
+\pa_X(\chi U_{[K],0}^{\mathrm{OI}}P_{K,0})
+P_{K,0}\pa_XP_{K,0}=G_{K,0},\\
&\pa_TP_{K,\ell}-\pa_X^2P_{K,\ell}
+\pa_X(U_0^{\mathrm{in}}P_{K,\ell})=G_{K,\ell},\qquad\ell\ge1,
\end{aligned}
\end{equation}
where summing \eqref{4.5} over $0\le k\le K$ gives
\begin{equation}\label{eq:Ccutforcing}
\begin{aligned}
&G_{K,\ell}=\chi\pa_X^2U_{K,\ell}^{\mathrm{OI}}
-\frac\chi2\pa_X
\sum_{\substack{0\le k_1,k_2\le K,\ k_1+k_2>K\\\ell_1+\ell_2=\ell}}
U_{k_1,\ell_1}^{\mathrm{OI}}U_{k_2,\ell_2}^{\mathrm{OI}}\\
&\quad+2\chi'\pa_XU_{[K],\ell}^{\mathrm{OI}}
+\chi''U_{[K],\ell}^{\mathrm{OI}}
+\frac{\chi(1-\chi)}2\pa_X\sum_{\ell_1+\ell_2=\ell}
U_{[K],\ell_1}^{\mathrm{OI}}U_{[K],\ell_2}^{\mathrm{OI}}\\
&\quad-\chi\chi'\sum_{\ell_1+\ell_2=\ell}
U_{[K],\ell_1}^{\mathrm{OI}}U_{[K],\ell_2}^{\mathrm{OI}}
+\sqrt T\sum_{\ell_1+\ell_2=\ell-1}\mf S_{\ell_1}
\bigl(\pa_XP_{K,\ell_2}+\chi'U_{[K],\ell_2}^{\mathrm{OI}}\bigr)\\
&\quad-\pa_X\sum_{\substack{\ell_1+\ell_2=\ell\\\ell_1\ge1}}
\chi U_{[K],\ell_1}^{\mathrm{OI}}P_{K,\ell_2}
-\tfrac12\pa_X\sum_{\substack{\ell_1+\ell_2=\ell\\\ell_1,\ell_2\ge1}}
P_{K,\ell_1}P_{K,\ell_2}.
\end{aligned}
\end{equation}
The terms containing $\chi'$, $\chi''$ or $\chi(1-\chi)$ are supported in
$\tfrac12\le|X|\le1$ and have bounded derivatives.
Induct on $\ell$, assuming the full conclusion for lower inner orders. By $\mf{D}(T,X)\sim (1+|X|)^{\frac13}$ for $T\in[0,1], |X|\geq 1$, 
Proposition~\ref{prop4.2} and the induction hypothesis give
\begin{equation}\label{eq:Cgridforce}
\begin{aligned}
|\pa_T^m\pa_X^nG_{K,\ell}|
\lesssim(1+X^2)^{(\ell-5-4K-2m-3n)/6},\qquad X\in\R.
\end{aligned}
\end{equation}
For sufficiently large $C$, $Q_{n}(T,X)^{\pm}:=\pm Ce^{CT}(1+X^2)^{(\ell-3-4K-2m-3n)/6}$
are super- and subsolutions for the equations for $\pa_T^m\pa_X^nP_{K,\ell}$  and dominate
the initial data by \eqref{eq:Cgridinitial}. The comparison principle gives
\begin{equation}\label{eq:Cgridspatial}
\begin{aligned}
|\pa_T^m\pa_X^nP_{K,\ell}|
\lesssim(1+X^2)^{(\ell-3-4K-2m-3n)/6}.
\end{aligned}
\end{equation}
So this is the desired estimate.
\end{proof}

We use the shifted coordinate \eqref{inner}: the spatial argument in the
definition of $U_{K,L}^{\mathrm{app}}$ in Section~7.1 is understood as $\x$.

\begin{proof}[Proof of Proposition~\ref{prop6.1}]
\emph{1. The rescaled equation and its forcing.}
Set $\gamma_{K,L}:=\min\{(1-2\beta)K,L\beta/3\}$ and first assume
$\gamma_{K,L}\ge21$,  $\beta=\frac12-\frac23 \beta_3$. Choose the cutoff in \eqref{0.20} with transition at
$\mf d\asymp\nu^{\beta/2}$ and
$|\pa_t^m\pa_x^n\xi|\lesssim\nu^{-\beta m-3\beta n/2}$ for the derivatives
used below. All time norms in this proof are on
$T\in [\nu^{-1/2}t_{\mathrm{in}},1]$. By \eqref{inner},
$V(T,X):=U^{(\nu)}(T,X)-U_{K,L}^{\mathrm{app}}(T,X)$ satisfies
\begin{equation}\label{eq:Ceq}
\begin{aligned}
&\pa_TV
+\pa_X\bigl((U_{K,L}^{\mathrm{app}}-\dot{\mf S})V\bigr)
+\tfrac12\pa_X(V^2)-\pa_X^2V
=\widetilde E_0,\\
&V(\nu^{-1/2}t_{\mathrm{in}},X)=0,\\
&\widetilde E_0
=-\bigl(\pa_TU_{K,L}^{\mathrm{app}}
-\dot{\mf S}\pa_XU_{K,L}^{\mathrm{app}}+U_{K,L}^{\mathrm{app}}\pa_XU_{K,L}^{\mathrm{app}}
-\pa_X^2U_{K,L}^{\mathrm{app}}\bigr).
\end{aligned}
\end{equation}
Here $\mf S=0$ for $T\le0$, as stipulated for $\mf s$ in Section~4.1.
To handle the long rescaled interval, use the weight
$\om_\nu(X):=1+\nu^{3/2}X^2$; the desired weight $1+X^2$ will be recovered
at the end.

By Assumption~\ref{ass:init}(H2), finite propagation for the transport
hierarchy \eqref{3.3} gives $u_0^{\mathrm{out}}=c$ and
$u_k^{\mathrm{out}}=0$ for $k\ge1$ outside a fixed bounded physical interval.
Since $\xi$ also has bounded spatial support, enlarging this interval gives
\begin{equation}\label{eq:Csupport}
\begin{aligned}
\operatorname{supp}_x(u_{K,L}^{\mathrm{app}}(t)-c)
\cup\operatorname{supp}_xE_0(t)
\subset[-C,C],\qquad t_{\mathrm{in}}\le t\le\sqrt\nu,
\end{aligned}
\end{equation}
with $C$ independent of $\nu,K,L$. We use this support bound to obtain
the weighted estimate below from pointwise estimates.

The outer and inner residual bounds $(\nu\mf d^{-4})^{K+1}\mf d^{-1}$ and
$(\mf d\vee\nu^{1/4})^L$ are obtained in
\cite[(5.23)--(5.24), proof of Proposition~5.4]{CG}.
For $t\le0$ and our cutoff, \cite[Proposition~4.5]{CG} gives on the matching region
\begin{equation}\label{eq:Cmismatch}
\begin{aligned}
\bigl|\pa_x^j(u_{[L]}^{\mathrm{in}}-u_{[K]}^{\mathrm{out}})\bigr|
\lesssim\nu^{\gamma_{K,L}+\frac\beta2(1-3j)},
\qquad 0\le j\le4.
\end{aligned}
\end{equation}
Differentiate the residual decomposition \cite[(5.18)--(5.20)]{CG}, using
its Corollaries~3.10 and~4.4 for the profile derivatives.
The cutoff terms use \eqref{eq:Cmismatch}: a time derivative costs
$\nu^{-\beta}$ and each spatial derivative at most $\nu^{-3/4}$;
the transport and viscous cutoff terms obey the same bound since
$\beta<1/2$. This yields the following estimate for $t\le0$.
For $0\le t\le\sqrt\nu$, the same calculation uses
Propositions~\ref{prop3.2}, \ref{prop4.2}, \ref{prop4.3},
Lemma~\ref{lem:52} and the hierarchies \eqref{3.3}, \eqref{4.4},
with \eqref{2.60} controlling the truncated shock speed:
\begin{equation}\label{eq:Cpointwise}
\begin{aligned}
\|\pa_x^jE_0(t)\|_{L^\infty_x(\R)}
\lesssim\nu^{\gamma_{K,L}-(1+3j)/4},
\qquad 0\le j\le3,\quad t_{\mathrm{in}}\le t\le\sqrt\nu.
\end{aligned}
\end{equation}
Thus \eqref{eq:Csupport} gives
$\|(1+x^2)E_0(t)\|_{H^3_x(\R)}
\lesssim\nu^{\gamma_{K,L}-5/2}$.
Under \eqref{inner}, $\widetilde E_0=\nu^{1/4}E_0$ and
$\md X=\nu^{-3/4}\md x$, so
\begin{equation}\label{eq:Cforcing}
\begin{aligned}
\|\om_\nu\widetilde E_0(T)\|_{H^3_X(\R)}
&\lesssim\nu^{-1/8}
\|(1+x^2)E_0(\nu^{1/2}T)\|_{H^3_x(\R)}
\lesssim\nu^{\gamma_{K,L}-5}.
\end{aligned}
\end{equation}
The cited profile and matching estimates also give
\begin{equation}\label{eq:Ccoeff}
\begin{aligned}
\nu^{1/4}\|U_{K,L}^{\mathrm{app}}-\dot{\mf S}\|_{L^\infty_TL^\infty_X}
+\sum_{r=1}^4\|\pa_X^rU_{K,L}^{\mathrm{app}}\|_{L^\infty_TL^\infty_X}
\lesssim1.
\end{aligned}
\end{equation}

By \eqref{eq:Cmismatch} and the cutoff bounds, the matching-region term
$(\pa_x\xi)(u_{[L]}^{\mathrm{in}}-u_{[K]}^{\mathrm{out}})$ is
$\ml O(\nu^{\gamma_{K,L}-\beta})=\ml O(\nu)$ since $\gamma_{K,L}\ge21$.
Thus the argument of \cite[Proposition~5.6, (5.30)]{CG}, after the
rescaling \eqref{inner}, gives, for $T\le-1$,
\[
\|\pa_XU_{K,L}^{\mathrm{app}}(T)\|_{L^\infty_X(\R)}
\le |T|^{-1}+C_{K,L}\bigl(\nu^{1/4}|T|^{-1/2}+|T|^{-3}\bigr).
\]
Integrating the leading term gives
$\tfrac12|\log\nu|+\log|t_{\mathrm{in}}|$, while the error terms have
uniformly bounded time integrals. Together with \eqref{eq:Ccoeff} on
$[-1,1]$, this yields
\begin{equation}\label{eq:Cslope}
\begin{aligned}
\int_{\nu^{-1/2}t_{\mathrm{in}}}^1
\|\pa_XU_{K,L}^{\mathrm{app}}(T)\|_{L^\infty_X(\R)}\,\md T
\le\tfrac12|\log\nu|+C_{K,L}.
\end{aligned}
\end{equation}
Implicit constants may depend on $K,L,\beta$, but the losses in powers of
$\nu$ below do not depend on $K,L$.

\emph{2. The weighted energy inequality.}
For $0\le j\le3$, define
\begin{equation}\label{eq:Cenergydef}
\begin{aligned}
\ml E_j(T):=\sum_{r=0}^j\|\om_\nu\pa_X^rV(T)\|_{L^2_X(\R)}^2,
\qquad \ml E_{-1}:=0.
\end{aligned}
\end{equation}
Applying $\pa_X^r$ to \eqref{eq:Ceq}, multiplying
by $\om_\nu^2\pa_X^rV$ and integrating by parts gives
\begin{align}
&\tfrac12\frac{\md}{\md T}\|\om_\nu\pa_X^rV\|_{L^2_X(\R)}^2
+\|\om_\nu\pa_X^{r+1}V\|_{L^2_X(\R)}^2\notag\\
&=-\tfrac{2r+1}{2}\int_{\R}\om_\nu^2(\pa_XU_{K,L}^{\mathrm{app}})(\pa_X^rV)^2\,\md X
+\int_{\R}\om_\nu\om_\nu'(U_{K,L}^{\mathrm{app}}-\dot{\mf S})(\pa_X^rV)^2\,\md X\notag\\
&\quad+\tfrac12\int_{\R}(\om_\nu^2)''(\pa_X^rV)^2\,\md X
-\sum_{k=2}^{r+1}\binom{r+1}{k}\int_{\R}\om_\nu^2\pa_X^rV(\pa_X^kU_{K,L}^{\mathrm{app}})\pa_X^{r+1-k}V\,\md X\notag\\
&\quad+\tfrac12\int_{\R}\pa_X(\om_\nu^2\pa_X^rV)\pa_X^r(V^2)\,\md X
+\int_{\R}\om_\nu^2\pa_X^rV\pa_X^r\widetilde E_0\,\md X.
\label{eq:Cibp}
\end{align}
Use \eqref{eq:Ccoeff}, the bounds
$|\om_\nu'|\lesssim\nu^{3/4}\om_\nu$,
$|(\om_\nu^2)''|\lesssim\nu^{3/2}\om_\nu^2$, and the elementary product estimate
$\|\om_\nu\pa_X^r(V^2)\|_{L^2_X(\R)}
\lesssim\|V\|_{L^\infty_X(\R)}\ml E_r^{1/2}$, $r\le3$.
The latter follows by Sobolev interpolation on unit intervals, where
$\om_\nu$ is comparable to a constant, and summation.
Absorb the nonlinear flux into the diffusion and apply Young's inequality
with parameter $\nu^{1/2}$ to the commutator and forcing terms. Summing gives
\begin{equation}\label{eq:Cenergy}
\begin{aligned}
&\ml E_j'+\tfrac12\sum_{r=0}^j
\|\om_\nu\pa_X^{r+1}V\|_{L^2_X(\R)}^2\\
&\le\Bigl((2j+1)\|\pa_XU_{K,L}^{\mathrm{app}}\|_{L^\infty_X(\R)}
+C_{K,L}\nu^{1/2}
+C\|V\|_{L^\infty_X(\R)}^2\Bigr)\ml E_j\\
&\quad+C_{K,L}\nu^{-1/2}\ml E_{j-1}
+C\nu^{-1/2}\|\om_\nu\widetilde E_0\|_{H^j_X(\R)}^2.
\end{aligned}
\end{equation}

\emph{3. Bootstrap.}
Assume $\|V(T)\|_{L^\infty_X(\R)}\le\nu^{1/4}$ on the bootstrap interval.
Its length is $O(\nu^{-1/2})$, so \eqref{eq:Cslope} gives, for
$\nu^{-1/2}t_{\mathrm{in}}\le\tau\le T\le1$,
\begin{equation}\label{eq:Cgronwall}
\begin{aligned}
\exp\!\int_\tau^T\Bigl((2j+1)
\|\pa_XU_{K,L}^{\mathrm{app}}\|_{L^\infty_X(\R)}
+C_{K,L}\nu^{1/2}+C\|V\|_{L^\infty_X(\R)}^2\Bigr)\,\md T'
\lesssim\nu^{-(2j+1)/2}.
\end{aligned}
\end{equation}
Since $\ml E_j(\nu^{-1/2}t_{\mathrm{in}})=0$, \eqref{eq:Cforcing},
\eqref{eq:Cenergy} and induction in $j$ imply
\begin{equation}\label{eq:Cweighted}
\begin{aligned}
\sup_T\ml E_j(T)
&\lesssim\nu^{2\gamma_{K,L}-11-j-(j+1)^2/2},
\qquad 0\le j\le3,\\
\sup_T\|V(T)\|_{L^\infty_X(\R)}
&\lesssim\sup_T\ml E_3(T)^{1/2}
\lesssim\nu^{\gamma_{K,L}-11}
\le\tfrac12\nu^{1/4}
\end{aligned}
\end{equation}
for sufficiently small $\nu$. Continuity closes the bootstrap up to $T=1$.
The weighted derivative norms are equivalent to the corresponding norms
of the weighted function, uniformly in $\nu$. Hence
\begin{equation}\label{eq:Cscaled}
\begin{aligned}
\|(1+X^2)(U^{(\nu)}-U_{K,L}^{\mathrm{app}})\|_{L^\infty_T H^3_X([0,1]\times\R)}
&\lesssim\nu^{-3/2}\sup_T\ml E_3(T)^{1/2}\\
&\lesssim\nu^{\gamma_{K,L}-25/2}
\lesssim\nu^{\gamma_{K,L}-20}.
\end{aligned}
\end{equation}

For arbitrary $K,L$, choose $K'\ge K$, $L'\ge L$ with
$\gamma_{K',L'}\ge21$, using the same cutoff. The cited profile estimates
applied to the omitted finite sums give
\begin{equation}\label{eq:Ctrunc}
\begin{aligned}
\|(1+X^2)(U_{K',L'}^{\mathrm{app}}-U_{K,L}^{\mathrm{app}})
\|_{L^\infty_T H^3_X([0,1]\times\R)}
\lesssim\nu^{\gamma_{K,L}-7}.
\end{aligned}
\end{equation}
The triangle inequality proves \eqref{eq:Cscaled} for all $K,L$.
For each fixed $K,L$, the remaining viscosities in $(0,\nu_0]$ are bounded
away from zero. For these viscosities, the estimate follows from \eqref{eq:Cenergy} with bounded coefficients.
Thus $C_1=20$ suffices in the range covered by the construction.
\end{proof}

\bigskip

\section*{Acknowledgments}
	C. Wang is partially supported by NSF of China under Grant No. 12471189. Z. Zhang is partially supported by NSF of China under Grant No. 12288101.

	\section*{Declaration of competing interest}
	
	The authors declare that they have no conflict of interest.

    \section*{AI Statements}

The authors declare that the main ideas of this paper are their own, although AI tools were used to polish the paper and to assist with parts of the technical derivations and with checking the paper. In particular, DeepSeek V4-Pro Preview and DeepSeek V4-Flash 0731 were used to assist with the derivation of the equations satisfied by the auxiliary variables in the paper and with the computation of identities similar to those in Appendices A and B; these models were also used to polish the text and to assist in checking the paper. The figure in this paper was drawn with DeepSeek V4-Flash 0731. The authors also used ChatGPT 5.6 Sol and ChatGPT 6 Astra to polish the text, to assist in checking the paper.

\end{document}